\documentclass[12pt]{amsart}
\usepackage[margin=1in,letterpaper,portrait]{geometry}
\usepackage{ytableau}
\ytableausetup{notabloids}
\ytableausetup{centertableaux}
\usepackage{latexsym,amsmath,amssymb,verbatim, tikz}
\usepackage{graphicx}
\usepackage{amsmath, amssymb, tikz}
\usepackage[colorlinks,pdfencoding = auto]{hyperref}
\usepackage{mathrsfs}
\usepackage{mathrsfs}
\usepackage{amsfonts}
\usepackage[sorted,compressed-cites,sorted-cites,initials]{amsrefs}
\usepackage{amsthm,youngtab}
\usepackage{graphicx}
\usepackage{caption}
\usepackage{enumerate,enumitem}
\usepackage{tikz-cd}
\usepackage{rotating}
\usepackage[all]{xy}
\usepackage[titletoc, title]{appendix}
\usepackage{amscd}
\usepackage{float}
\usepackage{dsfont}
\usepackage{tkz-graph}
\tikzstyle{vertex}=[circle, draw, inner sep=0pt, minimum size=6pt]

\theoremstyle{plain}
   
   \newtheorem{theorem}{Theorem}[section]
   \newtheorem{proposition}[theorem]{Proposition}
   
   \newtheorem{lemma}[theorem]{Lemma}
   \newtheorem{corollary}[theorem]{Corollary}
   \newtheorem{conjecture}[theorem]{Conjecture}
   
\theoremstyle{definition}
   \newtheorem{definition}[theorem]{Definition}
   \newtheorem{defn}[theorem]{Definition}
   \newtheorem{example}[theorem]{Example}

   \newtheorem{remark}[theorem]{Remark}

\numberwithin{equation}{section}

\newcommand{\CC}{{\mathbb {C}}}
\newcommand{\QQ}{{\mathbb {Q}}}

\newcommand{\ZZ}{{\mathbb {Z}}}

\newcommand{\bbc}{{\mathbb {C}}}

\newcommand{\SSYT}{{\rm SSYT}}
\newcommand{\GL}{{\rm GL}}

\newcommand{\ch}{{\operatorname{ch}}}

\DeclareMathOperator*{\diag}{diag}

\DeclareMathOperator*{\wt}{wt}

\DeclareMathOperator{\Gr}{Gr}

\newlength{\mysizetiny}
\newlength{\mysizesmall}
\newlength{\mysize}
\newlength{\mysizelarge}
\begin{document}

\title{Quantum affine algebras and Grassmannians}
\author{Wen Chang, Bing Duan, Chris Fraser, and Jian-Rong Li}
\address{Wen Chang, School of Mathematics and Information Science, Shaanxi Normal University, Xi'an, China,
 changwen161@163.com}
\address{Bing Duan, School of Mathematics and Statistics, Lanzhou University, Lanzhou, China, duan890818@163.com}
\address{Chris Fraser, School of Mathematics, University of Minnesota, Minneapolis, USA, \hspace{.3cm} \textcolor{white}{.} cfraser@umn.edu}
\address{Jian-Rong Li, Department of Mathematics and scientific computing, University of Graz, Graz, Austria; Department of Mathematics, The Weizmann Institute of Science, Rehovot 7610001, Israel, lijr07@gmail.com}
\date{}

\begin{abstract}
We study the relation between quantum affine algebras of type $A$ and Grassmannian cluster algebras. Hernandez and Leclerc described an isomorphism from the Grothendieck ring of a certain subcategory $\mathcal{C}_{\ell}$ of $U_q(\widehat{\mathfrak{sl}_n})$-modules to a quotient of the Grassmannian cluster algebra in which certain frozen variables are set to 1. We explain how this induces an isomorphism between the monoid of dominant monomials, used to parameterize simple modules, and a quotient of the monoid of rectangular semistandard Young tableaux with $n$ rows and with entries in $[n+\ell+1]$. Via the isomorphism, we define an element ch$(T)$ in a Grassmannian cluster algebra for every rectangular tableau $T$. By results of Kashiwara, Kim, Oh, and Park, and also of Qin, every Grassmannian cluster monomial is of the form ch$(T)$ for some $T$. Using a formula of Arakawa-Suzuki, we give an explicit expression for ch$(T)$, and also give 
explicit $q$-character formulas for finite-dimensional $U_q(\widehat{\mathfrak{sl}_n})$-modules.  We give a tableau-theoretic rule for performing mutations in Grassmannian cluster algebras. We suggest how our formulas might be used to study reality and primeness of modules, and compatibility of cluster variables.
\end{abstract}

\maketitle

\tableofcontents

\section{Introduction}

Let $\mathfrak{g}$ be a simple Lie algebra and let $U_q(\widehat{\mathfrak{g}})$ be the corresponding quantum affine algebra. Chari and Pressley \cite{CP94} have classified the simple objects in the category of all finite-dimensional $U_q(\widehat{\mathfrak{g}})$-modules and Nakajima \cite{Nak01} has computed the characters of the simple objects in this category in terms of the cohomology of certain quiver varieties.

Fomin and Zelevinsky \cite{FZ02} introduced the theory of cluster algebras to study
canonical bases of quantum groups introduced by Lusztig \cite{L90} and Kashiwara \cite{Kas} and total
positivity for semisimple algebraic groups developed by Lusztig \cite{L94}. 

Hernandez and Leclerc \cites{HL10,HL16} applied the theory of cluster algebras to study quantum affine algebras. They introduced the notion of a {\sl monoidal categorification} of a cluster algebra. In a monoidal category $(\mathcal{C}, \otimes)$, a simple object $S$ of $\mathcal{C}$ is called {\sl real} if its tensor square $S \otimes S$ is also simple, and is called {\sl prime} if it admits no nontrivial tensor factorization $S \cong S_1 \otimes S_2$. Hernandez and Leclerc called $\mathcal{C}$ a monoidal categorification of a cluster algebra $A$ if the Grothendieck ring of $\mathcal{C}$ is isomorphic to $A$, any cluster monomial of $A$ corresponds to the class of a real simple object of $\mathcal{C}$, and any cluster variable (including the frozen ones) of $A$ corresponds to the class of a real simple prime object of $\mathcal{C}$. 

Kang, Kashiwara, Kim, and Oh \cite{KKKO} proved that the quantum unipotent coordinate algebra has a monoidal categorification as conjectured in \cites{GLS,Ki}. The connection between monoidal categorification and quantum affine algebras is as follows. Let $\mathcal{C}^{\mathfrak{g}}$ be the category of finite-dimensional $U_q(\widehat{\mathfrak{g}})$-modules.  For each $\ell \in \mathbb{Z}_{\ge 0}$, Hernandez and Leclerc \cite{HL10} introduced a full monoidal subcategory $\mathcal{C}_{\ell}^{\mathfrak{g}}$ of $\mathcal{C}^{\mathfrak{g}}$ whose objects are characterized by certain restrictions on the roots of the Drinfeld polynomials of their composition factors. They constructed a cluster algebra $\mathcal{A}_{\ell}^{\mathfrak{g}}$ and conjectured that $\mathcal{C}_{\ell}^{\mathfrak{g}}$ is a monoidal categorification of $\mathcal{A}_{\ell}$. 

Denote by $K_0(\mathcal{C}_{\ell}^{\mathfrak{g}})$ the Grothendieck ring of $\mathcal{C}_{\ell}^{\mathfrak{g}}$, as an algebra over the complex numbers.  
We use brackets $[S]$ to denote the Grothendieck class of an object $S \in \mathcal{C}_{\ell}$. Qin \cite{Qin} proved that for $\mathfrak{g}$ of type $A, D, E$, every cluster monomial (resp. cluster variable) in $K_0(\mathcal{C}_{\ell}^{\mathfrak{g}})$ is a simple (resp. prime simple) module. Recently, Kashiwara, Kim, Oh, and Park proved that when $\mathfrak{g}$ is of type $A$ or $B$, every cluster monomial in $K_0(\mathcal{C}_{\ell}^{\mathfrak{g}})$ can be identified with a real module \cite{KKOP}. 

In this paper, we focus on $\mathfrak{g}=\mathfrak{sl}_n$ and study finite-dimensional representations of $U_q(\widehat{\mathfrak{g}})$. We abbreviate $\mathcal{C}_{\ell}^{\mathfrak{g}} = \mathcal{C}_{\ell},  K_0(\mathcal{C}_{\ell}^{\mathfrak{g}}) = K_0(\mathcal{C}_{\ell})$, etc. Hernandez and Leclerc descibed an isomorphism~$\Phi$ from $K_0(\mathcal{C}_{\ell})$ to a certain quotient of the cluster algebra $\bbc[\Gr(n,m)]$ for the Grassmannian \cite[Section 13]{HL10}, where $m = n+\ell+1$. 
This quotient, which we denote by $\bbc[\Gr(n,m,\sim)]$, is the one in which solid frozen Pl\"ucker coordinates are specialized to~1 (the frozen Pl\"ucker coordinates whose columns wrap around modulo $m$ are not specialized).  

Simple objects in $\mathcal{C}_{\ell}^{\mathfrak{g}}$ are parametrized by elements of a free abelian monoid $\mathcal{P}_{\ell}^+=\mathcal{P}^+_{\ell, \mathfrak{g}}$ (cf. \cites{CP94,HL10}) thought of as monomials in variables $Y_{i,i-2k-2}$, $i \in I$, $k \in [0, \ell]$, where $I$ is the vertex set of the Dynkin diagram of $\mathfrak{g}$. Denote by $L(M)$ the simple module corresponding to $M \in \mathcal{P}^+_\ell$. 

Our first theorem interprets this parameterization of simple modules in terms of Grassmannians. Denote by ${\rm SSYT}(n, [m])$ the set of semistandard Young tableaux of rectangular shape, with $n$ rows and with entries in $[m] = \{1,\dots,m\}$. The set ${\rm SSYT}(n, [m])$ carries two important structures. First, a weight map ${\rm SSYT}(n, [m])$ to the weight lattice $P_{\mathfrak{g}}$ for $\mathfrak{g}$, which provides a notion of when one tableau has higher weight than another. Second, the structure of a commutative monoid, with multiplication ``$\cup$'' defined as follows: for $S,T \in {\rm SSYT}(n, [m])$, $S \cup T$ is the semistandard tableau whose $i$th row is the (multiset) union of the $i$th rows of $S$ and $T$, for $i=1,\dots,n$. 

For comparison with $U_q(\widehat{\mathfrak{g}})$-modules, we define a quotient monoid
${\rm SSYT}(n, [m],\sim)$ in which certain tableaux equal~1, mirroring the frozen Pl\"ucker coordinates which are trivialized in $\bbc[\Gr(n,m,\sim)]$. 
We use the notation $S \sim T$ to say that two tableaux are equal in this quotient. The weight map descends to 
${\rm SSYT}(n, [m],\sim)$. 

For $T \in {\rm SSYT}(n, [m])$ with columns $T_1,\dots,T_a$, let $P_T \in \bbc[\Gr(n,m)]$ denote the monomial in Pl\"{u}cker coordinates $P_{T_1} \cdots P_{T_a}$. The monomials $\{P_T\}$, where $T \in {\rm SSYT}(n,[m])$, are a basis for $\bbc[\Gr(n,m)]$ known as the {\sl standard monomial} basis \cite{Ses}. Thus for any simple module $L(M)$, $\Phi([L(M)])$ can be written as a linear combination of standard monomials. We show that one can define a ``top tableau''  
${\rm Top}(\Phi([L(M)])) \in {\rm SSYT}(n,[m],\sim)$ appearing with highest weight in such an expression. We denote the map $M  \mapsto {\rm Top}(\Phi([L(M)]))$ by $\widetilde{\Phi}$. 
\begin{theorem}[{Theorem \ref{thm: parametrization of simple modules by tableaux}}]\label{thm:first}
The map $\widetilde{\Phi} \colon \mathcal{P}^+_{\ell,A_{n-1}} \to {\rm SSYT}(n, [n+\ell+1], \sim) $ is an isomorphism of monoids. 
\end{theorem}
Therefore the finite-dimensional simple $U_q(\widehat{\mathfrak{g}})$-modules in $\mathcal{C}_{\ell}^{A_{n-1}}$ are also parametrized by ($\sim$-classes of) semistandard tableaux. We call a tableau $T$ {\sl real} (resp. {\sl prime}) if this is true of the corresponding module $L(\widetilde{\Phi}(T))$. We also show 
(cf.~Proposition~\ref{prop:posetsaresame}) that the map ${\rm SSYT}(n, [n+\ell+1]) \to \mathcal{P}^+_{\ell,A_{n-1}}$ respects familiar partial orders on both sides. 

One sense in which Theorem~\ref{thm:first} is interesting is that the monoid ${\rm SSYT}(n, [m])$ is not free, while the theorem asserts that ${\rm SSYT}(n, [m],\sim)$ {\sl is} free (on explicit generators). 

Table \ref{table:correspondence Gr36 and Uqsl3hat} illustrates the resulting correspondence between tableaux and modules. We describe tableaux by their column sets
and denote e.g. the module $L(Y_{1,-3} Y_{1,-1})$ by $1_{-3} 1_{-1}$. 

\begin{table}[ht]
\begin{equation*}
\hspace{2cm}
\begin{minipage}{0.5\textwidth}
\begin{tabular}{|c|c|}
\hline
tableau & module \\
\hline
 $[1,2,4]$ & $1_{-1}$   \\
\hline 
 $[1,2,5]$ & $1_{-3}1_{-1}$  \\
\hline
 $[1,2,6]$ & $1_{-5}1_{-3}1_{-1}$  \\
\hline
 $[1,3,4]$ & $2_{0}$  \\
\hline
 $[1,3,5]$ & $1_{-3}2_0$  \\
\hline
 $[1,3,6]$ & $1_{-5}1_{-3} 2_{0}$  \\
\hline
 $[1,4,5]$ & $2_{-2}2_{0}$  \\
\hline
 $[1,4,6]$ & $1_{-5} 2_{-2}2_{0}$  \\
\hline
 $[1,5,6]$ & $2_{-4}2_{-2}2_{0}$  \\
\hline
\end{tabular}
\end{minipage}
\hspace{-3cm}
\begin{minipage}{0.5\textwidth}
\begin{tabular}{|c|c|}
\hline
tableau & module  \\
\hline
 $[2,3,5]$ & $1_{-3}$  \\
\hline
 $[2,3,6]$ & $1_{-5} 1_{-3}$  \\
\hline
 $[2,4,5]$ & $2_{-2}$  \\
\hline
 $[2,4,6]$ & $1_{-5} 2_{-2}$  \\
\hline
 $[2,5,6]$ & $2_{-4}2_{-2}$  \\
\hline
 $[3,4,6]$ & $1_{-5}$  \\
\hline
 $[3,5,6]$ & $2_{-4}$  \\
\hline
$[1,2,4], [3,5,6]$ & $2_{-4}1_{-1}$  \\
\hline
$[1,3,5],[2,4,6]$ & $1_{-5}1_{-3} 2_{-2}2_{0}$  \\
\hline
\end{tabular}
\end{minipage}
\end{equation*}
\caption{Labeling prime simple modules in $\mathcal{C}_{2}^{A_2}$ by tableaux in ${\rm SSYT}(3,[6], \sim)$.}
\label{table:correspondence Gr36 and Uqsl3hat}
\end{table}

Via the correspondence we define elements $\ch(T) = \Phi([L(\widetilde{\Phi}^{-1}(T))]) \in \bbc[\Gr(n,m,\sim)]$, forming a basis for $\bbc[\Gr(n,m,\sim)]$. Making use of a well-known grading on $\bbc[\Gr(n,m)]$, for each tableau $T$ we define 
a homogeneous lift of $\ch(T)$ from $\bbc[\Gr(n,m,\sim)]$ to a localization of $\bbc[\Gr(n,m)]$ (a priori, the lifts might have frozen variables in the denominator, so they naturally live in a localization). 
By deep results of Kashiwara, Kim, Oh, and Park \cite{KKOP} and Qin \cite{Qin}, we have the following. 
\begin{theorem} [{Theorem \ref{thm: cluster monomials are tableaux}}]
Every cluster monomial (resp. cluster variable) in $\bbc[\Gr(n,m)]$ is of the form $\ch(T)$ for some real tableau (resp. prime real tableau) $T \in {\rm SSYT}(n, [m])$.
\end{theorem}
We expect that the lift $\ch(T)$ always lies in $\bbc[\Gr(n,m)]$ (not in the localization), so that $\{\ch(T)\}_{T \in {\rm SSYT}(n,[m])}$ is a homogeneous basis for $\bbc[\Gr(n,m)]$ containing the cluster monomials. 
As the simplest example, if $T$ has a single column, then $\ch(T) = P_T$ is the Pl\"ucker coordinate given by the entries of $T$.

We translate a formula of Arakawa-Suzuki \cite{AS} (see also \cites{BaCi, Hen, LM}) in the setting of $p$-adic groups to our setting of quantum affine algebras and Grassmannians. We obtain an explicit formula for the $q$-character of a finite-dimensional simple module and also for $\ch(T)$:
\begin{theorem}[Theorems~\ref{cor:qcharacter formula} and \ref{cor:qcharacter formula tableaux}]
For a simple $U_q(\widehat{\mathfrak{sl}_n})$-module $L(M)$, the $q$-character of $L(M)$ is given by
\begin{align*}
\chi_q(L(M)) = \sum_{u \in S_k} (-1)^{\ell(uw_M)} p_{uw_0, w_Mw_0}(1) \prod_{M' \in {\rm Fund}_M(u\mu_M, \lambda_M)} \chi_q(L(M')),
\end{align*}
where $k$ is the degree of the monomial $M$, $w_0 \in S_k$ is the longest permutation, $w_M \in S_k$ is determined by $M$, ${\rm Fund}_M(u\mu_M, \lambda_M)$ is a subset of the variables $Y_{i,s}$, and $p_{y,y'}(t)$ is a Kazhdan-Lusztig polynomial \cite{KL}. 

For $T\in \text{SSYT}(n, [m])$ we have 
\begin{align*}
\ch(T) = \sum_{u \in S_k} (-1)^{\ell(uw_T)} p_{uw_0, w_Tw_0}(1) P_{u \cdot T'} \in \bbc[\Gr(n,m,\sim)],
\end{align*}
where $k$ is a certain {\sl gap weight} of $T$, $w_T \in S_k$ is determined by $T$, 
$T' \sim T$ is a tableau equivalent to $T$ that has {\sl small gaps} (cf.~Definition~\ref{defn:smallgaps}), and 
$P_{u;T'}$ is a standard monomial encoded by appropriately permuting entries of $T'$. 
\end{theorem}

A similar $q$-character formula in more geometric language is due to Ginzburg and Vasserot \cites{Vas,GV}. Our formula for $\ch(T)$ can be formulated as a {\sl Kazhdan-Lusztig immanant} \cite{RS}, which implies that $\ch(T) \geq 0$ on the totally nonnegative Grassmannian $\Gr(n,m)_{\geq 0}$. 

In particular, we obtain an explicit formula for Grassmannian cluster monomials, with the caveat that it is a difficult problem to determine when a given $T$ corresponds to a cluster monomial. Our formula is very closely related to the expression of $\ch(T)$ in the standard monomial basis. For example, if $T= T'$ has small gaps, then our formula {\sl is} the standard monomial expression. And for every tableau $T$, there is a small gaps tableau $T'$ such that $\ch(T)$ and $\ch(T')$ are related by a Laurent monomial in frozens. 

Our results are a step in developing the cluster combinatorics of $\CC[\Gr(n,m)]$, which is poorly understood when $n \geq 3$. Fomin and Pylyavskyy \cite{FP} suggested an approach to this combinatorics when $n=3$. They conjectured that each cluster monomial in $\CC[\Gr(3,m)]$ is a {\sl web invariant}, an element in a basis for $\CC[\Gr(3,m)]$ indexed by planar diagrams \cite{Kup}. Khovanov and Kuperberg gave a bijection between these diagrams and ${\rm SSYT}(3, [m])$ \cite{KK}. We conjecture that for cluster monomials, $\ch(T) \in \bbc[\Gr(3,m)]$ is the web invariant labeled by $T$.
We show that $\ch(T)$ is not {\sl always} a web invariant (cf.~Example~\ref{example:non-real tableaux}). We conjecture also that two cluster variables $\ch(T),\ch(T')$ lie in a common cluster only if $\ch(T)\ch(T')=\ch(T \cup T')$, which is in the spirit of \cite[Conjecture 9.2]{FP}.

Theorem \ref{thm: cluster monomials are tableaux} suggests how many aspects of Grassmannian clusters are controlled by the monoid ${\rm SSYT}(n, [m])$. For example, we highlight how $g$-vectors of cluster monomials can be computed in the monoid (Corollary \ref{corollary:correspondence between g vectors and real tableaux}), and how mutations of cluster variables and simple modules can be simulated in the monoid (cf.~Section \ref{sec:mutations description}). Similar ideas appeared implicitly in the work of Shen and Weng \cite{SW}, who showed that the {\sl theta basis} \cite{GHKK} for $\bbc[\Gr(n,m)]$ is parameterized by {\sl Gelfand-Tsetlin patterns}, a family of objects which are in well known bijection with semistandard Young tableaux.

An action of the extended affine braid group on Grassmannian cluster algebras was introduced in~\cite{Fra}. It yields an action on cluster monomials, hence on a subset of tableaux. We hope to give a concrete description of this action on {\sl all} tableaux, thus on the set of simple $U_q(\widehat{\mathfrak{g}})$-modules, in a future paper.  

We illustrate our results in a few examples, explained more fully in the main text. 
\begin{example}
Consider the simple $U_q(\widehat{\mathfrak{sl}_3})$-module $L(M)$, where $M=Y_{1,-5} Y_{1,-3} Y_{2,-2} Y_{2,0}$. Our $q$-character formula says that
\begin{align*}
\begin{split}
\chi_q(L(M)) & = - 1 + \chi_q(Y_{1,-1}) \chi_q(Y_{2,-4}) - \chi_q(Y_{2,0}) \chi_q(Y_{2,-2}) \chi_q(Y_{2,-4}) \\
& \quad - \chi_q(Y_{1,-1}) \chi_q(Y_{1,-3}) \chi_q(Y_{1,-5}) + \chi_q(Y_{2,0}) \chi_q(Y_{1,-3}) \chi_q(Y_{2,-2}) \chi_q(Y_{1,-5}).
\end{split}
\end{align*}

The tableau $T = \widetilde{\Phi}(M)$ associated to this module has  
\begin{align}\label{eq:inhomogeneous}
\ch(T) = \ch(\begin{ytableau}
1 & 2  \\
3 & 4 \\
5 & 6 
\end{ytableau}) & = - 1 + P_{124}P_{356} - P_{134}P_{245}P_{356} - P_{124}P_{235}P_{346} + P_{134}P_{235}P_{245}P_{346} \\
& = P_{1 3 5} P_{2 4 6} - P_{1 2 5} P_{3 4 6}- P_{1 3 4} P_{2 5 6}+ P_{1 2 4} P_{3 5 6}- 2  P_{1 2 3} P_{4 5 6} \label{eq:stdmonomialexpr}.
\end{align}
Here, $P_{j_1j_2j_3} \in \bbc[\Gr(3,6)]$ denotes a Pl\"ucker coordinate. The first line is an inhomogeneous expression for $\ch(T) \in \bbc[\Gr(3,6,\sim)]$ (in which $P_{123}=P_{234}=P_{345}=P_{456}=1$) obtained by translating the $q$-character formula term by term. The second line is its homogeneous lift $\ch(T) \in \bbc[\Gr(3,6)]$ expressed in the basis of standard monomials. This element of 
$\bbc[\Gr(3,6)]$ is a web invariant with diagram
\begin{tikzpicture}[scale=0.3]
\draw[thick] (0,0) circle (2 cm);
  \newdimen\R
   \R=2cm
   \foreach \x/\l/\p/\q in
     {120- 60/{1}/above/v1,
      120-120/{2}/right/v2,
      120-180/{3}/below/v3,
      120-240/{4}/below/v4,
      120-300/{5}/left/v5,
      120-360/{6}/above/v6
     }
     \node[inner sep=1pt,circle,draw,fill,label={\p:\l}] at (\x:\R) (\q) {};
     \node[inner sep=1pt,circle,draw,fill,label={above:6}] at (120-360:\R) (v6) {};
     \node[inner sep=1pt,circle,draw] at (-0.5,0.5) (p1) {};   
     \node[inner sep=1pt,circle,draw] at (0.5,0) (p2) {};
     \node[inner sep=1pt,circle,draw] at (0,1.5) (p3) {};
     \node[inner sep=1pt,circle,draw,fill] at (0,1) (p4) {};
     \draw (v1)--(p3);
     \draw (v6)--(p3);
     \draw (p4)--(p3);
     \draw (p4)--(p1);
     \draw (p4)--(p2);
     \draw (v5)--(p1); 
     \draw (v4)--(p1); 
     \draw (v2)--(p2); 
     \draw (v3)--(p2);
\end{tikzpicture}, and this web is labeled by $T$ in the Khovanov-Kuperberg bijection (see e.g~\cite[Section 3.2]{Tym}). This is one of the two non-Pl\"ucker cluster variables in $\CC[\Gr(3,6)]$ (it is $Y^{123456}$ in the notation of \cite{Sco}). The other non-Pl\"ucker cluster variable corresponds to the tableau with columns $[1,2,4],[3,5,6]$. See Examples~\ref{example:character of the tableau 135246}, \ref{example:chT is equal to web invariant} for more details. 
\end{example}

\begin{example}\label{eg:exchangeintro}
The following is an exchange relation in $\bbc[\Gr(3,8)]$: 
\begin{align}\label{eq:mutnexample}
& \ch(\begin{ytableau}
2  \\
3 \\
8 
\end{ytableau}) \ch(\begin{ytableau}
1 & 3 & 4  \\
2 & 5 & 6 \\
4 & 7 & 8 
\end{ytableau}) = \ch(\begin{ytableau}
1  \\
2 \\
8 
\end{ytableau}) \ch(\begin{ytableau}
3 & 4  \\
5 & 6 \\
7 & 8
\end{ytableau}) \ch(\begin{ytableau}
2  \\
3 \\
4 
\end{ytableau}) + \ch(\begin{ytableau}
3  \\
4 \\
8 
\end{ytableau}) \ch(\begin{ytableau}
2 & 4  \\
5 & 6 \\
7 & 8 
\end{ytableau}) \ch(\begin{ytableau}
1  \\
2 \\
3 
\end{ytableau}).
\end{align}
If one is trying to mutate the cluster variable $\ch(\begin{ytableau}
2  \\
3 \\
8 
\end{ytableau})$ using this exchange relation, then the ``new'' cluster variable 
$\ch(\begin{ytableau}
1 & 3 & 4  \\
2 & 5 & 6 \\
4 & 7 & 8 
\end{ytableau})$ can be computed from the other tableaux in the exchange relation using the weight map and the monoid 
{\rm SSYT}$(3,[8])$ (cf.~Section~\ref{sec:mutations description}). 
\end{example}

Our explicit formulas are an approach for determining reality and primeness of modules. A module $L(M)$ is real if and only if $\chi_q(L(M))^2  = \chi_q(L(M^2))$ and it is prime if and only if there are no $L(M'), L(M'') \ne \CC$ such that $\chi_q(L(M)) = \chi_q(L(M')) \chi_q(L(M''))$ (cf. Lemma~\ref{lem:qcharacter condition of real modules and prime modules}). 
There are analogous statements for tableaux using $\ch(T)$. 
For example, one can check that $\ch(T)$ is not a cluster monomial by checking that $\ch(T)^2 \neq \ch(T \cup T)$. If the Fomin-Pylyavskyy conjectures are proved, they would imply diagrammatic recipes for determining reality and primeness of $U_q(\widehat{\mathfrak{sl}_3})$-modules.

Recently, Brito and Chari \cite{BrCh} studied a category related to $\mathcal{C}_{\ell}$. They derived character formulas for the prime objects in the category and also the tensor product rules for these objects. 

We conclude by noting that it remains an open problem to find satisfactory descriptions of the following: first, the sets of real and prime tableaux; second, the condition for when two tableaux index compatible cluster variables; and third, a list of {\sl all} exchange relations, in the spirit of Example~\ref{eg:exchangeintro}. 

The paper is organized as follows. Section \ref{sec:preliminary} introduces Hernandez and Leclerc's approach to monoidal categorification of quantum affine algebras, and the relation with the cluster structure on $\Gr(n,m)$. Section \ref{sec:Uqg hat modules and tableaux} gives the relation between $U_q(\widehat{\mathfrak{g}})$-modules and semistandard Young tableaux, and establishes that every cluster monomial in a Grassmannian cluster algebra is of the form $\ch(T)$. Section \ref{sec:mutations description} describes mutations of cluster variables and modules in terms of semistandard Young tableaux. Section \ref{sec:qcharacter formulas} gives explicit formulas for $q$-characters and $\ch(T)$, and the relation with Kazhdan-Lusztig immanants. Section \ref{sec:Fomin Pylyavskyy conjecture} recalls Fomin and Pylyavskyy's conjectures and states a conjecture naturally extending theirs. Section \ref{sec:g vectors and dominant monomials and tableaux} highlights a tableau-theoretic rule for $g$-vectors in Grassmannian cluster algebras. Section \ref{sec: real modules} gives examples illustrating how our formulas can be used to test reality and primeness of modules, and compatibility of cluster variables.  

\subsection*{Acknowledgements}
The authors express their gratitude to Arkady Berenstein, Maxim Gurevich, Erez Lapid, and Evgeny Mukhin for helpful discussions. We are thankful to Greg Warrington for his Kazhdan-Lusztig code used in Section~\ref{sec: real modules}, and to Erez Lapid for his code which computed \eqref{eq: decomposition of tensor product 2}. We are thankful to Hiraku Nakajima for pointing us to the references \cites{GV,Vas}. W. Chang is supported by the National Natural Science Foundation of China (no. 11601295) and Shaanxi Normal University. B. Duan is supported by the National Natural Science Foundation of China (no. 11771191). W. Chang and B. Duan are supported by China Scholarship Council to visit Department of Mathematics at University of Connecticut and they thank Ralf Schiffler for hospitality during their visit. C. Fraser is supported by the NSF grant DMS-1745638. J.-R. Li is supported by the Minerva foundation with funding from the Federal German Ministry for Education and Research, by the Austrian Science Fund (FWF): M 2633-N32 Meitner Program, and by the European Research Council (ERC) under the European Union's Horizon 2020 research and innovation program (QUASIFT grant agreement 677368).

\subsection*{Notation}
For convenience of the reader, we collect key notation here. 

\begin{itemize}
\item $U_q(\widehat{\mathfrak{g}})$ the quantum affine algebra for $\mathfrak{g}$; almost everywhere we take $\mathfrak{g} =\mathfrak{sl}_n$.

\item $\mathcal{P} = \mathcal{P}_{A_{n-1}}$ the free abelian group in formal variables $Y_{i,s}^{\pm 1}$, $i \in I$, $s \in \ZZ$, $\mathcal{P}^+=\mathcal{P}^+_{A_{n-1}}$ the submonoid of $\mathcal{P}$ generated by $Y_{i,s}$, $i \in I$, $s \in \ZZ$, and $\mathcal{P}^+_\ell=\mathcal{P}^+_{\ell, A_{n-1}}$ the submonoid generated by $Y_{i,i-2k-2}$, $i \in I$, $k \in [0, \ell]$.

\item $\mathcal{C}=\mathcal{C}^{\mathfrak{g}}$ the category of all finite-dimensional $U_q(\widehat{\mathfrak{g}})$-modules; $\mathcal{C}_{\ell}=\mathcal{C}_{\ell}^{A_{n-1}}$ the subcategory defined in Section~\ref{definition of quantum affine algebras} and $K_0(\mathcal{C}_{\ell})$ its Grothendieck ring. 


\item $L(M)$ the simple $U_q(\widehat{\mathfrak{g}})$-module with highest $l$-weight $M$ and $\chi_q(M) = \chi_q(L(M))$ its $q$-character. 

\item $\Gr(n,m) \subset \mathbb{P}^{\binom m n-1}$ the Grassmannian of $n$-planes in $\mathbb{C}^m$ and $\bbc[\Gr(n,m)]$ its homogeneous coordinate ring; $\CC[\Gr(n,m,\sim)]$ the quotient of $\CC[\Gr(n,m)]$ by the Pl\"{u}cker coordinates with column set a consecutive interval; $P_{i_1, \ldots, i_n} \in \CC[{\rm Gr}(n, m)]$ a Pl\"ucker coordinate. 





\item $\Phi: K_0(\mathcal{C}_\ell) \overset{\cong}{\to} \CC[\Gr(n,n+\ell+1,\sim)]$ the isomorphism of Hernandez-Leclerc; $\widetilde{\Phi}$ the isomorphism of monoids in Theorem~\ref{thm: parametrization of simple modules by tableaux}. 


\item ${\rm SSYT}(n, [m])$ the monoid of rectangular semistandard Young tableaux with $n$ rows and with entries in $[m]$; ${\rm SSYT}(n, [m],\sim)$ the monoid of $\sim$-equivalence classes. 

\end{itemize}

\section{Hernandez-Leclerc's category and the Grassmannian} \label{sec:preliminary}
We review background of quantum affine algebras of type $A$ and their connection with Grassmannian cluster algebras. 
\subsection{Cluster algebras}
Cluster algebras were invented by Fomin and Zelevinsky \cite{FZ02}. We give a brief definition. 

For $m \in \ZZ_{\ge 1}$, we denote $[m]=\{1, \ldots, m\}$. 

A quiver $Q=(Q_0, Q_1, s, t)$ is a finite directed graph without loops or $2$-cycles, with vertex set $Q_0$, arrow set $Q_1$, and with maps $s,t: Q_1 \to Q_0$ taking an arrow to its source and target, respectively. We identify $Q_0 = [m] = \{1,\dots,m\}$. As part of the data of $Q$, one further declares vertices $1,\dots,n$ as {\sl mutable} and vertices $n+1,\dots,m$ as {\sl frozen}. 

For $k \in [n]$, the {\sl mutated quiver} $\mu_k(Q)$ is a quiver on the same vertex set and value of $n$, with arrows obtained as follows:
\begin{enumerate}
\item[(i)] for each sub-quiver $i \to k \to j$, add a new arrow $i \to j$,

\item[(ii)] reverse the orientation of every arrow with target or source equal to $k$,

\item[(iii)] remove the arrows in a maximal set of pairwise disjoint $2$-cycles. 
\end{enumerate}

Let $\mathcal{F}$ be an ambient field abstractly isomorphic to a field of rational functions in $m$ independent variables. A {\sl seed} in $\mathcal{F}$ is a pair $({\bf x}, Q)$, where ${\bf x} = (x_1, \ldots, x_m)$ form a free generating set of $\mathcal{F}$ and $Q$ is a quiver as above. 

The set ${\bf x}$ is the {\sl cluster} of the seed $({\bf x}, Q)$. The variables $x_1, \ldots, x_n$ are the {\sl cluster variables} for this seed, and 
the variables $x_{n+1}, \ldots, x_m$ are called {\sl frozen variables}. 

For a seed $({\bf x}, Q)$ and $k \in [n]$, the {\sl mutated seed} $\mu_k({\bf x}, Q)$ is  $({\bf x}', \mu_k(Q))$, where ${\bf x}' = (x_1', \ldots, x_m')$ with $x_j'=x_j$ for $j\ne k$ and $x_k' \in \mathcal{F}$ determined by
\begin{align*}
x_k' x_k = \prod_{\alpha \in Q_1, s(\alpha)=k} x_{t(\alpha)} + \prod_{\alpha \in Q_1, t(\alpha)=k} x_{s(\alpha)}.
\end{align*}

After making a choice of initial labeled seed, say that a seed is {\sl reachable} if it can be obtained from the initial seed by a finite sequence of mutations. One defines the {\sl clusters} (resp. {\sl cluster variables}) to be the clusters (resp. cluster variables) appearing in all reachable seeds. Two cluster variables are {\sl compatible} if they are in a common cluster. The {\sl cluster monomials} are the products of compatible cluster variables. The {\sl cluster algebra} is the $\mathbb{C}$-algebra generated by all cluster and frozen variables.

\subsection{Quantum affine algebras}\label{definition of quantum affine algebras}
Let $\mathfrak{g}$ be a simple Lie algebra and $I$ the indices of the Dynkin diagram of $\mathfrak{g}$. Let $C=(C_{ij})_{i,j\in I}$ be the Cartan matrix of $\mathfrak{g}$, where $C_{ij}=\frac{2 ( \alpha_i, \alpha_j ) }{( \alpha_i, \alpha_i )}$. There is a matrix $D=\diag(d_{i}\mid i\in I)$ with entries in $\mathbb{Z}_{>0}$ such that $B=DC=(b_{ij})_{i,j\in I}$ is symmetric. The matrix $D$ is an identity matrix in type $A$. 

Denote by $P = P_{\mathfrak{g}}$ the {\sl weight lattice} of $\mathfrak{g}$ and by $Q \subset P$ the {\sl root lattice} of $\mathfrak{g}$. The weight lattice is partially ordered via $\lambda \le \lambda'$ if and only if $\lambda' - \lambda$ is expressible as a nonnegative sum of positive simple roots. 

In this paper, we take $q$ to be a nonzero complex number which is not a root of unity. The {\sl quantum affine algebra} $U_q(\widehat{\mathfrak{g}})$ in Drinfeld's realization \cite{Dri} is generated by $x_{i, m}^{\pm}$ ($i\in I, m\in \mathbb{Z}$), $k_i^{\pm 1}$ $(i\in I)$, $h_{i, m}$ ($i\in I, m\in \mathbb{Z}\backslash \{0\}$) and central elements $c^{\pm 1/2}$, subject to certain relations.

\subsection{\texorpdfstring{Finite-dimensional modules and the category $\mathcal{C}_{\ell}^{A_{n-1}}$}{Finite-dimensional modules and the HL-subcategory}
}
In this section, we recall the standard facts about finite-dimensional $U_q(\widehat{\mathfrak{g}})$-modules and their $q$-characters, as well as Hernandez-Leclerc's category $\mathcal{C}_{\ell}$, see \cites{CP94, CP95a, FR, HL10}.

Let $\mathcal{C}$ be the category of finite-dimensional $U_q(\widehat{\mathfrak{g}})$-modules. In \cite{HL10}, \cite{HL16}, Hernandez and Leclerc introduced a full subcategory $\mathcal{C}_{\ell}$ ($\ell \in \mathbb{Z}_{\geq 0}$) of $\mathcal{C}$. We reproduce the definitinon for $\mathfrak{g}$ of type $A$. 

Let $\mathfrak{g}=\mathfrak{sl}_n$ and $I=[1,n-1]$ be the set of vertices of the Dynkin diagram of $\mathfrak{g}$. We fix $a \in \CC^{\times}$ and denote $Y_{i,s} = Y_{i,aq^s}$, $i \in I$, $s \in \ZZ$. Denote by $\mathcal{P} = \mathcal{P}_{A_{n-1}}$ the free abelian group generated by $Y_{i,s}^{\pm 1}$, $i \in I$, $s \in \ZZ$, denote by $\mathcal{P}^+ = \mathcal{P}_{A_{n-1}}^+$ the submonoid of $\mathcal{P}$ generated by $Y_{i,s}$, $i \in I$, $s \in \ZZ$, and denote by $\mathcal{P}^+_\ell= \mathcal{P}_{\ell,A_{n-1}}^+$ the submonoid of $\mathcal{P}^+$ generated by $Y_{i,i-2k-2}$, $i \in I$, $k \in [0, \ell]$. An object $V$ in $\mathcal{C}_{\ell}=\mathcal{C}_{\ell}^{A_{n-1}}$ is a finite-dimensional $U_q(\widehat{\mathfrak{g}})$-module which satisfies the condition: for every composition factor $S$ of $V$, the highest $l$-weight of $S$ is a monomial in $Y_{i, i-2k-2}$, $k \in [0, \ell]$, $i\in I$, \cite{HL10}. Simple modules in $\mathcal{C}_{\ell}$ are of the form $L(M)$ (cf. \cite{CP94}, \cite{HL10}), where $M \in \mathcal{P}_{\ell, A_{n-1}}^+$ and $M$ is called the highest \textit{$l$-weight} (or sometimes, {\sl loop-weight}) of $L(M)$. The elements of $\mathcal{P}^+$ are called {\sl dominant monomials}.

Denote by $K_0(\mathcal{C}_{\ell}^{A_{n-1}})$ the Grothendieck ring of $\mathcal{C}_{\ell}^{A_{n-1}}$. By a slight abuse of notation, sometimes we write $[L(M)]$ ($M \in \mathcal{P}^+$) in $K_0(\mathcal{C}_{\ell}^{A_{n-1}})$ as $L(M)$ or as $[M]$ and we refer to elements $\mathcal{R}_{\ell}^{A_{n-1}}$ merely as modules. 

Let $\mathbb{Z}\mathcal{P} = \mathbb{Z}[Y_{i, s}^{\pm 1}]_{i\in I, s\in \mathbb{Z}}$ be the group ring of $\mathcal{P}$. The $q$-{\sl character} of a $U_q(\widehat{\mathfrak{g}})$-module $V$ is given by (cf. \cite{FR})
\begin{align*}
\chi_q(V) = \sum_{M\in \mathcal{P}} \dim(V_{M}) M \in \mathbb{Z}\mathcal{P},
\end{align*}
where $V_M$ is an $l$-weight space of $M$. For a module $L(M)$, $M \in \mathcal{P}^+$, we also write $\chi_q(M) = \chi_q(L(M))$. 

We denote $\wt: \mathcal{P} \to P_{\mathfrak{g}}$ the group homomorphism defined by sending $Y_{i,a}^{\pm} \mapsto \pm \omega_i$, $i \in I$, where $\omega_i$'s are fundamental weights of $\mathfrak{g}$. For a finite-dimensional simple $U_q(\widehat{\mathfrak{g}})$-module $L(M)$, we write $\wt(L(M)) = \wt(M)$ and call it the highest weight of $L(M)$. 

Let $\mathcal{Q}$ be the subgroup of $\mathcal{P}$ generated (when $\mathfrak{g} = \mathfrak{sl}_n$) by
\begin{align} \label{eq:Aia}
A_{i,s} = Y_{i,s+1}Y_{i,s-1} \prod_{j \in I, |j-i|=1} Y_{j,s}^{-1}, \quad i\in I, \ s\in \mathbb{Z}.
\end{align}
Let $\mathcal{Q}^{\pm}$ be the monoids generated by $A_{i, a}^{\pm 1}, i\in I, a\in \mathbb{C}^{\times}$. There is a partial order $\leq$ on $\mathcal{P}$ (cf. \cites{FM,Nak00}) in which
\begin{align}
M \leq M' \text{ if and only if } M'M^{-1}\in \mathcal{Q}^{+}. \label{partial order of monomials}
\end{align}

A finite-dimensional $U_q(\widehat{\mathfrak{g}})$-module is called \textit{prime} if it is not isomorphic to a tensor product of two nontrivial $U_q(\widehat{\mathfrak{g}})$-modules (cf. \cite{CP97}). A simple $U_q(\widehat{\mathfrak{g}})$-module $M$ is {\sl real} if $M \otimes M$ is simple (cf. \cite{Lec}).

\subsection{\texorpdfstring{\text{Cluster structure on $K_0(\mathcal{C}_{\ell}^{A_{n-1}})$}}{Cluster structure on the Grothendieck ring}}\label{subsec:Grothendieck ring and Grassmannian}
Hernandez and Leclerc introduced monoidal categorifications of cluster algebras in \cites{HL10,HL16}. We recall the definition of Hernandez and Leclerc's cluster algebras introduced in \cite{HL16}, again in type $A_{n-1}$ only. Let $Q_{\ell}$ be a quiver with the vertex set $V_{\ell}=I \times  [0, \ell]$ (i.e., a rectangular grid), and with edge set: 
\begin{align*}
& (i,r) \to (j,r+1), \quad j-i=1, \\
& (i, r) \to (i, r-1), \\
& (i, r) \to (i-1, r).
\end{align*}
Let $\mathbf{z}=\{z_{i,t}: (i,t)\in V_{\ell}\}$ and let $\mathcal{A}_\ell=\mathcal{A}_{\ell}^{A_{n-1}}$ be the cluster algebra defined by the initial seed $(\mathbf{z}, Q_{\ell}^{A_{n-1}})$, where $z_{i, \ell}$, $i \in I$, are frozen variables. 

For $i \in I$, $s \in \mathbb{Z}$, $k \in \ZZ_{\ge 1}$, we denote
\begin{align}\label{eq:KR}
X_{i,k}^{(s)} = Y_{i,s} Y_{i,s+2} \cdots Y_{i,s+2k-2}.
\end{align}
The modules $L(X_{i,k}^{(s)})$ are called {\sl Kirillov-Reshetikhin modules} and their classes $[L(X_{i,k}^s]$ serve as initial cluster variables. When $k=1$, the modules $L(X_{i,1}^{(s)}) = L(Y_{i,s})$ are called {\sl fundamental modules}. 

It is shown in \cite{HL10}, \cite{HL16} that the assignments $z_{i,t} \mapsto L(X_{i, t+1}^{(i-2t-2)})$, $i \in I$, $t \in [0, \ell]$, extend to a ring isomorphism $\mathcal{A}_{\ell}^{A_{n-1}} \to \mathcal{R}_{\ell}^{A_{n-1}}$. Figure \ref{fig:initial cluster for Uqslhat5} is the initial cluster for $\mathcal{R}_{4}^{A_4}$. The copies of the trivial module $\mathbb{C}$ are not part of the initial cluster, but are only drawn for comparison with the initial cluster for the Grassmannian defined in the next section. Likewise, the quiver $Q_4^{A_4}$ is obtained from the one in Figure \ref{fig:initial cluster for Uqslhat5} by deleting these vertices. We identify throughout the elements of our initial cluster with entries in the $[n-1] \times  [0, \ell]$ rectangular grid.

\begin{figure}
\begin{tikzpicture}[scale=0.5]
     \node at (0,0) (v00) {\fbox{$\mathbb{C}$}};
     \node at (0,-4) (v10) {$X_{1,1}^{(-1)}$};
     \node at (0,-8) (v20) {$X_{1,2}^{(-3)}$};
     \node at (0,-12) (v30) {$X_{1,3}^{(-5)}$};
     \node at (0,-16) (v40) {$X_{1,4}^{(-7)}$};
     \node at (0,-20) (v50) {\fbox{$X_{1,5}^{(-9)}$}};
    
     \node at (4,-4) (v11) {$X_{2,1}^{(0)}$};
     \node at (4,-8) (v21) {$X_{2,2}^{(-2)}$};
     \node at (4,-12) (v31) {$X_{2,3}^{(-4)}$};
     \node at (4,-16) (v41) {$X_{2,4}^{(-6)}$};
     \node at (4,-20) (v51) {\fbox{$X_{2,5}^{(-8)}$}};
     
     \node at (8,-4) (v12) {$X_{3,1}^{(1)}$};
     \node at (8,-8) (v22) {$X_{3,2}^{(-1)}$};
     \node at (8,-12) (v32) {$X_{3,3}^{(-3)}$};
     \node at (8,-16) (v42) {$X_{3,4}^{(-5)}$};
     \node at (8,-20) (v52) {\fbox{$X_{3,5}^{(-7)}$}};
     
     \node at (12,-4) (v13) {$X_{4,1}^{(2)}$};
     \node at (12,-8) (v23) {$X_{4,2}^{(0)}$};
     \node at (12,-12) (v33) {$X_{4,3}^{(-2)}$};
     \node at (12,-16) (v43) {$X_{4,4}^{(-4)}$};
     \node at (12,-20) (v53) {\fbox{$X_{4,5}^{(-6)}$}};
     
     \node at (16,-4) (v14) {\fbox{$\mathbb{C}$}};
     \node at (16,-8) (v24) {\fbox{$\mathbb{C}$}};
     \node at (16,-12) (v34) {\fbox{$\mathbb{C}$}};
     \node at (16,-16) (v44) {\fbox{$\mathbb{C}$}};
     \node at (16,-20) (v54) {\fbox{$\mathbb{C}$}};
     
     \draw[->] (v10)--(v00);
     \draw[->] (v20)--(v10);
     \draw[->] (v30)--(v20);
     \draw[->] (v40)--(v30);
     \draw[->] (v50)--(v40);

     \draw[->] (v21)--(v11);
     \draw[->] (v31)--(v21);
     \draw[->] (v41)--(v31);
     \draw[->] (v51)--(v41);
     
     \draw[->] (v22)--(v12);
     \draw[->] (v32)--(v22);
     \draw[->] (v42)--(v32);
     \draw[->] (v52)--(v42);
     
     \draw[->] (v23)--(v13);
     \draw[->] (v33)--(v23);
     \draw[->] (v43)--(v33);
     \draw[->] (v53)--(v43);
     
     \draw[->] (v11)--(v10);
     \draw[->] (v12)--(v11);
     \draw[->] (v13)--(v12);
     \draw[->] (v14)--(v13);
     
     \draw[->] (v21)--(v20);
     \draw[->] (v22)--(v21);
     \draw[->] (v23)--(v22);
     \draw[->] (v24)--(v23);
     
     \draw[->] (v31)--(v30);
     \draw[->] (v32)--(v31);
     \draw[->] (v33)--(v32);
     \draw[->] (v34)--(v33);
     
     \draw[->] (v41)--(v40);
     \draw[->] (v42)--(v41);
     \draw[->] (v43)--(v42);
     \draw[->] (v44)--(v43);
     
     \draw[->] (v10)--(v21);
     \draw[->] (v21)--(v32);
     \draw[->] (v32)--(v43);
     \draw[->] (v43)--(v54);
     
     \draw[->] (v20)--(v31);
     \draw[->] (v31)--(v42);
     \draw[->] (v42)--(v53);
     
     \draw[->] (v30)--(v41);
     \draw[->] (v41)--(v52);
     
     \draw[->] (v40)--(v51);
     
     \draw[->] (v11)--(v22);
     \draw[->] (v22)--(v33);
     \draw[->] (v33)--(v44);
     
     \draw[->] (v12)--(v23);
     \draw[->] (v23)--(v34);
     
     \draw[->] (v13)--(v24);
     
\end{tikzpicture}
             \caption{The initial cluster for $\mathcal{R}_{4}^{A_4}$.}
             \label{fig:initial cluster for Uqslhat5}
\end{figure}
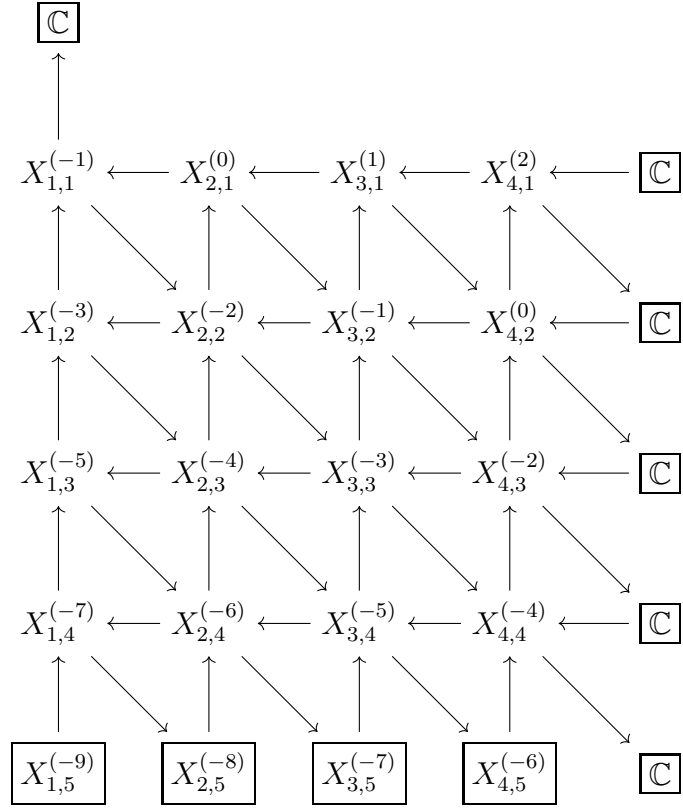
\subsection{Grassmannian cluster algbras}\label{susbsecn:Grassmannians}
Let $\Gr(n,m) \subset \mathbb{P}^{\binom m n-1}$ denote the Grassmannian of $n$-planes in $\mathbb{C}^m$, together with its Pl\"ucker embedding in projective space. Let $\CC[\Gr(n,m)]$ denote the homogeneous coordinate ring. This algebra is generated by Pl\"{u}cker coordinates 
\begin{align*}
P_{i_1, \ldots, i_{n}}, \quad 1 \leq i_1 < \cdots < i_{n} \leq m.
\end{align*}

Scott \cite{Sco} (see also \cite{GSV}) introduced a cluster algebra structure on $\CC[\Gr(n,m)]$. Setting $m = n + \ell+1$, this cluster algebra has initial seed $(Q'_{\ell}, {\bf z}')$, where $Q'_{\ell}$ is the quiver obtained from $Q_{\ell}^{A_{n-1}}$ by adding frozen vertices $(0, 0)$, $(n, t)$, $t \in [0, \ell]$ and adding arrows $(1,0) \to (0,0)$, $(n,t) \to (n-1, t)$, $(n-1, t) \to (n, t+1)$, $t \in [0, \ell-1]$,
and ${\bf z}'=\{ P_{[1, n-i] \cup [n-i+t+2, n+t+1]} : (i,t)\in V'\}$.

Let $\ZZ^m$ be the free abelian group with standard basis vectors $e_1,\dots,e_m$. There is a $\ZZ^m$-grading on the algebra $\CC[\Gr(n,m)]$ in which the Pl\"ucker coordinate $P_{i_1, \ldots, i_{n}}$ has $\ZZ^m$-degree $e_{i_1}+ \cdots + e_{i_n}$. It is well known every cluster monomial in $\CC[\Gr(n,m)]$, and moreover all exchange relations, are homogeneous with respect to this grading.

Denote by $\CC[\Gr(n,m,\sim)]$ the quotient of $\CC[\Gr(n,m)]$ by the inhomogeneous ideal 
\begin{align}\label{eq:idealdefn}
\langle P_{i,i+1, \ldots, i+n-1}-1, \quad i \in [m-n+1] \rangle.
\end{align}
We use the same notation $P_{i_1,\ldots,i_{n}}$ for the image of a Pl\"ucker coordinate in this quotient. We refer to the frozen Pl\"ucker coordinates appearing in \eqref{eq:idealdefn} as {\sl trivial frozens} to distinguish them from the frozens we do not specialize to~1. Deleting the trivial frozen variables from the initial seed for $\CC[\Gr(n,m)]$ yields a cluster structure on $\CC[\Gr(n,m,\sim)]$. The initial cluster for $\CC[\Gr(n,m,\sim)]$ is in Figure~\ref{fig:initial cluster for a quotient of Gr(5,10)}. 

Note that the $\ZZ^m$-degrees of the trivial frozens are linearly independent. Thus any Laurent monomial in these trivial frozen variables is determined by its $\ZZ^m$-degree. We use this idea later on to lift certain inhomogeneous formulas valid in $\CC[\Gr(n,m,\sim)]$ to homogeneous ones valid in $\CC[\Gr(n,m)]$.

\begin{figure}
\begin{tikzpicture}[scale=0.5]
     \node at (0,0) (v00) {\fbox{$P_{1,2,3,4,5}=1$}};
     \node at (0,-4) (v10) {$P_{1,2,3,4,6}$};
     \node at (0,-8) (v20) {$P_{1,2,3,4,7}$};
     \node at (0,-12) (v30) {$P_{1,2,3,4,8}$};
     \node at (0,-16) (v40) {$P_{1,2,3,4,9}$};
     \node at (0,-20) (v50) {\fbox{$P_{1,2,3,4,10}$}};
    
     \node at (4,-4) (v11) {$P_{1,2,3,5,6}$};
     \node at (4,-8) (v21) {$P_{1,2,3,6,7}$};
     \node at (4,-12) (v31) {$P_{1,2,3,7,8}$};
     \node at (4,-16) (v41) {$P_{1,2,3,8,9}$};
     \node at (4,-20) (v51) {\fbox{$P_{1,2,3,9,10}$}};
     
     \node at (8,-4) (v12) {$P_{1,2,4,5,6}$};
     \node at (8,-8) (v22) {$P_{1,2,5,6,7}$};
     \node at (8,-12) (v32) {$P_{1,2,6,7,8}$};
     \node at (8,-16) (v42) {$P_{1,2,7,8,9}$};
     \node at (8,-20) (v52) {\fbox{$P_{1,2,8,9,10}$}};
     
     \node at (12,-4) (v13) {$P_{1,3,4,5,6}$};
     \node at (12,-8) (v23) {$P_{1,4,5,6,7}$};
     \node at (12,-12) (v33) {$P_{1,5,6,7,8}$};
     \node at (12,-16) (v43) {$P_{1,6,7,8,9}$};
     \node at (12,-20) (v53) {\fbox{$P_{1,7,8,9,10}$}};
     
     \node at (17,-4) (v14) {\fbox{$P_{2,3,4,5,6}=1$}};
     \node at (17,-8) (v24) {\fbox{$P_{3,4,5,6,7}=1$}};
     \node at (17,-12) (v34) {\fbox{$P_{4,5,6,7,8}=1$}};
     \node at (17,-16) (v44) {\fbox{$P_{5,6,7,8,9}=1$}};
     \node at (17,-20) (v54) {\fbox{$P_{6,7,8,9,10}=1$}};
     
     \draw[->] (v10)--(v00);
     \draw[->] (v20)--(v10);
     \draw[->] (v30)--(v20);
     \draw[->] (v40)--(v30);
     \draw[->] (v50)--(v40);

     \draw[->] (v21)--(v11);
     \draw[->] (v31)--(v21);
     \draw[->] (v41)--(v31);
     \draw[->] (v51)--(v41);
     
     \draw[->] (v22)--(v12);
     \draw[->] (v32)--(v22);
     \draw[->] (v42)--(v32);
     \draw[->] (v52)--(v42);
     
     \draw[->] (v23)--(v13);
     \draw[->] (v33)--(v23);
     \draw[->] (v43)--(v33);
     \draw[->] (v53)--(v43);
     
     \draw[->] (v11)--(v10);
     \draw[->] (v12)--(v11);
     \draw[->] (v13)--(v12);
     \draw[->] (v14)--(v13);
     
     \draw[->] (v21)--(v20);
     \draw[->] (v22)--(v21);
     \draw[->] (v23)--(v22);
     \draw[->] (v24)--(v23);
     
     \draw[->] (v31)--(v30);
     \draw[->] (v32)--(v31);
     \draw[->] (v33)--(v32);
     \draw[->] (v34)--(v33);
     
     \draw[->] (v41)--(v40);
     \draw[->] (v42)--(v41);
     \draw[->] (v43)--(v42);
     \draw[->] (v44)--(v43);
     
     \draw[->] (v10)--(v21);
     \draw[->] (v21)--(v32);
     \draw[->] (v32)--(v43);
     \draw[->] (v43)--(v54);
     
     \draw[->] (v20)--(v31);
     \draw[->] (v31)--(v42);
     \draw[->] (v42)--(v53);
     
     \draw[->] (v30)--(v41);
     \draw[->] (v41)--(v52);
     
     \draw[->] (v40)--(v51);
     
     \draw[->] (v11)--(v22);
     \draw[->] (v22)--(v33);
     \draw[->] (v33)--(v44);
     
     \draw[->] (v12)--(v23);
     \draw[->] (v23)--(v34);
     
     \draw[->] (v13)--(v24);
     
\end{tikzpicture}
            \caption{The initial cluster for $\CC[\Gr(5,10,\sim)]$.}
            \label{fig:initial cluster for a quotient of Gr(5,10)}
\end{figure}

For $a,b,c \geq 0$ denote by $P^{(a, b, c)}$ the Pl\"{u}cker coordinate $P_{j_1, \ldots, j_n}$ where $j_{1}=b$, $j_k=j_{k-1}+1$ for $k \in [2, a] \cup [a+2, n]$, and $j_{a+1}-j_{a}=c$. Thus, $P^{(a, b, c)}$ consists of two intervals (one of size $a$ and the other of size $n-a$) separated by a gap of size 
$c$. 

\begin{theorem}[{\cite[Section 13]{HL10}}]\label{thm:Hernandez-Leclerc quantum affine algebras and Grassmannians}
The assignments 
\begin{align*}
L(X_{i,t+1}^{(i-2t-2)}) \mapsto P^{(n-i, 1, t+2)},  \quad i \in I, \ t \in [0, \ell], 
\end{align*}
extend to an algebra isomorphism $\Phi \colon K_0(\mathcal{C}_\ell) \to \CC[\Gr(n,n+\ell+1,\sim)]$, respecting cluster structures. 
\end{theorem}

\section{\texorpdfstring{Simple $U_q(\widehat{\mathfrak{g}})$-modules and tableaux}{Simple modules and tableaux}} \label{sec:Uqg hat modules and tableaux}
We introduce three structures (the weight map, dominance partial order, and monoid structure) on semistandard tableaux and then make the connection between these and simple $U_q(\widehat{\mathfrak{g}})$-modules. We end the section by defining the elements $\ch(T) \in \bbc[\Gr(n,m,\sim)]$ and comparing the partial order on tableaux with the partial order on dominant monomials. 

A {\sl semistandard Young tableau} is a Young tableau with weakly increasing rows and strictly increasing columns. For $n, m \in \ZZ_{\ge 1}$, we denote by ${\rm SSYT}(n, [m])$ the set of rectangular semistandard Young tableaux with $n$ rows and with entries in $[m]$ (with arbitrarly many columns). We denote the empty tableau by $\mathds{1}$ and consider it an element of ${\rm SSYT}(n, [m])$. The {\sl content} of a tableau $T$ is the vector $(\nu_1,\dots,\nu_m) \in \ZZ^m$, where $\nu_i$ is the number of $i$-filled boxes in $T$.

\subsection{Weight and partial order for tableaux} \label{sec:Weights of semistandard Young tableaux}
Semistandard tableaux with one column $T$ are in apparent bijection with Pl\"ucker coordinates $P$. We let $T_P$ be the tableau corresponding to a Pl\"ucker coordinate and $P_T$ be the Pl\"ucker coordinate corresponding to a single-column tableau $T$. Extending this, for tableau $T$ with columns $T_1, \ldots, T_k$, let $P_T = P_{T_1}\cdots P_{T_k}$ be the corresponding monomial in Pl\"ucker coordinates. This definition makes sense even when $T$ is not semistandard. A {\sl standard monomial} for $\CC[\Gr(n,m)]$ is one of the form $P_T$ for $T \in {\rm SSYT}(n,[m])$. The standard monomials are a basis for $\CC[\Gr(n,m)]$.

\begin{definition}\label{def:weight of product of Plucker coordinaes}
For a Pl\"{u}cker coordinate $P=P_{i_1, \ldots, i_n} \in \CC[\Gr(n,m)]$, the {\sl weight} of $P$ is
\begin{align*}
\wt(P) = \sum_{j=2}^{n} (i_j-i_{j-1}-1) \omega_{n-j+1} \in P_{\mathfrak{g}}
\end{align*}
where $\omega_k$'s are fundamental weights of $\mathfrak{g}$. We define also the {\sl gap weight} of $P$ to be $\sum_{j}(i_j-i_{j-1}-1)$, i.e. the image of $\wt(P)$ under the map $P_{\mathfrak{g}} \to \ZZ$ specializing all  $\omega_{j} \mapsto 1$. We additively extend the notions of weight and gap weight to monomials in Pl\"ucker coordinates, so that the weight of a product is the sum of the weights. We define the weight (resp. gap weight) of a tableau to be the weight (resp. gap weight) of $P_T$. (We define $\wt(\mathds{1}) = 0$).
\end{definition}

Using the partial order on the weight lattice $P_{\mathfrak{g}}$, one obtains a preorder on ${\rm SSYT}(n, [m])$. We now recall the definition of a partial order (sometimes called the {\sl dominance order on tableaux}) which closely matches the partial order on dominant monomials in $\mathcal{P}^+$. For computing exchange relations in the cluster algebra, it will turn out that one can use either this dominance order or the (weaker) weight order.

The definition of the partial order uses the more familiar dominance order on partitions. Let $\lambda = (\lambda_1,\dots,\lambda_\ell)$ with $\lambda_1 \geq \lambda_2 \ge \cdots \geq \lambda_\ell \geq 0$ be a partition, and $\mu = (\mu_1,\dots,\mu_\ell)$ another partition. Then $\lambda \geq \mu$ in dominance order if $\sum_{j \leq i}\lambda_j \geq \sum_{j \leq i}\mu_j$ for $i=1,\dots,\ell$. For a tableau $T$, let ${\rm sh}(T)$ denote the shape of $T$. For $i \in [m]$, let $T[i]$ denote the restriction of $T \in {\rm SSYT}(n,[m])$ to the entries in $[i]$. 

\begin{definition}\label{defn:dominanceontableaux}
For $T,T' \in {\rm SSYT}(n, [m])$ with the same content, we say that $T \geq T'$ if ${\rm sh}(T[i]) \geq {\rm sh}(T'[i])$ in the dominance order on partitions, for $i=1,\dots,m$. 
\end{definition}

In our proofs, we use the following description of cover relations in this poset: if $T \geq T'$, then there exists a sequence of tableaux $T = T_0 \geq T_1 \geq \dots \geq T_a = T'$, with successive terms in this sequence related by transposing the entries in a pair of boxes (this follows, e.g., by standardization and \cite[Proposition 2.3]{Cast} applied to tableaux of the same shape). 

\begin{lemma}\label{lem:dominanceimpliesweight}
If $T \geq T'$ (in the sense of Definition~\ref{defn:dominanceontableaux}) then $\wt(T) \geq \wt(T')$ in $P_{\mathfrak{g}}$. 
\end{lemma}


\begin{proof}
It suffices to prove this when $T> T'$ are related by a transposition in a pair of boxes. Suppose $T$ is obtained from $T'$ by swapping an entry $y$ in row $i$ of $T$ with an entry $x$ in row $j$ of $T'$, with $i < j$. Then $y > x$, and $x,y$ are in different columns of $T'$. Let $x_U,x_D$ be entries above and below $x$ respectively and define $y_U,y_D$ similarly. The degenerate cases when $y$ is in the first row or $x$ is in the last row work by the same analysis (treating $\omega_0 = \omega_n = 0$). From the definition of $\wt$, performing the transposition we have that 
$$\wt(T) - \wt(T') = (y-x)(\omega_{n-i+1}-\omega_{n-i+2} +\omega_{n-j+2}-\omega_{n-j+1}),$$
and the second factor is the sum of positive simple roots $\alpha_{i}+ \cdots +\alpha_{j-1}$.  
\end{proof}


\subsection{The tableau monoid} \label{subsec:monoid of rectangular semistandard Young tableaux}
As in the introduction, for $S,T \in {\rm SSYT}(n, [m])$, we denote by $S \cup T$ the row-increasing tableau whose $i$th row is the union of the $i$th rows of $S$ and $T$ (as multisets). Note for instance that every $T \in {\rm SSYT}(n, [m])$ factors as the $\cup$-product of its columns. 

We call $S$ a {\sl factor} of $T$, and write $S \subset T$, if the $i$th row of $S$ is contained in that of $T$ (as multisets), for $i \in [n]$. In this case, we define $\frac{T}{S}=S^{-1}T=TS^{-1}$ to be the row-increasing tableau whose $i$th row is obtained by removing that of 
of $S$ from that of $T$ (as multisets), for $i \in [n]$.

\begin{defn}\label{defn:sim}
A tableau $T \in {\rm SSYT}(n, [m])$ is {\sl trivial} if $\wt(T) = 0 \in P_{\mathfrak{g}}$. That is, each entry of $T$ is one less than the entry below it. 

For any $T \in {\rm SSYT}(n, [m])$, we  denote by $T_{\text{red}} \subset T$ the semistandard tableau obtained by removing a maximal trivial factor from $T$. That is, $T_{\text{red}}$ is the tableau with the minimal number of columns such that $T = T_{\text{red}} \cup S$ for a trivial tableau $S$. For trivial $T$ one has $T_{\text{red}} = \mathds{1}$. For $S, T \in {\rm SSYT}(n, [m])$, define $S \sim T$ if $S_{\text{red}} = T_{\text{red}}$. It is clear that~``$\sim$'' is an equivalence relation. We denote by ${\rm SSYT}(n, [m],\sim)$ the set of $\sim$-equivalence classes.
\end{defn}

We use the same notation for a tableau $T$ and its equivalence class, writing either $T \in {\rm SSYT}(n, [m])$ or $T \in {\rm SSYT}(n, [m],\sim)$ when it is important to distinguish these.

\begin{example}
We illustrate the operations $\cup$ and $\sim$: 
\begin{align*}
\begin{ytableau}
1 & 3    \\
2 & 7 \\
6 & 11
\end{ytableau}
\cup
\begin{ytableau}
1 & 7   \\
2 & 9 \\
8 & 10
\end{ytableau} = 
\begin{ytableau}
1 & 1 & 3 & 7   \\
2 & 2 & 7 & 9 \\
6 & 8 & 10 & 11
\end{ytableau}
\text{ and }
\begin{ytableau}
1 \\
3 \\
6 
\end{ytableau} \sim 
\begin{ytableau}
1 & 2 & 3  \\
3 & 3 & 4 \\
4 & 5 & 6
\end{ytableau}.
\end{align*}
\end{example}

A commutative monoid $\mathcal{M}$ is called {\sl cancellative} if for every $a,b,c \in \mathcal{M}$, $ab=ac$ implies that $b=c$. Any such monoid embeds in its Grothendieck group $K_0(\mathcal{M})$, i.e. the set of ``fractions'' of elements of $\mathcal{M}$ (subject to the same equivalences of fractions one uses to define the rational numbers from the integers). 

\begin{lemma} \label{lem:SSYT is a monoid}
The set ${\rm SSYT}(n, [m])$, and also ${\rm SSYT}(n, [m], \sim)$, form a commutative cancellative monoid with the multiplication $``\cup$''.
\end{lemma}

\begin{proof}
We will prove that for $T, T' \in {\rm SSYT}(n, [m])$, we have $T \cup T' \in {\rm SSYT}(n, [m])$. The other results in the lemma are immediate. 

Denote by $T(i)$ the $i$th row of a tableau $T$. We need to prove that for any $i < j$, the $2$-row tableau with the first row $T(i) \cup T'(i)$ and the second row $T(j) \cup T'(j)$ is semistandard. It suffices to prove this when $T'$ has one column. We can write the $i, j$ rows of $T$ as
\begin{align*}
\begin{matrix}
a_1 & a_2 & \cdots & a_m \\
b_1 & b_2 & \cdots & b_m,
\end{matrix}
\end{align*}
and suppose $T'$ has entries $a'$ and $b'$ in rows $i$ and $j$. 
There are $k, l \in [0, m]$ such that $a_1 \le \cdots \le a_k \le a' \le a_{k+1} \le \cdots \le a_m$ and $b_1 \le \cdots \le b_l \le b' \le b_{k+1} \le \cdots \le b_m$. 

If $k=l$, then the $i,j$ rows of $T \cup T'$ form a $2$-row semistandard tableau. If $k > l$, then the $i,j$ rows of $T \cup T'$ are 
\begin{align*}
a_1 & a_2 & \cdots & a_l & a_{l+1} & a_{l+2} & \cdots & a_k & a' & a_{k+1} & \cdots & a_m \\
b_1 & b_2 & \cdots & b_l & b' & b_{l+1} & \cdots & b_{k-1} & b_{k} & b_{k+1} & \cdots & b_m.
\end{align*}
We have $a' < b' \le b_k$, $a_{l+1} \le a' < b'$, and for all $j \in [l+2, k]$, $a_j \le a' < b' \le b_{j-1}$. Therefore the $i,j$ rows of $T \cup T'$ form a $2$-row semistandard tableau. 

If $k < l$, then the $i,j$ rows of $T \cup T'$ are 
\begin{align*}
\begin{array}{cccccccccccc}
a_1 & a_2 & \cdots & a_k & a' & a_{k+1} & \cdots & a_{l-1} & a_l & a_{l+1} & \cdots & a_m \\
b_1 & b_2 & \cdots & b_k & b_{k+1} & b_{k+2} & \cdots & b_{l} & b' & b_{l+1} & \cdots & b_m.
\end{array}
\end{align*}
We have $a' \le a_{k+1} < b_{k+1}$, $a_{l} < b_l \le b'$, and for all $j \in [k+1, l-1]$, $a_j < b_j \le b_{j+1}$. Therefore the $i,j$ rows of $T \cup T'$ form a $2$-row semistandard tableau. 
\end{proof}

\begin{remark} A {\sl Gelfand-Tsetlin pattern} (abbreviated {\sl G-T pattern}) is a triangular array of nonnegative integers satisfying certain inequalities. The specific details are not important for our purposes. These inequalities are preserved under entrywise addition of G-T patterns, so the set of G-T patterns is naturally a monoid (as is well known). One can check that the standard bijection between semistandard Young tableaux and G-T patterns intertwines the monoid structure $\cup$ on tableaux with the entrywise addition of G-T patterns. Thus, the current paper could be phrased purely in terms of G-T patterns. We prefer tableaux because the translation to $\CC[\Gr(n,m)]$ is clearer. For example, a Pl\"ucker coordinate $P$ is directly the same data as a single-column tableau $T_P$, and webs are already naturally labeled by tableaux (cf.~Section~\ref{sec:Fomin Pylyavskyy conjecture}).  
\end{remark}

\begin{lemma}\label{lem:weightmaphomom}
The weight map $\wt \colon {\rm SSYT}(n,[m]) \to P_{\mathfrak{g}}$ is a homomorphism of monoids. 
\end{lemma}

It follows that $\wt(T \cup T') = \wt(T)$ when $T'$ is trivial; thus ${\rm SSYT}(n, [m],\sim)$ is endowed with a weight map. 

\begin{proof}
Let $S,T \in {\rm SSYT}(n,[m])$ be given. Let $s_1,\dots,s_k$ be the entries in row $s-1$ of $S$, and $s'_1,\dots,s'_k$ be the entries directly beneath them. Let $t_1,\dots,t_j$ and $t'_1,\dots,t'_j$ be the elements in the corresponding rows of $T$. Write
$A = \{s_1,\dots,s_k,t_1,\dots,t_j\}$ as $a_1 \leq \cdots \leq a_{j+k}$ in sorted order, and likewise write 
$B = \{s'_1,\dots,s'_k,t'_1,\dots,t'_j\}$ as $b_1 \leq \cdots \leq b_{j+k}$. Then in $S \cup T$, this row contributes
$\sum_{i=1}^{k+j}(b_i - a_i-1)$ to the fundamental weight $\omega_{n-s}$. Rearranging terms, this agrees with $\sum_{i=1}^k (s'_i - s_i-1)+\sum_{i=1}^{j} (t'_i - t_i-1)$, which is the sum of the contributions from $\wt(S)$ and $\wt(T)$. 
\end{proof}

\subsection{Tableaux and modules} \label{subsec:tableaux and modules}
We begin making the connection between simple modules and tableaux. 

For starters, we describe the images $\Phi(L(M))$ for fundamental modules $L(M)$. For $(i,s) \in I \times (2 \ZZ_{\le 0} + i-2)$, denote by $P_{(i,s)} = P^{(n-i, \frac{i-s}{2}, 2)}$ the Pl\"ucker coordinate as defined just before Theorem~\ref{thm:Hernandez-Leclerc quantum affine algebras and Grassmannians}. Thus the index set of $P_{(i,s)}$ is an interval with an element removed, namely $[\frac{i-s}{2},\frac{i-s}{2}+n] \setminus \{\frac{i-s}{2}+n-i\}$. 
\begin{lemma}\label{lem:fundamental modules and Plucker coordinates}
For fundamental modules $L(Y_{i,s}) \in  \mathcal{C}_{\ell}^{A_{n-1}}$, $i \in I$, $s \in 2 \ZZ_{\le 0} + i-2$, we have $\Phi([L(Y_{i,s})]) = P_{(i,s)}$. Moreover, $\wt(Y_{i,s)} = \wt(P_{(i,s)}) = \omega_i$. 
\end{lemma}

Table \ref{table:correspondence Plucker coordinates and fundamental modules sl3} illustrates the correspondence $Y_{i,s} \mapsto P_{(i,s)}$ for $\mathcal{C}_{9}^{A_2}$. We call the Pl\"ucker coordinates $P_{(i,s)}$ arising in this correspondence {\sl fundamental Pl\"ucker coordinates}, and call a single-column tableau $T$ {\sl fundamental} if $P_T$ is. The fundamental tableaux are those with one column and with gap weight exactly equal to~1. They play an important role in what follows. 

\begin{proof}
The modules $L(Y_{i,s})$ satisfy the following $T$-{\sl system} relations \cite{Her}:
\begin{align*}
[L(Y_{i,s})] [L(Y_{i,s-2})] = [L(Y_{i,s} Y_{i, s-2})] + [L(Y_{i-1, s-1})] [L(Y_{i+1, s-1})].
\end{align*}

On the other hand, by the Pl\"ucker relations, one has 
\begin{align*}
P_{(i,s)} P_{(i,s-2)} & = P^{(n-i, \frac{i-s}{2}, 2)} P^{(n-i, \frac{i-s}{2}+1, 2)} \\
& = P^{(n-i, \frac{i-s}{2}, 3)} P_{j_1+1, j_1+2, \ldots, j_{n}} + P^{(n-i+1, \frac{i-s}{2}, 2)} P^{(n-i-1, \frac{i-s}{2}+1, 2)} \\
& = P^{(n-i, \frac{i-s}{2}, 3)}  + P_{(i-1, s-1)} P_{(i+1,s-1)},
\end{align*}
since $P_{j_1+1, j_1+2, \ldots, j_{n}}=1$, where $j_1 = \frac{i-s}{2}$. 

By the definition of $\Phi$, for the Kirillov-Reshetikhin module $L(Y_{i,s} Y_{i, s-2})$, we have that $\Phi(L(Y_{i,s} Y_{i, s-2})) = P^{(n-i, \frac{i-s}{2}, 3)}$. 

We now prove the result by induction on $s$. By the definition of $\Phi$, we have $\Phi(L(Y_{i,s})) = P_{(i,s)}$ for $i \in I$, $s = i-2$. Suppose that $\Phi(L(Y_{i,s})) = P_{(i,s)}$ for $i \in I$, $s \in 2 \ZZ_{\le 0} + i-2$. Then 
\begin{align*}
\Phi(L(Y_{i,s-2})) & = \Phi(L(Y_{i,s}))^{-1} (\Phi(L(Y_{i,s} Y_{i, s-2})) + \Phi(L(Y_{i-1, s-1})) \Phi(L(Y_{i+1, s-1})) ) \\
& = P_{(i,s)}^{-1}( P^{(n-i, \frac{i-s}{2}, 3)} + P_{(i-1, s-1)} P_{(i+1,s-1)} ) \\
& = P_{(i, s-2)}.
\end{align*}
The statement about weights follows from the definitions $\wt(Y_{i,s}) = \omega_i$ and Definition~\ref{def:weight of product of Plucker coordinaes}. 
\end{proof}

\begin{example}
In $\CC[\Gr(3,5,\sim)]$, the Pl\"{u}cker relation
\begin{align*}
& P_{124} P_{235} = P_{125} P_{234} + P_{123} P_{245} = P_{125} + P_{245}
\end{align*}
corresponds to $[Y_{1,-1}][Y_{1,-3}]=[Y_{1,-3} Y_{1,-1}]+[Y_{2,-2}]$ in the $T$-system of type $A_2$. 

In $\CC[\Gr(4,6,\sim)]$, the Pl\"{u}cker relation
\begin{align*}
& P_{1235} P_{2346} = P_{1236} P_{2345} + P_{1234} P_{2356} = P_{1236} + P_{2356}
\end{align*}
corresponds to $[Y_{1,-1}][Y_{1,-3}]=[Y_{1,-3} Y_{1,-1}]+[Y_{2,-2}]$ in the $T$-system of type $A_3$.
\end{example}

\begin{table}
\begin{equation*}
\hspace{3cm}
\begin{minipage}{0.5\textwidth}
\begin{tabular}{|c|c|}
\hline
modules & Pl\"{u}cker \\
\hline
 $Y_{1,-1}$ & $P_{1, 2, 4}$ \\ 
\hline
 $Y_{1,-3}$ & $P_{2, 3, 5}$ \\ 
\hline
  $Y_{1,-5}$ & $P_{3, 4, 6}$ \\ 
\hline
   $Y_{1,-7}$ & $P_{4, 5, 7}$ \\ 
\hline
  $Y_{1,-9}$ & $P_{5, 6, 8}$ \\ 
\hline
   $Y_{1,-11}$ & $P_{6, 7, 9}$ \\ 
\hline
    $Y_{1,-13}$ & $P_{7, 8, 10}$ \\
\hline
     $Y_{1,-15}$ & $P_{8, 9, 11}$ \\  
\hline
     $Y_{1,-17}$ & $P_{9, 10, 12}$ \\ 
\hline
      $Y_{1,-19}$ & $P_{10, 11, 13}$ \\
\hline
\end{tabular}
\end{minipage}
\hspace{-1cm}
\begin{minipage}{0.5\textwidth}
\begin{tabular}{|c|c|}
\hline
module & Pl\"{u}cker \\
\hline
 $Y_{2,0}$ & $P_{1, 3, 4}$ \\ 
\hline
 $Y_{2,-2}$ & $P_{2, 4, 5}$ \\ 
\hline
 $Y_{2,-4}$ & $P_{3, 5, 6}$ \\ 
\hline
 $Y_{2,-6}$ & $P_{4, 6, 7}$ \\ 
\hline
 $Y_{2,-8}$ & $P_{5, 7, 8}$ \\ 
\hline
 $Y_{2,-10}$ & $P_{6, 8, 9}$ \\ 
\hline
 $Y_{2,-12}$ & $P_{7, 9, 10}$ \\ 
\hline
 $Y_{2,-14}$ & $P_{8, 10, 11}$ \\ 
\hline
 $Y_{2,-16}$ & $P_{9, 11, 12}$ \\ 
\hline
 $Y_{2,-18}$ & $P_{10, 12, 13}$ \\ 
\hline
\end{tabular}
\end{minipage}
\end{equation*}
\caption{Correspondence between fundamental modules in $\mathcal{C}_{9}^{A_2}$ and fundamental Pl\"{u}cker coordinates in $\CC[\Gr(3,13)]$.}
\label{table:correspondence Plucker coordinates and fundamental modules sl3}
\end{table}

Now we make what turns out to be an important definition.
\begin{definition}\label{defn:smallgaps}
A tableau $T \in {\rm SSYT}(n,[m])$ has {\sl small gaps} if each of its columns has gap weight exactly~1. The tableau has {\sl nonlarge gaps} if each of its columns has gap weight at most~1. 
\end{definition} 

To reiterate, $T$ has small gaps if each of its columns is a fundamental tableau, i.e. if each of its columns has content $[i,i+n] \setminus \{r\}$ for $r \in (i,i+n)$. It has nonlarge gaps if each of its columns is either fundamental or trivial.

We lexicographically order the single-column tableaux with gap weight $\leq 1$, e.g. 
$$\begin{ytableau}
1 \\
2 \\
3 
\end{ytableau} < \begin{ytableau}
1 \\
2 \\
4 
\end{ytableau}
< 
\begin{ytableau}
1 \\
3 \\
4 
\end{ytableau}
< 
\begin{ytableau}
2 \\
3 \\
4 
\end{ytableau} <
\begin{ytableau}
2 \\
3 \\
5 
\end{ytableau}<
\begin{ytableau}
2 \\
4 \\
5 
\end{ytableau}
< 
\begin{ytableau}
3 \\
4 \\
5 
\end{ytableau} \in {\rm SSYT}(3,[5]).$$

\begin{lemma}\label{lem:nosorting}
If $S,T \in {\rm SSYT}(n,[m])$ have small gaps, then the columns of the monoid product $S \cup T$ are exactly the columns of $S$ union the columns of $T$ (as multisets), sorted in the above lexicographic order. The standard monomials $P_S,P_T \in \bbc[\Gr(n,m)]$ satisfy $P_SP_T = P_{S \cup T}$. Both statements remain true if we replace {\sl small gaps} with {\sl nonlarge gaps}. 
\end{lemma}

In particular, the set of small gaps tableaux is stable under the monoid product $\cup$, and the set of small gaps standard monomials $P_S$ is stable under multiplication. The $\CC$-linear span of these standard monomials is therefore a polynomial subalgebra of $\CC[\Gr(n,m)]$. The same statements hold with ``small'' replaced by ``nonlarge.'' 
 
\begin{proof}
All statements follow from noting that for tableaux with small gaps (resp., with nonlarge gaps),  when calculating the monoid product $T \cup S$ in lexicographically ordered fashion, there is no sorting along rows. 
\end{proof}

\begin{lemma}\label{lem:plucker and minimal affinization}
Every tableau $T \in {\rm SSYT}(n,[m])$ is $\sim$-equivalent to a unique $T' \in {\rm SSYT}(n,[m])$ with small gaps (for trivial $T$ we understand $T' = \mathds{1}$). If $T$ has gap weight $k$, then $T'$ has $k$ columns.
\end{lemma}

In other words, the monoid ${\rm SSYT}(n, [m],\sim)$ is free on the equivalence class of the fundamental tableaux.

\begin{proof}
That $T'$ must have $k$ columns follows since the weight map is a homomorphism (Lemma~\ref{lem:weightmaphomom}). We call an expression $T \sim T'$ for $T'$ with small gaps a {\sl factorization} of $T$ into fundamentals, and we need to prove the existence and uniqueness of this factorization. For the uniqueness, suppose that $T_1,T_2$ are small gaps tableaux and $T_1 \sim T_2$. From the definition of $\sim$, one concludes that there are trivial tableaux $A_1,A_2$ such that 
$A_2 \cup T_1 = A_1 \cup T_2 \in {\rm \SSYT}(n,[m])$. But by Lemma~\ref{lem:nosorting}, we can recover the columns of $A_2$ and the columns of $T_1$ uniquely from the product $A_2 \cup T_1$ (the columns of $A_2$ are those with gap weight zero, and the columns of $T_1$ are those with gap weight one). We conclude that $A_2 = A_1$ and $T_1 = T_2$ as claimed. 

It suffices to prove the existence of factorizations when $T$ has a single column, and we do this by induction on the gap weight.  If $T$ has gap weight zero, then $T$ is trivial and admits the empty factorization. Suppose inductively that $T$ has positive gap weight $k$, and list its entries as $j_1+1, j_1+2,\dots, j_1+c,j_1+d, j_{c+2}, \ldots, j_n$ where $d-c \geq 2$. Then 
$T \sim T_{1} \cup T_{2}$, where 
\begin{align}
& T_1 = T_{[j_1+1,j_1+c] \cup [j_1+c+2,j_1+n]}, \\
& T_2 = T_{[j_1+2,j_1+c+1] \cup \{j_1+d,j_{c+2},\dots,j_n\}}.
\end{align}
The tableau $T_1$ has gap weight one. The tableau $T_2$ has strictly smaller gap weight and can be factored by induction. The existence follows.  
\end{proof}

\begin{remark}\label{rmk:TtoTprime}
The $\sim$-equivalence class of a tableau $T$ bears two distinguished elements, the reduced tableau $T_{{\rm red}}$ and the small gaps tableau $T'$. The existence portion of the proof of Lemma~\ref{lem:plucker and minimal affinization} establishes that the ratio $\frac{T'}{T_{{\rm red}}}$ is a trivial tableau. Rather than considering the factorizations in Lemma~\ref{lem:plucker and minimal affinization} as equivalences $T \sim T'$, we could instead think of them as equalities $T = T'' \cup T'$, with $T'' \in K_0({\rm SSYT}(n,[m]))$ a uniquely defined fraction of trivial tableaux. Moreover, the denominator of $T''$ is controlled: it divides the ratio $\frac{T'}{T_{{\rm red}}}$. 
\end{remark}

\begin{example}\label{eg:factorourexample}
We have the equality
$$ \begin{ytableau}
1 & 2 \\
3 & 4 \\
5 & 6
\end{ytableau} = \begin{ytableau} 2 & 3\\3 & 4   \\4 & 5 \end{ytableau}^{-1} \cup \begin{ytableau} 1 & 2 & 2& 3\\3 & 3 & 4 &4  \\4 & 5 & 5 & 6 \end{ytableau},$$
which is of the form $T = T'' \cup T'$ as in the previous remark. The tableau $T = T_{{\rm red}}$ has gap weight $4$, and $T'$ is its factorization into 4 fundamental tableaux. \end{example}


Lemma~\ref{lem:plucker and minimal affinization} asserted that the monoid ${\rm SSYT}(n,[m],\sim)$ was free on the tableaux with gap weight one. Hernandez and Leclerc gave the following algebraic counterpart (using Theorem~\ref{thm:Hernandez-Leclerc quantum affine algebras and Grassmannians} and our Lemma~\ref{lem:fundamental modules and Plucker coordinates}). One can give a direct proof using standard monomials.  

\begin{proposition}[{\cite[Theorem 5.1]{HL16}}]\label{prop:basisforquotient}
The set $\{P_T\}_{\text{small gaps $T \in {\rm SSYT}(n,[m])$}}$ is a basis for $\CC[\Gr(n,m,\sim)]$. 
\end{proposition}

Now we make the main definitions of this section. Recall the isomorphism $\Phi \colon K_0(\mathcal{C}_\ell) \to \CC[\Gr(n, n+\ell+1,\sim)]$. By Proposition~\ref{prop:basisforquotient}, for any module $[L(M)] \in K_0(\mathcal{C}_\ell)$, one can therefore uniquely express 
\begin{align} \label{eq:expression of Phi(L(M)) in terms of tableaux}
\Phi([L(M)]) = \sum_{\text{small gaps }T} c_T P_T \in \bbc[\Gr(n,n+\ell+1,\sim)],
\end{align}
where $c_T \in \CC$. We denote by ${\rm Top}(\Phi([L(M)]))$ the tableau which appears in (\ref{eq:expression of Phi(L(M)) in terms of tableaux}) with highest weight. In Lemma \ref{lem:weight of module is the same as the weight of tableau corresponding to the module}, we will prove the existence ${\rm Top}(\Phi(L(M)))$ for every $L(M) \in K_0(\mathcal{C}_\ell)$. Assuming for the moment this lemma, we define a map 
\begin{align}
\widetilde{\Phi}: \mathcal{P}^+_{\ell,A_{n-1}} \to {\rm SSYT}(n, [n+\ell+1],\sim) \label{eq:tildephi} \hspace{.7cm}
M \mapsto {\rm Top}(\Phi(L(M))),
\end{align}
sending a dominant monomial to this tableau of highest weight. We denote $T_M = \widetilde{\Phi}(M)$.

Next, we define a map in the other direction, producing a dominant monomial from a tableau. For $T \in {\rm SSYT}(n,[m])$, let $T\sim \cup_{i=1}^k T_{P^{(a_i,b_i,2)}}$ be its unique factorization as a $\cup$-product of fundamental tableaux, as described in Lemma~\ref{lem:plucker and minimal affinization}. 
Define the map  
\begin{align}
\Psi: {\rm SSYT}(n, [n+\ell+1])  \to \mathcal{P}^+_{\ell,A_{n-1}}  \hspace{.7cm} T \mapsto \prod_{i=1}^k Y_{n-a_i, n-a_i-2b_i}, \label{eq:psi}
\end{align}
replacing $T$ by the corresponding product of fundamental monomials. We denote $M_T = \Psi(T)$. Clearly, $\Psi$ descends to a map on $\sim$-equivalence classes.

\begin{theorem} \label{thm: parametrization of simple modules by tableaux}
The map $\widetilde{\Phi} \colon \mathcal{P}^+_{\ell,A_{n-1}} \to {\rm SSYT}(n, [n+\ell+1],\sim)$ is an isomorphism of monoids, with inverse $\Psi$. 
\end{theorem}

In the remainder of this subsection, we explain that $\Psi$ is a homomorphism, that $\widetilde{\Phi}$ is well defined and is a homomorphism, and finally we prove the theorem. 

\begin{lemma} \label{lem:T and T prime equivalent implies that the corresponding heighest weight monomials are the same}
The map $\Psi$ is a monoid homomorphism ${\rm SSYT}(n, [n+\ell+1]) \to \mathcal{P}^+_{\ell,A_{n-1}}$. 
\end{lemma}

\begin{proof}
Since $\Psi(T)$ only depends on the equivalence class of $T$, it suffices to check that $\Psi(T)\Psi(S) = \Psi(S \cup T)$ when $S,T$ have small gaps. By Lemma~\ref{lem:nosorting}, the product $S \cup T$ also has small gaps, and moreover the columns of $S \cup T$ are obtained as the union of the columns of $S$ and $T$ respectively. By definition, to evaluate $\Psi$ on a tableaux with small gaps is to apply the bijection between fundamental tableaux and monomials, column by column. It follows that $\Psi(T)\Psi(S) = \Psi(S \cup T)$. 
\end{proof}

\begin{example}
Let $T = \begin{ytableau}
1 \\
3 \\
6 
\end{ytableau}$ and
$T' = \begin{ytableau}
2  \\
4  \\
5 
\end{ytableau}$.
Then 
\begin{align*}
\Psi(T \cup T') & = \Psi( \begin{ytableau}
1 & 2 \\
3 & 4 \\
5 & 6 
\end{ytableau}) 
= Y_{1,-3}Y_{2,0}Y_{1,-5}Y_{2,-2}, \text{ whereas }\\
\Psi(T)\Psi(T') & = (Y_{1,-5}Y_{1,-3}Y_{2,0})(Y_{2,-2}).
\end{align*}
\end{example}

Now we set out to show that $\widetilde{\Phi}$ is a well-defined monoid homomorphism. 
\begin{lemma} \label{lem:LMLMprime decomposition}
Let $L(M), L(M') \in \mathcal{C}_{\ell}^{A_{n-1}}$. Then
\begin{align} \label{eq:LMLMprime decomposition}
[L(M)] [L(M')] = [ L(M) \otimes L(M') ] = [L(M M')] + \sum_{\tilde{M}, \wt(\tilde{M}) < \wt(MM')} c_{\tilde{M}} [L(\tilde{M})],
\end{align}
for some $c_{\tilde{M}} \in \ZZ_{\ge 0}$. 
\end{lemma}

\begin{proof}
The equation (\ref{eq:LMLMprime decomposition}) is equivalent to 
\begin{align*}
\chi_q(L(M))\chi_q(L(M')) = \chi_q(L(MM')) + \sum_{\tilde{M}, \wt(\tilde{M}) < \wt(MM')} c_{\tilde{M}} \chi_q(L(\tilde{M})).
\end{align*}
The unique highest $l$-weight monomial in $\chi_q(L(M))$ is $M$ and the unique highest $l$-weight monomial in $\chi_q(L(M'))$ is $M'$. Therefore the unique highest $l$-weight in $\chi_q(L(M)) \otimes \chi_q(L(M'))$ is $MM'$. All other $l$-weight monomials in $\chi_q(L(M)) \otimes \chi_q(L(M'))$ are less than $MM'$. 
\end{proof}

The next lemma is not needed to prove Theorem~\ref{thm: parametrization of simple modules by tableaux}, but is used in Section \ref{sec: real modules}. 
 
\begin{lemma} \label{lem:qcharacter condition of real modules and prime modules}
A module $L(M)$ is real if and only if $\chi_q(L(M))\chi_q(L(M)) = \chi_q(L(M^2))$ and it is prime if and only if there are no $L(M'), L(M'') \ne \CC$ such that $\chi_q(L(M)) = \chi_q(L(M')) \chi_q(L(M''))$.
\end{lemma}

\begin{proof}
By definition and Lemma \ref{lem:LMLMprime decomposition}, a module $L(M)$ is real if and only if the right hand side of (\ref{eq:LMLMprime decomposition}) has only one term $[L(MM')]$. Therefore $L(M)$ is real if and only if 
\begin{align*}
\chi_q(L(M))\chi_q(L(M)) = \chi_q(L(M^2)).
\end{align*} 

By definition, a module $L(M)$ is prime if and only if there are no $L(M'), L(M'') \ne \CC$ such that $\chi_q(L(M)) = \chi_q(L(M')) \chi_q(L(M''))$.
\end{proof}

\begin{lemma}\label{lem:weight of module is the same as the weight of tableau corresponding to the module}
For a module $L(M) \in K_0(\mathcal{C}_\ell)$, ${\rm Top}(\Phi(L(M)))$ exists, and $\widetilde{\Phi}$ is a homomorphism. Moreover, $\wt(M) = \wt( {\rm Top}(\Phi(L(M))) )$.
\end{lemma}

\begin{proof}
By induction on the weight of $M$. The base case of fundamental modules is covered in Lemma \ref{lem:fundamental modules and Plucker coordinates}.

Now suppose we have a simple module corresponding to a dominant monomial of degree $\geq 2$. Choose any factorization of this monomial as $MM'$ with both factors nontrivial. By Lemma \ref{lem:LMLMprime decomposition} and the fact that $\Phi$ is an isomorphism, we have
\begin{align} 
\Phi(L(M)) \Phi(L(M')) = \Phi(L(M M')) + \sum_{\tilde{M}, \wt(\tilde{M}) < \wt(MM')} c_{\tilde{M}} \Phi( L(\tilde{M}) ),
\end{align}
for some $c_{\tilde{M}} \in \ZZ_{\ge 0}$. 

Then all of the terms $M,M',\tilde{M}$ have smaller weight than $MM'$, so by the inductive hypothesis ${\rm Top}(\Phi(L(M)))$ exists and its weight is $\wt(M)$ (and likewise for $M',\tilde{M}$). Moreover, each $\wt({\tilde{M}}) < \wt(M) + \wt(M')$. By comparing the highest weight terms in the left and right hand side, we conclude that ${\rm Top}(\Phi(L(MM')))$ exists, and in fact this top term coincides with ${\rm Top}(\Phi(L(M))) \cup {\rm Top}(\Phi(L(M')))$ (recalling that for small gaps tableaux, the product of Pl\"ucker coordinates corresponds to $\cup$). This establishes that $\widetilde{\Phi}$ is a well-defined homomorphism. The statement about weights follows: 
\begin{align*}
\wt(MM') & = \wt(M) + \wt(M') \\
& = \wt( {\rm Top}(\Phi(L(M))) )  + \wt( {\rm Top}(\Phi(L(M'))) ) = \wt( {\rm Top}(\Phi(L(MM'))) ). 
\end{align*}
\end{proof}

Finally, we prove the main theorem of this subsection. 
\begin{proof}[Proof of Theorem~\ref{thm: parametrization of simple modules by tableaux}]
By the proof of Lemma~\ref{lem:plucker and minimal affinization}, the monoid ${\rm SSYT}(n,[m],\sim)$ is free on the classes of the fundamental tableaux $T_{P^{(i,s)}}$. The monoid $\mathcal{P}_\ell^+$ is free by definition. We have defined monoid homomorphisms  $\widetilde{\Phi}$ and $\Psi$ in both directions, and we need to check that both composites $\Psi \widetilde{\Phi}$ and $\widetilde{\Phi} \Psi$ are the identity map. One can check such a statement on the free monoid generators. But then both statements follow from the definition of $\Psi$, and the equality $\widetilde{\Phi}(Y_{i,s}) = T_{P_{(i,s)}}$ established in Lemma~\ref{lem:fundamental modules and Plucker coordinates}. 
\end{proof}

\subsection{\texorpdfstring{The elements $\ch(T) \in \CC[\Gr(n,m,\sim)]$}{A basis for the coordinate ring}}
\begin{definition} \label{def:definition of ch(T)}
For a semistandard tableau $T \in {\rm SSYT}(n, [n+\ell+1],\sim)$ define $\ch(T) \in \CC[\Gr(n,n+\ell+1,\sim)]$ by $\ch(T) = \Phi([L(\Psi(T))])$ with $\Psi$ defined in \eqref{eq:psi}. 
\end{definition}

We use homogeneity to lift this definition from $\CC[\Gr(n,n+\ell+1,\sim)]$ to (a localization of) $\CC[\Gr(n,n+\ell+1)]$ in Definition~\ref{defn:chTforGr}.

\begin{example}
Tables \ref{table:correspondence Gr36 and Uqsl3hat} and \ref{table:correspondence Gr37 and Uqsl3hat} are examples of the correspondence between tableaux and modules. To save space, we write a tableau by listing its column sets. For example, $[1,2,4], [3,5,6]$ denotes the tableau $\begin{ytableau} 1 & 3 \\ 2 & 5 \\ 4 & 6 \end{ytableau}$. In Tables \ref{table:correspondence Gr36 and Uqsl3hat} and \ref{table:correspondence Gr37 and Uqsl3hat}, we write $Y_{i,s}$ as $i_s$. For example, $1_{-3} 1_{-1}$ denotes the module $L(Y_{1,-3} Y_{1,-1})$. 
\end{example}

\begin{table}
{\tiny
\begin{equation*}
\begin{minipage}{0.3\textwidth}
\begin{tabular}{|c|c|}
\hline
tableaux & modules \\
\hline
$[1,2,4]$ & $1_{-1}$ \\
\hline
$[1,2,5]$ & $1_{-3}1_{-1}$ \\  
\hline 
 $[1,2,6]$ & $1_{-5} 1_{-3} 1_{-1}$ \\  
\hline
 $[1,2,7]$ & $1_{-7} 1_{-5} 1_{-3} 1_{-1}$ \\  
\hline
$[1,3,4]$ & $2_{0}$ \\  
\hline 
$[1,3,5]$ & $2_{0} 1_{-3}$ \\  
 \hline
 $[1,3,6]$ & $2_{0} 1_{-3} 1_{-5}$ \\  
\hline
 $[1,3,7]$ & $2_{0} 1_{-3} 1_{-5} 1_{-7}$ \\  
\hline
 $[1,4,5]$ & $2_{0} 2_{-2}$ \\  
\hline
$[1,4,6]$ & $2_{0} 2_{-2} 1_{-5}$ \\  
\hline
 $[1,4,7]$ & $2_{0} 2_{-2} 1_{-5} 1_{-7}$ \\  
\hline
 $[1,5,6]$ & $2_{0} 2_{-2} 2_{-4}$ \\  
\hline
 $[1,5,7]$ & $2_{0} 2_{-2} 2_{-4} 1_{-7}$ \\  
\hline
 $[1,6,7]$ & $2_{0} 2_{-2} 2_{-4} 2_{-6}$ \\  
\hline
$[2,3,5]$ & $1_{-3}$ \\  
\hline
\end{tabular}
\end{minipage}
\hspace{-0.5cm}
\begin{minipage}{0.3\textwidth}
\begin{tabular}{|c|c|}
\hline
tableaux & modules \\
\hline 
 $[2,3,6]$ & $1_{-3} 1_{-5}$ \\ 
 \hline 
 $[2,3,7]$ & $1_{-3} 1_{-5} 1_{-7}$ \\   
\hline
$[2,4,5]$ & $2_{-2}$ \\  
\hline
 $[2,4,6]$ & $2_{-2} 1_{-5}$ \\  
\hline
 $[2,4,7]$ & $2_{-2} 1_{-5} 1_{-7}$ \\  
\hline
 $[2,5,6]$ & $2_{-2} 2_{-4}$ \\  
\hline
 $[2,5,7]$ & $2_{-2} 2_{-4} 1_{-7}$ \\  
 \hline
 $[2,6,7]$ & $2_{-2} 2_{-4} 2_{-6}$ \\
\hline
$[3,4,6]$ & $1_{-5}$ \\  
\hline 
 $[3,4,7]$ & $1_{-5} 1_{-7}$ \\  
 \hline
$[3,5,6]$ & $2_{-4}$ \\  
\hline
 $[3,5,7]$ & $2_{-4} 1_{-7}$ \\  
\hline
 $[3,6,7]$ & $2_{-4} 2_{-6}$ \\  
\hline 
$[4,5,7]$ & $1_{-7}$ \\ 
\hline 
$[4,6,7]$ & $2_{-6}$ \\  
\hline
\end{tabular}
\end{minipage}
\hspace{-1cm}
\begin{minipage}{0.3\textwidth}
\begin{tabular}{|c|c|}
\hline
tableaux & modules  \\ 
\hline 
 $[1,2,4],[3,5,6]$ & $1_{-1} 2_{-4}$ \\  
\hline
 $[1,2,4],[3,5,7]$ & $1_{-1} 2_{-4} 1_{-7}$ \\  
\hline
 $[1,2,4],[3,6,7]$ & $1_{-1} 2_{-4} 2_{-6}$ \\  
 \hline
 $[1,2,5],[3,6,7]$ & $1_{-1} 1_{-3} 2_{-4} 2_{-6}$ \\  
\hline
 $[1,2,5],[4,6,7]$ & $1_{-1} 1_{-3} 2_{-6}$ \\  
\hline
 $[1,3,5],[2,4,6]$ & $2_{0} 1_{-3} 2_{-2} 1_{-5}$ \\  
\hline
 $[1,3,5],[2,4,7]$ & $2_{0} 1_{-3} 2_{-2} 1_{-5} 1_{-7}$ \\  
\hline
 $[1,3,5],[4,6,7]$ & $2_{0} 1_{-3} 2_{-6}$ \\
\hline
 $[1,3,6],[2,4,7]$ & $2_{0} 1_{-3} 2_{-2} 1_{-5}^2 1_{-7}$ \\  
\hline
$[1,3,6],[2,5,7]$ & $2_{0} 1_{-3} 2_{-2} 1_{-5} 2_{-4} 1_{-7}$ \\
\hline
 $[1,4,6],[2,5,7]$ & $2_{0} 2_{-2}^2 1_{-5} 2_{-4} 1_{-7}$ \\    
\hline
 $[1,4,6],[3,5,7]$ & $2_{0} 2_{-2} 1_{-5} 2_{-4} 1_{-7}$ \\  
 \hline
 $[2,3,5],[4,6,7]$ & $1_{-3} 2_{-6}$ \\  
\hline
 $[2,4,6],[3,5,7]$ & $2_{-2} 1_{-5} 2_{-4} 1_{-7}$ \\  
\hline
\end{tabular}
\end{minipage}
\end{equation*}
}
\caption{Correspondence between tableaux for ${\rm SSYT}(3,[7])$ and modules in $\mathcal{C}_{3}^{A_2}$.}
\label{table:correspondence Gr37 and Uqsl3hat}
\end{table}

Using the isomorphism $\widetilde{\Phi}$ and the results of Kashiwara, Kim, Oh, and Park \cite{KKOP} and Qin \cite{Qin}, we have the following. 
\begin{theorem} \label{thm: cluster monomials are tableaux}
Every cluster monomial (resp. cluster variable) in $\bbc[\Gr(n,n+\ell+1,\sim)]$ is of the form $\ch(T)$ for some real tableau (resp. prime real tableau) $T \in {\rm SSYT}(n, [n+\ell+1])$.
\end{theorem}

\begin{proof}
By \cite[Theorem 1.2.1]{Qin} and \cite[Theorem 6.10]{KKOP}, any cluster monomial (resp. cluster variable) in $\mathcal{R}_{\ell}^{\mathfrak{g}}$ corresponds to the Grothendieck class of a real (resp. real prime) simple object $L(M) \in \mathcal{C}_{\ell}^{\mathfrak{g}}$. Then $\ch(\widetilde{\Phi}(M)) = \Phi([L(M)])$ is a cluster monomial (resp. cluster variable) as claimed. 
\end{proof}

\begin{proposition}
Let $T, T' \in {\rm SSYT}(n, [n+\ell+1])$. Then $\ch(T \cup T') = \ch(T) \ch(T')$ if and only if $[L(\Psi(T\cup T'))] = [L(\Psi(T)) \otimes L(\Psi(T'))]$.
\end{proposition}

\begin{proof}
This follows by comparing 
\begin{align*}
\ch(T\cup T') &= \Phi([L(\Psi(T \cup T'))])  \text{ versus } \\
\ch(T) \ch(T')&= \Phi([L(\Psi(T))]) \Phi([L(\Psi(T'))]) = \Phi( [L(\Psi(T)) \otimes L(\Psi(T'))]),
\end{align*}
since $\Phi$ is an isomorphism. 
\end{proof}

\subsection{Comparison of partial orders}\label{subsecn:comparison}
We end this section by comparing the partial order on tableaux with the one on dominant monomials. 

Following \cite[(3.7)]{FZ07} and \cite[Section 4.5.1]{HL16}, denote
\begin{align*}
\widehat{y}_{i,r} = \prod_{(i,r) \to (j,s)} z_{j,s} \prod_{(j,s) \to (i,r)} z_{j,s}^{-1}, \quad i \in [n-1], \quad r \in [0, \ell-1],
\end{align*}
the standard Fomin-Zelevinsky $\widehat{y}$-variables with respect to the initial seed in Figure~\ref{fig:initial cluster for Uqslhat5}. 

\begin{lemma}[{\cite[Lemma 4.15]{HL16}}] \label{lem:yhat is Ainverse}
For $i \in [n-1]$, $r \in [0, \ell-1]$, $\widehat{y}_{i,r} = A_{i,r-1}^{-1}$, where $A_{i,s}$ is defined in (\ref{eq:Aia}).
\end{lemma}

We remark that the partial order on monomials \eqref{partial order of monomials} therefore matches the partial order on Laurent monomials in initial cluster variables, as defined for arbitrary cluster algebras whose extended exchange matrices have full rank by Qin \cite{Qin} (see also \cite{Cas}). We believe that \eqref{partial order of monomials} was in fact an inspiration for Qin's definition of this partial order.

\begin{proposition}\label{prop:posetsaresame}
Let $T,T' \in {\rm SSYT}(n,[n+\ell+1])$ be tableaux with the same content. Then $T \leq T'$ (in the sense of Definition~\ref{defn:dominanceontableaux}) if and only if $M_T \leq M_{T'} \in \mathcal{P}^+_{\ell}$.
\end{proposition}

\begin{proof}
We set $m = n+\ell+1$. We work in the group $K_0({\rm SSYT}(n,[m]))$ of fractions of tableaux. By definition $M_T \leq M_{T'}$ means that $M_{T'}M_T^{-1} \in \mathcal{Q}^+$, and by Lemma~\ref{lem:yhat is Ainverse}, this is equivalent to requiring that $M_T$ equals  $M_{T'}$ times a monomial in the $\widehat{y}$'s with respect to our initial seed in Figure~\ref{fig:initial cluster for Uqslhat5}. So we need to show that $T \leq T' \in {\rm SSYT}(n,[m])$ implies that the equivalence class $T \in  {\rm SSYT}(n,[m],\sim)$ equals $T' \cdot A \in {\rm SSYT}(n,[m],\sim)$, where $A$ is the a monomial in the $\widehat{y}$'s with respect to our initial seed in Figure~\ref{fig:initial cluster for a quotient of Gr(5,10)}. 

To keep track of homogeneity, we prefer to work with $\widehat{y}$'s for $\bbc[\Gr(n,m)]$ rather than in the quotient $\bbc[\Gr(n,m,\sim)]$. By direct inspection of Figure~\ref{fig:initial cluster for a quotient of Gr(5,10)}, each $\widehat{y}$ is a fraction of two tableaux $\frac{N}{D}$, where $N \leq D \in {\rm \SSYT}(n,[m])$. In fact, $D$ is obtained from $N$ by swapping a single pair of entries of the form $x,x+1$, with these entries occupying adjacent rows. For example, the $\widehat{y}$ in the (1,0) entry of the grid amounts to a swap of entries $x,x+1 = 5,6$ in rows 4 and 5 (recall our unusual conventions on grid coordinates from Section~\ref{subsec:Grothendieck ring and Grassmannian}). Moving down the diagonal, the $\widehat{y}$ in entry (2,1),(3,2), and (4,3) corresponding to swapping 
$x,x+1 = 5,6$ in rows $\{3,4\},\{2,3\},\{1,2\}$, respectively. The $\widehat{y}$ for the $(1,1)$ entry of the grid corresponds to swapping entries $4,5$ in rows 3 and 4. Note that swapping entries swapping entries $3,4$ in rows $4,5$ does not appear as a $\widehat{y}$ of any variable in the grid, but this is because $3$ cannot appear in row $4$ of a semistandard tableaux. More subtly, swapping $4,5$ in rows $4,5$ also does not appear as a $\widehat{y}$ of any variable in the grid. But $4,5$ only appear in rows $4,5$ as part of a frozen column $1,2,3,4,5$, and the entries are never able to participate in a swap. With these two subtleties in mind, one can verify that all swaps of consecutive entries in adjacent rows that could occur appear as a $\widehat{y}$ of some entry in the grid. 

First we show that $M_T \leq M_{T'}$ (and $T,T'$ have the same content) implies that $T \leq T'$ in dominance order. Indeed, using the isomorphism of monoids we see that the equation $T = T' \cdot A$ holds in $K_0(n,[m],\sim)$, and therefore a similar equation $BT = T' \cdot A$ holds in $K_0(n,[m])$, where $B$ is a Laurent monomial in trivial tableaux. But $T$ and $T'$ have the same content, and multiplying by $A$ does not change the content (since the $\widehat{y}$'s have the same content in the numerator and denominator). Thus $B =1$ and $T = T' \cdot A$ holds in $K_0(n,[m])$. But this implies that $T \leq T'$ (since $T$ can be obtained from $T'$ by transpositions, each of which lowers in the partial order on tableaux). 

For the other direction, it suffices to assume that $T$ and $T'$ are related by swapping a single pair of entries $x,y$ as in the proof of Lemma~\ref{lem:dominanceimpliesweight}. By two paragraphs previous, repeatedly multiplying by $\widehat{y}$'s corresponds to repeatedly switching adjacent entries in adjacent rows. It remains to prove that, perhaps upon replacing $T$ and $T'$ by $\sim$-equivalent tableaux, we can swap the pair of entries $x,y$ in which $T,T'$ differ by repeatedly swapping consecutive entries in adjacent rows (in the manner described above).  

This relies on two ideas.  First, suppose that we want to swap nonadjacent entries $x,x+2$ in adjacent rows, with the $x$ in row $i$. Multiplying by the frozen variable with $x+1$ in row $i$, we can subsequently multiply by a $\widehat{y}$ that simulates swapping $x+1,x+2$ in rows $i,i+1$, and then multiply by the $\widehat{y}$ that simulates swapping $x,x+1$ in rows $i,i+1$. The result will be divisible again by the same frozen, but with $x,x+2$ swapped. Generalizing this idea in the obvious way, we can swap $x,x+a$ in adjacent rows for any $a \geq 1$. This allows us to relax the requirement that we only perform swaps of adjacent entries. Second, suppose for example that $T'$ and $T$ differ by a pair of entries $x,x+1$ in nonadjacent rows $i < j$ with $j-i \geq 2$. Then one can multiply $T'$ by the $\widehat{y}$ corresponding to switching $x,x+1$ in rows $i,i+1$, and also the $\widehat{y}$ corresponding to switching them in $i+1,i+2$, and so on until $j-1,j$. This product ``telescopes'' and the resulting fraction of tableaux equals $T$ inside $K_0({\rm SSYT}(n,[m]))$. Thus, one can relax the requirement that we only perform swaps in adjacent rows. Combining these two ideas, we can perform an arbitrary transposition, and the result follows. 
\end{proof}

\section{Mutation of tableaux and modules} \label{sec:mutations description}
By Theorem \ref{thm: cluster monomials are tableaux}, every cluster variable in $\bbc[\Gr(n,m,\sim)]$ is of the form $\ch(T)$ for some (real, prime) $T \in {\rm SSYT}(n,[m])$. Starting from the initial seed of $\bbc[\Gr(n,m,\sim)]$, each time we perform a mutation at the cluster variable $\ch(T_k)$, we obtain a cluster variable $\ch(T'_k)$ defined recursively by 
\begin{align*}
\ch(T'_k)\ch(T_k) = \prod_{i \to k} \ch(T_i) + \prod_{k \to i} \ch(T_i),
\end{align*}
with $\ch(T_i)$ the cluster variable at the vertex $i$. On the other hand, Theorem \ref{thm:Hernandez-Leclerc quantum affine algebras and Grassmannians} and Lemma \ref{lem:LMLMprime decomposition} imply that
\begin{align} \label{eq:decomposition of ch(T)ch(T')}
\ch(T_k) \ch(T'_k) = \ch(T_k \cup T'_k) + \sum_{T''} c_{T''} \ch(T'')
\end{align}
for some $T'' \in {\rm SSYT}(n,[m])$, $\wt(T'')<\wt(T_k \cup T'_k)$, $c_{T''} \in \ZZ_{\ge 0}$. Therefore one of the two tableaux $\cup_{i \to k} T_i$ or $\cup_{k \to i} T_i$ has strictly greater weight than the other, and moreover this leading terms agrees with 
$T_k \cup T'_k$ in ${\rm SSYT}(n,[m],\sim)$. Denoting by $\max\{\cup_{i \to k} T_i, \cup_{k \to i} T_i \}$ this higher weight tableau, it follows that the ``new'' cluster variable $\ch(T_k')$ can be computed in the monoid ${\rm SSYT}(n,[m])$:
\begin{align}
T'_k = T^{-1}_k \max\{\cup_{i \to k} T_i, \cup_{k \to i} T_i \}.\label{eq:howtomutate}
\end{align}
In principle, the above argument only specifies what $T'_k$ should be up to $\sim$. However, working in $\CC[\Gr(n,m)]$ rather than in the quotient, the $\ZZ^m$-degree of $\ch(T'_k)$ is determined by the $\ZZ^m$-degrees of the current cluster variables, and \eqref{eq:howtomutate} is the unique choice of $T'_k$ in its equivalence class so that the content of $T'_k$ matches the $\ZZ^m$-degree of the corresponding cluster variable in $\CC[\Gr(n,m)]$.

By the same reasoning, the mutation rule for a module is as follows. 
When we mutate at the vertex $k$ with a cluster variable $\chi_q(M_k)$, the $q$-character $\chi_q(M'_k)$ is given by 
\begin{align*}
\chi_q(M_k)\chi_q(M'_k) &= \prod_{k \to i} \chi_q(M_i) + \prod_{i \to k} \chi_q(M_i), \text{ with } \\
M'_k &= M_k^{-1} \max\{ \prod_{k \to i} M_i, \prod_{i \to k} M_i \}. 
\end{align*}
Again, $\chi_q(M_i)$ is the $q$-character of the module at vertex $i$, and $\max$ denotes taking the higher weight monomial in $\mathcal{P}^+$.

\begin{example} \label{example:exchange relations}
The following are some examples of mutations in $\bbc[\Gr(3,8)]$:
\begin{align*}
& \ch(\begin{ytableau}
1  \\
3 \\
4 
\end{ytableau}) \ch(\begin{ytableau}
2  \\
3 \\
5 
\end{ytableau}) = \ch(\begin{ytableau}
1  \\
3 \\
5 
\end{ytableau}) \ch(\begin{ytableau}
2  \\
3 \\
4 
\end{ytableau}) + \ch(\begin{ytableau}
1  \\
2 \\
3 
\end{ytableau}) \ch(\begin{ytableau}
3  \\
4 \\
5 
\end{ytableau}), \\
& \ch(\begin{ytableau}
2  \\
3 \\
8 
\end{ytableau}) \ch(\begin{ytableau}
1 & 3 & 4  \\
2 & 5 & 6 \\
4 & 7 & 8 
\end{ytableau}) = \ch(\begin{ytableau}
1  \\
2 \\
8 
\end{ytableau}) \ch(\begin{ytableau}
3 & 4  \\
5 & 6 \\
7 & 8
\end{ytableau}) \ch(\begin{ytableau}
2  \\
3 \\
4 
\end{ytableau}) + \ch(\begin{ytableau}
3  \\
4 \\
8 
\end{ytableau}) \ch(\begin{ytableau}
2 & 4  \\
5 & 6 \\
7 & 8 
\end{ytableau}) \ch(\begin{ytableau}
1  \\
2 \\
3 
\end{ytableau}).
\end{align*}
In both cases, we ordered the two terms on the right hand side so that the heigher weight monomial came first. 

The corresponding mutations in $\mathcal{R}_{4}^{A_{2}}$ are:
\begin{align*}
& \chi_q(Y_{2,0}) \chi_q(Y_{1,-3}) = \chi_q(Y_{2,0}Y_{1,-3})\chi_q(\CC) + \chi_q(\CC)\chi_q(\CC) = \chi_q(Y_{2,0}Y_{1,-3}) + 1, \\
& \chi_q(Y_{1,-3}  Y_{1,-5}  Y_{1,-7}  Y_{1,-9}) \chi_q(Y_{1,-1}  Y_{2,-4}  Y_{1,-7}  Y_{2,-6}  Y_{1,-9}) \\
& = \chi_q(Y_{1,-1}  Y_{1,-3}  Y_{1,-5}  Y_{1,-7}  Y_{1,-9}) \chi_q(Y_{2,-4}  Y_{1,-7}  Y_{2,-6}  Y_{1,-9}) \chi_q(\mathbb{C}) \\
& \qquad + \chi_q(Y_{1,-5}  Y_{1,-7}  Y_{1,-9}) \chi_q(Y_{2,-2}  Y_{2,-4}  Y_{1,-7}  Y_{2,-6}  Y_{1,-9}) \chi_q(\CC). 
\end{align*}
\end{example}

\begin{remark}
As far as we know, the previously known method for performing a mutation in $\bbc[\Gr(n,m)]$ was ``guess and check:'' one would make an educated guess what the neighboring cluster variable should be based on its $\ZZ^m$-degree (e.g., by exhausting over all webs of the right grading, cf.~Section~\ref
{sec:Fomin Pylyavskyy conjecture}), and then verify algebraically that such a guess solved the required exchange relation. Our observation is that mutations are in fact controlled by the monoid ${\rm SSYT}(n,[m])$. For example, the exchange graph of the cluster algebra can be computed (to any desired mutation depth) purely in this monoid. 
\end{remark} 

\begin{remark}
By the results of Section~4, together with the recurrence for $g$-vectors implied by the sign-coherence theorem (cf. e.g. \cite{Mul}), we have the following interpretation of green versus red vertices with respect to our initial seed. Consider a cluster $\ch(T_1),\dots,\ch(T_N)$ for the Grassmannian, and $\ch(T_k)$ a mutable variable in this cluster. Let $T^+ = \prod_{i \to k} T_i$ and $T^-= \prod_{k \to i}T_i$ be the $\cup$-product of all tableaux pointing inwards (resp. outwards) at $T_k$ in this cluster. Then $T_k$ is either green or red in its cluster according to whether $T^+$ or $T^-$ is larger in the partial order on tableaux. For example, all initial variables are green, and the inwards monomial has higher weight than the outwards monomial for every mutable vertex in Figure~\ref{fig:initial cluster for a quotient of Gr(5,10)}. Corollary \ref{corollary:correspondence between g vectors and real tableaux} gives a tableau-theoretic interpretation of $c$-vectors, and thus an alternative rule for determining green versus red-ness. 
\end{remark}

\section{\texorpdfstring{A $q$-character formula and formula for $\ch(T)$}{Character formulas}} \label{sec:qcharacter formulas}
\subsection{\texorpdfstring{Representations of $p$-adic groups}{Representations of p-adic groups}}
We recall certain results about representations of $p$-adic groups, \cites{BZ,Zel,LM}. 

Let $F$ be a non-archimedean local field with a normalized
absolute value $|\cdot|$. For any reductive group $G$ over $F$, let $\mathcal{C}(G)$ be the category of complex, smooth representations of $G(F)$ of finite length and let $\text{Irr}G$ be the set of irreducible objects of $\mathcal{C}$ up to equivalence. Let $G_n=GL_n(F)$, $n =0, 1, 2, \ldots$. For $\pi_i \in \mathcal{C}(G_{n_i})$, $i=1,2$, denote by $\pi_1 \times \pi_2 \in \mathcal{C}(G_{n_1+n_2})$ the representation which is parabolically induced from $\pi_1 \otimes \pi_2$. Denote by $\mathcal{R}^G_n$ (resp. $\mathcal{R}^G$) the Grothendieck ring of $\mathcal{C}(G_n)$ (resp. $\mathcal{C}=\oplus_{n \ge 0} \mathcal{C}(G_n)$). Then $\mathcal{R}^G = \oplus_{n \ge 0} \mathcal{R}^G_n$ is a commutative graded ring under $\times$.  Denote $\text{Irr} = \cup_{\ge 0} {\rm Irr}G_n$ and denote by ${\rm Irr}_c \subset {\rm Irr}$ the subset of supercuspidal representations of $G_n$, $n > 0$. For $\pi \in \mathcal{C}(G_n)$, we denote $\deg(\pi) = n$. 

For $\rho \in {\rm Irr}_c$, we denote $\overrightarrow{\rho} = \rho \nu$, $\overleftarrow{\rho}=\rho \nu^{-1}$, where $\nu$ is the character $\nu(g) = |\det(g)|$. A {\sl segment} is a finite nonempty subset of ${\rm Irr}_c$ of the form $\Delta = \{\rho_1, \ldots, \rho_k\}$, where $\rho_{i+1} = \overrightarrow{\rho_i}$, $i \in [k-1]$. We denote $b(\Delta)=\rho_1$, $e(\Delta)=\rho_k$, and $\deg(\Delta) = \sum_{i=1}^k \rho_i$ ($\deg(\Delta)$ is called the degree of $\Delta$). Usually we write $\Delta$ as $[b(\Delta), e(\Delta)]$. For a segment $\Delta =\{\rho_1, \ldots, \rho_k\}$, we denote ${\rm Z}(\Delta) = {\rm soc}(\rho_1 \times \cdots \times \rho_k) \in {\rm Irr} G_{\deg \Delta}$, where ${\rm soc}(\pi)$ denotes the socle of $\pi$, i.e., the largest semisimple subrepresentation of $\pi$. We use the convention that ${\rm Z}(\emptyset)=1$.

A {\sl multi-segment} is a formal finite sum ${\bf m} = \sum_{i=1}^k \Delta_i$ of segments. Let $\mathcal{M}$ denote the resulting commutative monoid of multi-segments. Denote $\deg {\bf m} = \sum_{i=1}^k \deg(\Delta_i)$. Denote
\begin{align*}
\overleftarrow{\Delta} = \{\overleftarrow{\rho_1}, \ldots, \overleftarrow{\rho_k}\}, \quad \overrightarrow{\Delta} = \{\overrightarrow{\rho_1}, \ldots, \overrightarrow{\rho_k}\}.
\end{align*}
For two segments $\Delta_1, \Delta_2$, say that $\Delta_1$ {\sl precedes} $\Delta_2$ (denoted by $\Delta_1 \prec \Delta_2$) if
\begin{align} \label{eq:Delta1 precede Delta2}
b(\Delta_1) \not\in \Delta_2, \ b(\overleftarrow{\Delta_2}) \in \Delta_1, \ e(\Delta_2) \not\in \Delta_1.
\end{align}

Let ${\bf m} = \sum_{i=1}^k \Delta_i$ be a multi-segment. We may assume that $\Delta_i \not\prec \Delta_j$ for all $i<j$ (any multi-segment can be ordered in this way). We denote 
\begin{align*}
\zeta({\bf m}) = {\rm Z}(\Delta_1) \times \cdots \times {\rm Z}(\Delta_k) \in \mathcal{C}(G_{\deg {\bf m}}),
\end{align*}
and ${\rm Z}({\bf m}) = {\rm soc}(\zeta({\bf m})) \in {\rm Irr}G_{\deg {\bf m}}$. The map ${\bf m} \in \mathcal{M} \mapsto {\rm Z}({\bf m})$ defines a bijection between $\mathcal{M}$ and ${\rm Irr}$, see \cites{BZ,Zel}.

From now on we fix $\rho \in {\rm Irr}_c$ and write a segment $\{\rho \nu^i: i \in [a,b] \}$ as $[a,b]$, $a, b \in \ZZ$, $a \le b$. 

Let $H_N$ ($N \in \ZZ_{\ge 1}$) be the Iwahori-Hecke algebra of $GL_N(F)$ and let $I_N$ be the standard Iwahori-subgroup of $GL_N(F)$. Then each finite-dimensional representation of $H_N$ can be identified with the subspace of $I_N$-fixed vectors in a smooth finite-length representation of $GL_N(F)$ \cite{LNT}. The category of finite-dimensional representations of $H_N$ is equivalent to the category of smooth finite-length representations of $GL_N(F)$ which are generated by the vectors which are fixed under the Iwahori subgroup.  

Chari and Pressley \cite{CP96b} proved that when $N \le n$, there is an equivalence between the category of finite-dimensional representation of $H_N$ and the subcategory of finite-dimensional representations of $U_q(\widehat{\mathfrak{sl}_n})$ consisting of those representations whose irreducible components under $U_q(\mathfrak{sl}_n)$ all occur in the $N$-fold tensor product of the natural representation of $U_q(\mathfrak{sl}_n)$.

By the theorem in \cite[Section 7.6]{CP96b}, under the equivalence of categories, multi-segments and dominant monomials are identified via the following correspondence between segments and fundamental monomials: 
\begin{equation}\label{eq:multisegtomonom}
[a, b] \mapsto Y_{b-a+1, a+b-1},  \hspace{1cm} Y_{i,s} \mapsto [\frac{s-i+2}{2}, \frac{s+i}{2}].
\end{equation}
We denote this correspondence by ${\bf m} \mapsto M_{\bf m}$ and $M \mapsto {\bf m}_M$ accordingly. 
We interpret an expression $M_{[a,a-1]}$ as encoding the trivial monomial 
$1 \in \mathcal{P}_+$ (noting that $[a,a-1]$ is not a segment). Likewise, we interpret $M_{[a,b]}$ with $b < a-1$ as encoding zero. Thus for any k-tuples $(\mu, \lambda) \in \ZZ^k \times \ZZ^k$, we can define a multi-set: 
\begin{align*}
{\rm Fund}_M(\mu, \lambda) = \{M_{[\mu_i,\lambda_i]}: i \in [k] \}.
\end{align*}

\begin{example}
Let $M=Y_{2,0} Y_{2,-4} Y_{4,-4} Y_{2,-8}$. Then
\begin{align*}
{\bf m}_M = [0, 1] + [-3, 0] + [-2, -1] + [-4, -3].
\end{align*}
\end{example}

\subsection{\texorpdfstring{A $q$-character formula and formula for $\ch(T)$}{A q-character formula and formula for ch(T)}}
We denote by $S_k$ denote the symmetric group on $k$ symbols, denote by $\ell(w)$ the Coxeter length of $w \in S_k$, and denote by $w_0 \in S_k$ the longest permutation.

For $\lambda = (\lambda_1, \ldots, \lambda_k) \in \ZZ^k$, denote by $S_{\lambda}$ the subgroup of $S_k$ consisting of elements $\sigma$ such that $\lambda_{\sigma(i)} = \lambda_i$. For $\mu=(\mu_1, \ldots, \mu_k)$, $\lambda = (\lambda_1, \ldots, \lambda_k) \in \ZZ^k$, we denote ${\bf m}_{\mu, \lambda} = \sum_{i=1}^k [\mu_i, \lambda_i]$. 

The action of $S_k \times S_k$ on $\ZZ^k \times \ZZ^k$ by permutation of coordinates determines an action on formal sums $(w,v) \cdot {\bf m}_{\mu, \lambda} = {\bf m}_{w \cdot \mu, v \cdot \lambda}$, for $(w,v) \in S_k \times S_k$. Using ${\bf m}_{w \cdot \mu, v \cdot \lambda} = {\bf m}_{v^{-1}w \cdot \mu, \lambda}$, it is clear that any formal sum as above can be written in the form
\begin{equation}\label{eq:lammu}
{\bf m}_{w \cdot \mu, \lambda} \text{ for } \lambda,\mu \in \ZZ^k \text{ weakly decreasing and } w \in S_k. 
\end{equation} 
Moreover, the formal sums ${\bf m}_{w \cdot \mu, \lambda}$ and ${\bf m}_{w' \cdot \mu, \lambda}$ are equal (possibly with terms rearranged) if and only if $w'$ is in the double coset $S_\lambda w S_\mu$. Since such a double coset is finite, one knows that it contains a unique permutation of maximal length (cf.~\cite[Sections 2.4, 2.5]{Bou}, \cite[Proposition 2.3]{Kob}, and \cite[Proposition 2.7]{BKPST}). For a multisegment ${\bf m}$ with $k$ terms, denote by $\mu_{\bf m},\lambda_{\bf m} \in \ZZ^k$ and $w_{\bf m} \in S_k$ the permutation of maximal length such that \eqref{eq:lammu} holds. We can equivalently index these quantities as $\mu_{M},\lambda_{M},w_M$ where $M = M_{\bf m}$ is the corresponding dominant monomial, or as $\mu_{T},\lambda_{T},w_T$, where $T = \widetilde{\Phi}(M)$ is the corresponding equivalence class of tableaux. 

For an object $M$ in $\mathcal{C}^G$, denote by $[M]$ the class of $M$ in $\mathcal{R}^G$. The following result is originally due to Arakawa-Suzuki \cites{AS, BaCi, Hen, LM}.
\begin{theorem} [{ \cite{AS}, \cite[Section 10.1]{LM} }] \label{thm:multi-segment character formula}
Let ${\bf m}$ be a multi-segment with $\lambda_{\bf m}, \mu_{\bf m} \in \ZZ^k, w_{\bf m} \in S_k$ as just defined. Then
\begin{align*}
[{\rm Z}({\bf m})] = \sum_{u \in S_k} (-1)^{\ell(u w_{{\bf m}})} p_{uw_0, w_{{\bf m}}w_0}(1) [ \zeta( {\bf m}_{u \mu, \lambda} ) ],
\end{align*} 
where $p_{y,y'}(1)$ ($y,y' \in S_k$) is the value at $1$ of the Kazhdan-Lusztig polynomial $p_{y,y'}(t)$. 
\end{theorem}

\begin{example}\label{eg:cosets}
Consider the monomial $M = Y_{1,-5} Y_{1,-3} Y_{2,-2} Y_{2,0} \in \mathcal{P}^+$ with ${\bf m}_M=[0,1]+[-1,0]+[-1,-1]+[-2,-2]$. Then 
$\mu_M=(0,-1,-1,-2)$, $\lambda_M=(1,0,-1,-2)$ are the left and right endpoints in sorted order. The subgroups $S_\mu,S_{\lambda} \subset S_4$ are $\{1,s_2\},\{1\}$ respectively. The identity permutation $u = \text{id}$ satisfies ${\bf m}_M = {\bf m}_{u \mu,\lambda}$, and the maximal length element that satisfies this is $u = s_2$.\end{example}

We have the following translation of Theorem \ref{thm:multi-segment character formula} to the language of $q$-characters. We interpret $\chi_q(L(M_{[a,a-1]})) = 1$ and $\chi_q(L(M_{[a,b]})) = 0$ when $b < a-1$, see above.   
\begin{theorem}\label{cor:qcharacter formula}
Let $M \in \mathcal{P^+}$ be a monomial of degree $k$, with $\lambda_M,\mu_M \in \ZZ^k$ and $w_M \in S_k$ as in \eqref{eq:lammu}. Then the $q$-character of the simple $U_q(\widehat{\mathfrak{sl}_n})$-module $L(M)$ is given by 
\begin{align} \label{eq:Kazhdan-Lusztig character formula qcharacter formula}
\chi_q(L(M)) = \sum_{u \in S_k} (-1)^{\ell(uw_M)} p_{uw_0, w_Mw_0}(1) \prod_{M' \in {\rm Fund}_M(u\mu_M, \lambda_M)} \chi_q(L(M')).
\end{align}
\end{theorem}

The quantities $\chi_q(L(M'))$ appearing on the right hand side of \eqref{eq:Kazhdan-Lusztig character formula qcharacter formula}, namely $q$-characters of fundamental modules, can be computed via the Frenkel-Mukhin algorithm~\cite{FM}, so this is indeed a formula for the $q$-character of $L(M)$. Ginzburg and Vasserot have given a formula similar to \eqref{eq:Kazhdan-Lusztig character formula qcharacter formula} in geometric language, cf.~\cite[Theorem 3]{Vas} and also \cite{GV}.

\begin{example}\label{eg:firstqcharacter}
Let $M = Y_{2,-4} Y_{1, -1}$ with $k=2$. Then ${\bf m}_M = [0,0] + [-2, -1]$, $\mu_M = (0, -2)$, $\lambda_M = (0, -1)$, $S_{\lambda}=S_{\mu}=\{1\}$ and $w_M = 1$. 
The right hand side of (\ref{eq:Kazhdan-Lusztig character formula qcharacter formula}) has two terms ($u=s_1$ and $u=s_1$): 
\begin{align*}
& (-1)^{\ell(1)} p_{s_1, s_1}(1) \chi_q(M_{[0,0]})\chi_q(M_{[-2,-1]})  = \chi_q(Y_{1,-1})\chi_q(Y_{2,-4}), \\
& (-1)^{\ell(s_1)} p_{1, s_1}(1)  \chi_q(M_{[0, -1]})\chi_q(M_{[-2, 0]}) = - \chi_q(Y_{3,-3}), 
\end{align*}
so that $\chi_q(Y_{2,-4} Y_{1, -1}) = \chi_q(Y_{1,-1})\chi_q(Y_{2,-4}) - \chi_q(Y_{3,-3})$, valid for any $n$. When $n=3$, the formula simplifies to $\chi_q(Y_{2,-4} Y_{1, -1}) = \chi_q(Y_{1,-1})\chi_q(Y_{2,-4}) - 1$. 
\end{example}

Composing \eqref{eq:multisegtomonom} with the bijection in Lemma~\ref{lem:fundamental modules and Plucker coordinates}, we obtain a correspondence between multi-segments and fundamental tableaux in ${\rm SSYT}(n,[m])$ (for any fixed value of $n$): 
\begin{equation}\label{eq:multisegtotab}
[a, b] \mapsto T_{[1-a,1-a+n] \setminus \{n-b\}}.
\end{equation}
Thus, one can directly translate Theorem~\ref{cor:qcharacter formula} to obtain an inhomogeneous formula for $\ch(T) \in \bbc[\Gr(n,m,\sim)]$, replacing $\chi_q(L(M'))$ for a fundamental monomial $M'$ by the corresponding Pl\"ucker coordinate.

Rather than writing down this inhomogeneous formula for $\ch(T) \in \bbc[\Gr(n,m,\sim)]$, we give the ``correct'' homogeneous lift to $\CC[\Gr(n,m)]$. Namely, by Lemma~\ref{lem:plucker and minimal affinization}, every $\sim$-equivalence class contains a unique tableau $T'$ with small gaps. We begin by giving the unique lift of $\ch(T) \in \bbc[\Gr(n,m,\sim)]$ to $\CC[\Gr(n,m)]$ whose $\ZZ^m$-degree is the content of $T'$. Once this is done, we define $\ch(T)$ for all $T$ by homogeneity. 

Let $T \in {\rm SSYT}(n,[m])$ of gap weight $k$ be given, corresponding to $\mu_T,\lambda_T \in \ZZ^k$ and $w_T \in S_k$ as in \eqref{eq:lammu}. It is convenient to repackage the indexing data $\mu_T,\lambda_T$ in the following form. Let $T'$ be the small gaps tableau equivalent to $T$. Clearly $T'$ is the same as the following data: the weakly increasing sequence ${\bf i} = i_1 \leq i_2 \dots \leq i_k$ of entries in the first row, and the elements $r_1,\dots,r_k$ that are ``deleted'' from each column, meaning that the $a$th column of $T'$ has content $[i_a,i_a+n] \setminus \{r_a\}$. We let ${\bf j} = j_1 \leq j_2 \leq \dots \leq j_k$ be the elements $r_1,\dots,r_k$ written in weakly increasing order. In the notation of \eqref{eq:lammu} and \eqref{eq:multisegtotab}, we have that $i_a = 1-\mu_a$ and $j_a = n-\lambda_a$ for $a \in [k]$. Moreover $r_a = j_{w_T^{-1}(a)}$. Rephrasing the above in a way that is more compatible with the next definition, we can say that the columns of $T'$ have content of the form $[i_{w_T(a)},i_{w_T(a)}+n] \setminus \{j_a\}$ for $a \in [k]$ (written this way, we are sorting based on ${\bf j}$ rather than sorting based on ${\bf i}$).

\begin{definition}\label{defn:usemicolonT}
Let $T$ be a small gaps tableau with $k$ columns, with ${\bf i,j} \in \ZZ^k$ the weakly increasing sequences just defined. For $u \in S_k$, define $P_{u;T} \in \CC[\Gr(n,m)]$ as follows. Provided  $j_a \in [i_{u(a)}, i_{u(a)}+n]$ for all $a \in [k]$, define the tableau $\alpha(u;T)$ to be the semistandard tableau whose columns have content $[i_{u(a)}, i_{u(a)}+n] \setminus \{j_a\}$ for $a \in [k]$, and define $P_{u; T} = P_{\alpha(u;T)} \in \CC[\Gr(n,m)]$ to be the corresponding standard monomial. On the other hand, if $j_a \notin [i_{u_a}, i_{u(a)}+n]$ for some $a$, then the tableau $\alpha(u;T)$ is {\sl undefined} and $P_{u ;T} = 0$. 
\end{definition}

The tableau $\alpha(u;T)$ has nonlarge gaps by construction, so the monomial $P_{u;T}$
is indeed standard by Lemma~\ref{lem:nosorting}. Note also that $P_{w_T;T} = P_{T}$, and more generally $P_{u;T} = P_{T}$ for $u \in S_{\lambda}w_TS_{\mu}$. 

\begin{example}
Let $T =\begin{ytableau}
1&1&2&2&3&3&4 \\
2&2&3&3&5&5&5 \\
4&4&4&4&6&6&6 \\
5&5&6&6&7&7&8
\end{ytableau}$ with $\substack{{\bf i} = 1,1,2,2,3,3,4\\{\bf j} = 3,3,4,4,5,5,7}$. The permutation $u\in S_7$ with one-line notation $u = 2165437$ has the property that the sets $\{[i_{u(a)},i_{u(a)}+4] \setminus j_a\}_{a \in [7]}$ describe the columns of $T$. Moreover, $u$ is of maximal length with this property (i.e. $u = w_{T}$). For a general $u \in S_7$, $P_{u;T}$ is nonzero exactly when $u(1) \neq  7$ and $u(7) \geq 5$. If we take $u = 3124576 \in S_7$, then $\alpha(u;T)$ is given by lexicographically sorting the columns of $\begin{ytableau}
2&1&1&2&3&4&3 \\
4&2&2&3&4&6&4 \\
5&4&3&5&6&7&5 \\ 
6&5&5&6&7&8&6
\end{ytableau}
$ and $P_{u;T}$ is the standard monomial with these columns. 
\end{example} 

With these preparations, we have the following corollary of Theorem~\ref{thm:multi-segment character formula} for Grassmannians. 
\begin{theorem} \label{cor:qcharacter formula tableaux}
Let $T \in {\rm SSYT}(n,[m])$ with gap weight $k$, and let $T' \sim T$ the small gaps tableau equivalent to $T$. Let $w_T \in S_k$ be the maximal length permutation described after \eqref{eq:lammu}. Then 
\begin{align}\label{eq:Kazhdan-Lusztig character formula tableau formula}
\ch(T) = \sum_{u \in S_k} (-1)^{\ell(uw_T)} p_{uw_0, w_Tw_0}(1) P_{u; T'} \in \bbc[\Gr(n,m,\sim)]
\end{align}
with $P_{u ; T'}$ the standard monomial just defined.  
\end{theorem}

Though the equality \eqref{eq:Kazhdan-Lusztig character formula tableau formula} holds in $\CC[\Gr(n,m,\sim)]$, the right hand side makes sense in $\bbc[\Gr(n,m)]$. It is homogeneous, with $\ZZ^m$-degree the content of the small gaps tableau $T'$. 

Let $\bbc[\Gr^\circ(n,m)]$ denote the localization of $\bbc[\Gr(n,m)]$ at all the frozen variables. For our purposes, we in fact only need to localize at the trivial frozen variables, but the localization $\bbc[\Gr^\circ(n,m)]$ is a more familiar object (it is the homogeneous coordinate ring of the {\sl open positroid variety}). 
The next definition lifts $\ch(T)$ from $\bbc[\Gr(n,m,\sim)]$ to $\CC[\Gr^\circ(n,m)]$. 
\begin{definition}\label{defn:chTforGr}
Let $T \in {\rm SSYT}(n,[m])$ and let $T = T'' \cup T'$ where $T'$ has small gaps and $T'' \in K_0({\rm SSYT}(n,[m]))$ is a fraction of two trivial tableaux (cf.~Remark~\ref{rmk:TtoTprime}). Define 
\begin{equation}\label{eq:liftofchT}
\ch(T) = P_{T''} \ch(T') \in \bbc[\Gr^\circ(n,m)]\end{equation}
where $\ch(T')$ is as defined in the right hand side of \eqref{eq:Kazhdan-Lusztig character formula tableau formula}, and $P_{T''}$ is the Laurent monomial in trivial frozen Pl\"ucker coordinates corresponding to $T''$.
\end{definition}

We conjecture that the elements $\{\ch(T)\}_{T \in {\rm SYYT}(n,[m])}$ lie in $\bbc[\Gr(n,m)]$, not merely in the localization $\bbc[\Gr^\circ(n,m)]$. We provide evidence that they are Lusztig's dual canonical basis (also known as Kashiwara's upper global base) for $\bbc[\Gr(n,m)]$. In a very closely related setting, it is already known that the basis of simples in a monoidal categorification matches the dual canonical basis \cite{HL15}. 

\begin{remark}\label{rmk:clearing}
The right hand side of \eqref{eq:liftofchT} is the unique homogeneous lift of $\ch(T) \in \CC[\Gr(n,m,\sim)]$ to $\CC[\Gr^\circ(n,m)]$ whose $\ZZ^m$-degree matches the content of $T$. However, because the formula \eqref{eq:liftofchT} can have frozen variables in the denominator, it is not obvious that this lift lies in 
$\CC[\Gr(n,m)]$ rather than in the localization $\CC[\Gr^\circ(n,m)]$. For cluster monomials, this well-behavedness is clear by the discussion in Section \ref{sec:mutations description}. We give further evidence that the lifts lie in $\bbc[\Gr(n,m)]$ for arbitrary tableaux by checking this for the ``smallest'' nonreal tableaux in Example~\ref{eg:catalogue}.  
\end{remark}

The following proposition follows from Theorem \ref{thm:Hernandez-Leclerc quantum affine algebras and Grassmannians}, Equation (\ref{eq:expression of Phi(L(M)) in terms of tableaux}), and Theorem~\ref{cor:qcharacter formula}.
\begin{proposition} \label{prop: characher formula in terms of semistandard tableaux}
For a tableau $T \in {\rm SSYT}(n,[m])$, 
\begin{align} \label{eq:ch(T) using semistandard tableau}
\ch(T)= P_T + P_{T''}\sum_{S} c_{S} P_{S} \in \bbc[\Gr^\circ(n,m)],
\end{align}
where the sum is over tableaux $S \in {\rm SSYT}(n,[m])$ with lower weight than $T$, $c_S \in \ZZ$, and $P_{T''}$ is the Laurent monomial in trivial frozen variables above. 
\end{proposition}

In particular, the passage from standard monomials $\{P_T\}$ to the elements $\{\ch(T)\}$ is triangular, so that the $\{\ch(T)\}$ are linearly independent.

\begin{example}\label{eg:124356}
We illustrate the formula~\eqref{eq:Kazhdan-Lusztig character formula tableau formula} for the simplest non-Pl\"ucker cluster variable, labeled by the tableau with columns $[1,2,4],[3,5,6]$. This tableau has small gaps, and has ${\bf i} = 1,3$, ${\bf j} = 3,4$, $w_T = {\rm id}$. It corresponds to the monomial 
$M = Y_{2,-4} Y_{1, -1}$ from Example~\ref{eg:firstqcharacter}.  
Applying the formula~\eqref{eq:Kazhdan-Lusztig character formula tableau formula} directly yields 
\begin{align}\label{eq:124356}
\ch( \begin{ytableau} 1 & 3 \\ 2 & 5 \\ 4 & 6 \end{ytableau} ) = P_{124}P_{356} - P_{123}P_{456},
\end{align}
with the first term corresponding to $u = {\rm id}$ and the second corresponding to $u = s_1$. This quadratic expression in Pl\"ucker coordinates is a cluster variable. An alternative way to compute the right hand side of \eqref{eq:124356} is to translate the $n=3$ version of Example~\ref{eg:firstqcharacter} using the correspondence between fundamental monomials and fundamental tableaux. The result is $P_{124}P_{356}-1$, which lifts to the right hand side of \eqref{eq:124356} using homogeneity. 

If instead we translate the $n=4$ instance of Example~\ref{eg:firstqcharacter}, the tableau formula is
\begin{align*}
\ch( \begin{ytableau} 1 & 3 \\ 2 & 4 \\ 3 & 6 \\ 5 & 7 \end{ytableau} ) = P_{1235}P_{3467} - P_{3567} = P_{1235}P_{3467} - P_{1234}P_{3567}.
\end{align*}
\end{example}

\begin{example} \label{example:qcharacter of a module}
We now address the other non-Pl\"ucker cluster variable in $\CC[\Gr(3,6)]$, corresponding to the tableau with columns $[1,3,5],[2,4,5]$. It has gap weight $k=4$ and $w_M = w_T = s_2 \in S_4$. 

The corresponding monomial is $M=Y_{1,-5} Y_{1,-3} Y_{2,-2} Y_{2,0}$ (cf.~Example~\ref{eg:cosets}). The Kazhdan-Lusztig polynomials $\{p_{w'w_0, s_2w_0}(t) \colon w' \in S_k\}$ are in the set $\{0,1,1+t\}$. Thus $p_{w'w_0, s_2w_0}(1) \in \{0,1,2\}$. Exactly twenty of these are nonzero, but there are certain cancellations. 
By the alternating nature of \eqref{eq:Kazhdan-Lusztig character formula qcharacter formula}, the contributions of $w',w's_2$ cancel whenever their K-L polynomials match. After these and similar cancellations, \eqref{eq:Kazhdan-Lusztig character formula qcharacter formula}) simplifies to
\begin{align} \label{eq: example of qcharacter formula}
\begin{split}
\chi_q(M) & = \chi_q(Y_{2,-2}) \chi_q(Y_{4,-2}) - \chi_q(Y_{3,-1}) \chi_q(Y_{3,-3}) + \chi_q(Y_{1,-1}) \chi_q(Y_{3,-1}) \chi_q(Y_{2,-4}) \\
& \quad - \chi_q(Y_{2,0}) \chi_q(Y_{2,-2}) \chi_q(Y_{2,-4}) - \chi_q(Y_{1,-1}) \chi_q(Y_{1,-3}) \chi_q(Y_{3,-1}) \chi_q(Y_{1,-5}) \\
& \quad + \chi_q(Y_{2,0}) \chi_q(Y_{1,-3}) \chi_q(Y_{2,-2}) \chi_q(Y_{1,-5}), 
\end{split}
\end{align}
valid for any~$n$. 

Specializing to $n=3$, certain terms vanish and we obtain 
\begin{align}\label{eq:nequals3q}
\begin{split}
\chi_q(M) & = - 1 + \chi_q(Y_{1,-1}) \chi_q(Y_{2,-4}) - \chi_q(Y_{2,0}) \chi_q(Y_{2,-2}) \chi_q(Y_{2,-4}) \\
& \quad - \chi_q(Y_{1,-1}) \chi_q(Y_{1,-3}) \chi_q(Y_{1,-5}) + \chi_q(Y_{2,0}) \chi_q(Y_{1,-3}) \chi_q(Y_{2,-2}) \chi_q(Y_{1,-5}).
\end{split}
\end{align}

To compute $\ch(T)$, we must first compute $\ch(T')$ where $T = T'' \cup T'$ as in Example~\ref{eg:factorourexample}. With the same KL cancellations as above, we have \begin{equation}\label{eq:chT'}
\begin{split}
\ch(T') =& -P_{123}P_{234}P_{345}P_{456}+P_{124}P_{234}P_{345}P_{356}-P_{134}P_{234}P_{245}P_{356} \\
        &-P_{124}P_{235}P_{345}P_{346}+ P_{134}P_{235}P_{245}P_{346}.
\end{split}
\end{equation}
Then by definition, $\ch(T) = \frac{\ch(T')}{P_{234}P_{345}}.$ Remark~\ref{rmk:clearing} asserts that $\ch(T')$ is in fact divisible by $P_{234}P_{345}$, so that $\ch(T)$ is in $\CC[\Gr(3,6)]$ (not merely in $\CC[\Gr^\circ(3,6)]$). We confirm this directly via Pl\"ucker relations in Example~\ref{example:character of the tableau 135246}, and show that $\ch(T)$ is a cluster variable in Example~\ref{example:chT is equal to web invariant}.

The $n=4$ version of \eqref{eq: example of qcharacter formula} only changes by setting $\chi_q(Y_{4,-2}) =1$. 
The $n=4$ (inhomogeneous) tableau formula is 
\begin{align*}
\ch(\begin{ytableau}
1 & 2  \\
2 & 3 \\
4 & 5 \\
6 & 7 
\end{ytableau}) & = P_{2356} - P_{2456} P_{3567} + P_{1235} P_{2456} P_{3467} - P_{1245} P_{2356} P_{3467} \\
& \qquad - P_{1235} P_{2346} P_{2456} P_{3457} + P_{1245} P_{2346} P_{2356} P_{3457} \in \bbc[\Gr(4,7,\sim)].
\end{align*}
\end{example}

\begin{example} \label{example:character of the tableau 135246}
Continuing the previous example, we claim that $\frac{1}{P_{234}P_{345}}\ch(T') \in \bbc[\Gr(3,6)]$, where $\ch(T')$ is the right hand side of \eqref{eq:chT'}. This is a consequence of the following Pl\"ucker relations: 
\begin{align}
P_{245}P_{356} &= P_{345}P_{256}+P_{235}P_{456} \label{eq:straightener3} \\
P_{124}P_{235} &= P_{234}P_{125}+P_{123}P_{245} \label{eq:straightener4} \\ 
P_{134}P_{235} &= P_{234}P_{135}+P_{345}P_{123} \label{eq:straightener5a}\\ 
P_{245}P_{346} &= P_{345}P_{246} +P_{234}P_{456}\label{eq:straightener5b} \\
P_{123}P_{345} &= P_{234}P_{135}-P_{134}P_{235} \label{eq:straightener35} \\
P_{234}P_{456} &= P_{345}P_{246}-P_{245}P_{346} \label{eq:straightener45}.
\end{align}
Applying \eqref{eq:straightener3} to the third term in $\ch(T')$, the first term on the right hand side is divisible by $P_{234}P_{345}$ and we consider the second term as a ``leftover term''. Likewise, applying \eqref{eq:straightener4} to the fourth term in $\ch(T')$, the first term on the right hand side is divisible and the second term is leftover. Applying both \eqref{eq:straightener5a} and \eqref{eq:straightener5b} to the fifth term in $\ch(T')$, two of the resulting terms are divisible and the other two terms are leftover. The first of these leftover terms cancels with the leftover term from \eqref{eq:straightener3} using \eqref{eq:straightener35}, and the second of these leftover terms cancels with the leftover term from \eqref{eq:straightener4} using \eqref{eq:straightener45}. The result is divisible by $P_{234}P_{345}$, and keeping track of the terms one gets   
\begin{equation}\label{eq:chTstdmonom}
\ch(T) = \frac{1}{P_{234}P_{345}}\ch(T') = 2P_{123}P_{456}+P_{124}P_{356}-P_{125}P_{346}-P_{134}P_{256}+P_{135}P_{246}.\end{equation}
This is the expression for $\ch(T)$ in terms of standard monomials. Note that the highest weight term is $P_T$, as expected.  
\end{example}

It would be interesting to generalize the calculations in Example~\ref{example:character of the tableau 135246}, obtaining the standard monomial expression for $\ch(T)$ from the one for $\ch(T')$. 

\subsection{Kazhdan-Lusztig Immanants}
We rephrase the formula \eqref{eq:Kazhdan-Lusztig character formula tableau formula} in the language of {\sl Kazhdan-Lusztig immanants} defined by Rhoades and Skandera \cite{RS}. As a corollary, we conclude that $\ch(T)$ is nonnegative on the totally nonnegative Grassmannian. 

\begin{definition}\label{defn:KLImmanant}
Let $A = (M_{i,j})_{i,j \in [m]} \in \GL_m$. For a permutation $v \in S_m$, the {\sl Kazhdan-Lusztig immanant} $\text{Imm}_v \in \CC[\GL_m]$ is the function  
\begin{equation}
A \mapsto \sum_{u \geq v}(-1)^{\ell(w)-\ell(v)}p_{w_0u,w_0v}(1)\prod_{i=1}^mA_{i,u(i)},
\end{equation}  
where $\geq$ denotes the (strong) Bruhat order on $S_m$. 
\end{definition}

Let ${\bf i} = 1 \leq i_1 \leq i_2 \leq \cdots \leq i_k \leq m$ and  ${\bf j} = 1 \leq j_1 \leq j_2 \leq \cdots \leq j_k \leq m$ be two weakly increasing sequences of indices, thought of as row and column indices, respectively. The {\sl generalized submatrix} of $A$ corresponding to ${\bf i,j}$
is the $k \times k$ matrix $A^{{\bf i,j}}$ whose $(a,b)$ entry is $A_{i_a,j_b}$ for $a,b \in [k]$. Du and Skandera \cites{Du,Ska} showed that the dual canonical basis for $\CC[\GL_m]$ are exactly the nonzero Kazhdan-Lusztig immanants of generalized submatrices, i.e. functions of the form  
$A \mapsto \text{Imm}_v(A^{{\bf i,j}})$ as ${\bf i,j}$ and $v \in S_k$ vary.

Consider the regular map 
\begin{equation}\label{eq:MS}
\text{MS} \colon \Gr(n,m) \to B \subset \GL_m \hspace{.5cm} \text{ sending } x \mapsto (P_{[i,i+n] \setminus \{j\}}(x))_{i,j \in [m]},  
\end{equation}
where $B \subset \GL_m$ is the subgroup of upper triangular matrices. As usual, we treat indices of Pl\"ucker coordinates modulo  
$m$ and treat $P_{[i,i+n] \setminus j}(x) = 0$ if $j \notin [i,i+n]$. The map MS is closely related to a map defined by Marsh and Scott (see Remark~\ref{rmk:MS}). 

\begin{example} For $x \in \Gr(3,5)$, one has MS$(x) = \begin{pmatrix}
P_{234} & P_{134} & P_{124} & P_{123} &  0\\
0 & P_{345} & P_{245} & P_{235} &  P_{234}\\
0 & 0 & P_{145} & P_{135} &  P_{134}\\
0 & 0 & 0 & P_{125} &  P_{124}\\
0 & 0 & 0 & 0 &  P_{123}\\
\end{pmatrix},$
where we everywhere abbreviate $P_{234}(x) = P_{234}$, etc. In general, $MS(x)$ has nonzero entries concentrated within $n+1$ diagonals, with frozen Pl\"ucker coordinates sitting on the outermost diagonals.  
\end{example}

A point $x \in \Gr(n,m)$ is {\sl totally nonnegative} (TNN) if all of its Pl\"ucker coordinates are real and nonnegative. A function $f \in \CC[\Gr(n,m)]$ or in $\CC[\Gr^\circ(n,m)]$ is a {\sl TNN function} if $f(x) \geq 0$ whenever $x$ is TNN. Likewise, a matrix $A \in \GL_m$ is {\sl TNN} if all of its minors are nonnegative, and $f \in \CC[\GL_m]$ is a {\sl TNN function} if $f(A) \geq 0$ whenever $A \in \GL_m$ is TNN.

Let $T \in {\rm SSYT}(n,[m])$ with small gaps tableau $T'$. These determine a triple $\mu,\lambda,w_T$ \eqref{eq:lammu}, and moreover weakly increasing sequences $\bf{i,j}$ with $i_a = 1-\mu_a, j_a = n-\lambda_a$ as above. 

\begin{proposition}\label{prop:KLandChT}
If $T'$ is a small gaps tableau, then $\ch(T') \in \CC[\Gr(n,m)]$ is the pullback ${\rm MS^*}(f)$ of a dual canonical basis element $f \in \CC[\GL_m]$. Explicitly, if $T'$ corresponds to ${\bf i,j} \in \ZZ^k$ and $w_T \in S_k$ as just defined, then $f$ is the Kazhdan-Lusztig immanant $A \mapsto \text{Imm}_{w_T^{-1}}(A^{{\bf i,j}})$ indexed by $w_T^{-1},{\bf i,j}$. In particular, $\ch(T) \in \bbc[\Gr^\circ(n,m)]$ is a TNN function, for any $T \in {\rm SSYT}(n,[m])$. 
\end{proposition}

We view this as further evidence that $\{\ch(T)\}_{T \in {\rm SSYT}(n,[m])}$ is the dual canonical basis for $\CC[\Gr(n,m)]$. 

\begin{proof}
This follows from the definitions, but we spell it out to match conventions. We use the following well known facts: $p_{u,v}(t) = 0$ unless $u \leq v$ in Bruhat order, $p_{u,v} = p_{u^{-1},v^{-1}}$, and multiplication by $w_0$ is an anti-automorphism of Bruhat order. Then for $x \in \Gr(n,m)$, 
\begin{align}
\ch(T')(x) &= \sum_{u \in S_k}(-1)^{\ell(uw_T)}p_{uw_0,w_Tw_0}(1)P_{u \cdot T}(x)\\
&= \sum_{u \in S_k}(-1)^{\ell(uw_T)}p_{uw_0,w_Tw_0}(1) \prod_{s=1}^k ({\rm MS}(x)^{{\bf i,j}})_{u(s),s} \\
&= \sum_{u \in S_k}(-1)^{\ell(u^{-1}w_T)}p_{w_0u^{-1},w_0w_T^{-1}}(1) \prod_{s=1}^k ({\rm MS}(x)^{{\bf i,j}})_{s,u^{-1}(s)} \\
&= \text{Imm}_{w_T^{-1}}({\rm MS}(x)^{{\bf i,j}}),
 \end{align}
where passing from the first line to the second we used the definition of MS$(x)$, passing from the second to the third we used properties of the sign character on $S_k$ and the KL property involving inverses, and passing to the final line we replaced $u$ by $u^{-1}$ and used the facts about Bruhat order. This string of equalities says that 
the KL-immanant pulls back to $\ch(T')$ under the map $x \mapsto {\rm MS}(x)$.  

To see that $\ch(T)$ is a totally nonnegative function, it suffices to prove that $\ch(T')$ is a TNN function, since the two differ by a Laurent monomial in frozen Pl\"ucker coordinates. To show that $\ch(T')$ is, it suffices to prove that ${\rm MS}(x)^{{\bf i,j}}$ is totally nonnegative, because Kazhdan-Lusztig immanants are TNN functions \cite[Proposition 2]{RS} (alternatively, one can use the well known fact that dual canonical basis elements for $\GL_m$ are TNN functions). To show that ${\rm MS}(x)^{{\bf i,j}}$ is TNN, it suffices to show that ${\rm MS}(x)$ is. By the Fekete criterion, one can check that a triangular matrix $A$ is TNN by checking that each of its row-solid minors is nonnnegative. But each such row-solid minor of ${\rm MS}(x)$ is a monomial in Pl\"ucker coordinates of $x$, as follows by an easy induction, or by the explicit formulas \cite[Lemma 5.2]{MS}.
\end{proof}

The dual canonical basis elements appearing in Proposition~\ref{prop:KLandChT} are a proper subset of the dual canonical basis for $\CC[\GL_m]$, namely those whose generalized submatrices use only the first $m-n$ rows. We expect that {\sl every} dual canonical basis element $f$ has the property ${\rm MS}^*(f)$ is either zero, or is of the form $\ch(T)$ for some tableau $T$ (not necessarily with small gaps). On the other hand, Remark~\ref{rmk:notsurjective} shows that not every $\ch(T)$ is not pullback of a dual canonical basis element for $\CC[\GL_m]$. We believe that ``matching up'' the two bases requires localizing at frozen variables on both sides. 

\begin{remark}\label{rmk:MS}
Marsh and Scott \cite{MS} gave an isomorphism of the open positroid variety $\Gr^\circ(n,m)$ with a certain {\sl unipotent cell} $N^w \subset N \subset \GL_m$ where $w \in S_m$ is the Grassmann permutation $i \mapsto i+n \mod m$ and $N\subset \GL_m$ is the subgroup of unipotent upper triangular matrices. Our map MS is a variant of theirs: their version rescales each row by the reciprocal of a frozen variable while we do not, and moreover, theirs differs from ours by the complementation map $P_I \to P_{[m] \setminus I}$. Hernandez and Leclerc \cite[Section 13.7]{HL10} described a certain map $\mathcal{R}_\ell \twoheadrightarrow K_0(\mathcal{C}_{\ell}^{A_{n-1}})$, where 
$\mathcal{R}_\ell$ is an algebra built from Grothendieck classes of modules for affine Hecke algebras. Under this map, each simple object is either sent to 
a simple object or to zero. We expect that this map is the pullback ${\rm MS}^*$, restricted to the subalgebra $\CC[N] \subset \CC[B]$, followed by the quotient map $\CC[\Gr(n,m)] \to \CC[\Gr(n,m,\sim)]$.   
\end{remark}

\section{Fomin and Pylyavskyy's conjectures}\label{sec:Fomin Pylyavskyy conjecture}
We recall some aspects of Fomin and Pylyavskyy's conjecture about cluster combinatorics in Grassmannian cluster algebras \cite{FP}. We begin by reviewing the definitions of tensor diagrams, webs, and web invariants. 

\begin{definition}
Let $D$ be a disk with some number $m$ of marked points on its boundary. A {\sl tensor diagram} is a finite bipartite multigraph $T$, drawn in $D$, with a fixed bipartition of its vertex set into black and white color sets, subject to the following additional requirements:
\begin{itemize}
\item the marked points of $D$ ({\sl boundary vertices}) are in the vertex set of $T$ and are colored black,
\item the remaining {\sl interior vertices} are in the interior of $D$ and are trivalent.
\end{itemize}
The boundary vertices can have arbitrary valence (including zero). Tensor diagrams are allowed can be nonplanar, with at most two edges crossing transversely at any point.

A tensor diagram is a {\sl web} (or an $\mathfrak{sl}_3$-web) if it is planar (i.e., if there are no such crossings). A web is {\sl non-elliptic} it has no multiple edges, and if any face formed by interior vertices has at least six sides. 
\end{definition}
Every tensor diagram $T$ with $m$ boundary vertices defines an invariant $[T] \in \CC[\Gr(3, m)]$ by an explicit formula \cite[Eq. (4.1)]{FP}. The $\ZZ^m$-degree of $[T]$ is the vector recording the valence of the boundary vertices. A {\sl web invariant} is an element $[W] \in \CC[\Gr(3, m)]$ for a non-elliptic web $W$. 

One can calculate with tensor diagrams $[T]$ implicitly via diagrammatic relations known as {\sl skein relations}. One considers the 
vector space of $\CC$-linear formal combinations of tensor diagrams with $m$ boundary vertices (considered up to isotopy fixing the boundary vertices), with multiplication given by superposition of diagrams. Then $\CC[\Gr(3, m)]$ is the quotient of this algebra by the diagrammatic relations in Figure \ref{fig:skein relations}. We illustrate these relations in many examples below. In this language, Pl\"ucker coordinates correspond to ``tripod'' diagrams.  
\begin{figure}
\begin{center}
\begin{tabular}{c}
\begin{tikzpicture}[scale = .6]
\coordinate (A) at (0,1);
\coordinate (B) at (2,1);
\coordinate (C) at (2,0);
\coordinate (D) at (0,0);
\draw (A)--(C);
\draw (B)--(D);

\draw [fill= white] (A) circle [radius = .065];
\draw [fill= white] (D) circle [radius = .065];
\draw [fill= black] (B) circle [radius = .065];
\draw [fill= black] (C) circle [radius = .065];
\node at (3,.5) {$=$};
\node at (7.2,.5) {$+$};

\begin{scope}[xshift = 4cm]
\coordinate (AA) at (0,1);
\coordinate (BB) at (2,1);
\coordinate (CC) at (2,0);
\coordinate (DD) at (0,0);
\draw (AA)--(BB);
\draw (CC)--(DD);

\draw [fill= white] (AA) circle [radius = .065];
\draw [fill= white] (DD) circle [radius = .065];
\draw [fill= black] (BB) circle [radius = .065];
\draw [fill= black] (CC) circle [radius = .065];
\end{scope}

\begin{scope}[xshift = 8cm]
\coordinate (AAA) at (0,1);
\coordinate (BBB) at (2,1);
\coordinate (CCC) at (2,0);
\coordinate (DDD) at (0,0);
\coordinate (E) at (.67,.5);
\coordinate (F) at (1.33,.5);
\draw (AAA)--(E)--(F)--(BBB);
\draw (CCC)--(F);
\draw (DDD)--(E);
\draw [fill= white] (AAA) circle [radius = .065];
\draw [fill= white] (DDD) circle [radius = .065];
\draw [fill= black] (BBB) circle [radius = .065];
\draw [fill= black] (CCC) circle [radius = .065];
\draw [fill= black] (E) circle [radius = .065];
\draw [fill= white] (F) circle [radius = .065];
\node at (2.5,0) {\hfill};

\end{scope}
\end{tikzpicture} \\
 \vspace{.05cm} \\ 
 \begin{tikzpicture}[scale = .7]
\coordinate (A) at (.2,.8);
\coordinate (B) at (1.8,.8);
\coordinate (C) at (1.8,.2);
\coordinate (D) at (.2,.2);
\draw (A)--(B)--(C)--(D)--(A);
\draw (A)--(0,1);
\draw (B)--(2,1);
\draw (C)--(2,0);
\draw (D)--(0,0);

\draw [fill= white] (A) circle [radius = .065];
\draw [fill= white] (C) circle [radius = .065];
\draw [fill= black] (B) circle [radius = .065];
\draw [fill= black] (D) circle [radius = .065];
\node at (3,.5) {$=$};
\node at (7.2,.5) {$+$};

\begin{scope}[xshift = 4cm]
\draw [rounded corners] (0,1)--(1,.8)--(2,1);
\draw [rounded corners] (0,0)--(1,.2)--(2,0);
\end{scope}

\begin{scope}[xshift = 8cm]
\draw [rounded corners] (0,1)--(.2,.5)--(0,0);
\draw [rounded corners] (2,1)--(1.8,.5)--(2,0);
\end{scope}
\end{tikzpicture} \\
\vspace{.05cm} \\ 
\begin{tikzpicture}[scale = .7]
\draw  (0,-2) circle [radius = .65];
\draw (-1.25,-2)--(-.65,-2);
\draw (1.25,-2)--(.65,-2);
\draw [fill= white] (-.65,-2) circle [radius = .07];
\draw [fill= black] (.65,-2) circle [radius = .07];
\node at (3,-2) {$= \hspace{.1cm} -2 \hspace{.2cm} \times$}; 
\draw (4.5,-2)--(7.5,-2);
\node at (9.5,-2) {\hfill};
\end{tikzpicture} \\
\vspace{.025cm} \\ 
\begin{tikzpicture}[scale = .7]
\draw (.5,0) circle [radius = .6];
\node at (-1.4,0) {\hfill};
\node at (3,0) {$= \hspace{.1cm} 3$}; 
\node at (11,0) {\hfill};
\end{tikzpicture} \\
\begin{tikzpicture}[scale = .7]
\draw [gray] (0,0) arc [radius=2, start angle=-90, end angle= -60];
\draw [gray] (0,0) arc [radius=2, start angle=-90, end angle= -120];
\draw [rounded corners] (0,0)--(-.4,.5)--(0,1);
\draw [rounded corners] (0,0)--(.4,.5)--(0,1);
\draw (0,1)--(0,1.35);
\draw [fill= black] (0,0) circle [radius = .065];
\draw [fill= white] (0,1) circle [radius = .065];
\node at (2.5,1) {$= \hspace{.1cm} 0$};
\node at (9.4,0) {\hfill};
\end{tikzpicture}
\end{tabular}
\end{center}
\caption{Skein relations for $\mathfrak{sl}_3$ webs. \label{fig:skein relations}}
\end{figure}
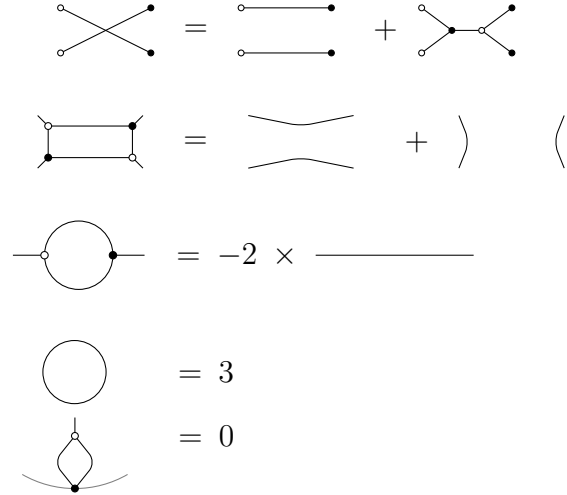

Kuperberg \cite{Kup} showed that web invariants are a basis for $\CC[\Gr(3, m)]$, and with Khovanov \cite{KK} gave a bijective labeling $T \mapsto [W(T)]$ of these webs by tableaux $T \in {\rm SSYT}(3,[m])$. For an easily computable description of this bijection we refer to \cite{Tym}. One useful property of the bijection is the following: if $T$ has an $i$ weakly northeast of an $i+1$, then  boundary vertices $i,i+1$ are joined by a ``fork'' or ``Y'' in $W(T)$ (see Example~\ref{eg:124356web} for an illustration of the meaning of ``fork''). This mnemonic allows for quick computation of $W(T)$ in examples. 

Fomin and Pylyavskyy \cite[Conjecture 9.3]{FP} conjectured that every cluster monomial in $\CC[\Gr(3, m)]$ is a web invariant. This is known in the finite mutation type cases $m \leq 9$ \cite{Fra}. We state the following more specific version.

\begin{conjecture}\label{conj:web invariant equals to ch(T)}
For $T \in {\rm SSYT}(3,[m])$, if $\ch(T) \in \CC[\Gr(3,m)]$ is a cluster monomial, then $\ch(T) = [W(T)]$. 
\end{conjecture}

\begin{remark}
If this conjecture, as well as further conjectures of Fomin-Pylavskyy, were proved, then one would obtain a direct pictorial procedure for testing reality and primeness of $U_q(\widehat{\mathfrak{sl}_n})$-modules. Given a simple module $L(M)$, one would draw the corresponding web $W(T_M)$. Using another explicit pictorial procedure  known as the {\sl arborization algorithm} \cite[Section 10]{FP}, one conjecturally should be able to test whether $W(T_M)$ is a cluster monomial and moreover find its factorization into cluster variables (i.e., to test whether $L(M)$ is real and find its tensor factorization into prime modules). 
\end{remark}

A potential approach to proving Conjecture~\ref{conj:web invariant equals to ch(T)} is to show that $[W(T)]$ satisfies the explicit formula for $\ch(T)$. We warn however Example~\ref{example:non-real tableaux} gives an example of a tableau $T$ for which $\ch(T) \neq [W(T)]$. Thus, we imagine that proving Conjecture~\ref{conj:web invariant equals to ch(T)} requires some understanding of exchange relations. 

Replacing the trivalence condition by a similar $n$-valence condition, one can define $\mathfrak{sl}_n$-web invariants for $n \geq 4$. A complete set of skein relations is known \cite{CKM}. We imagine that any cluster monomial $\ch(T) \in \CC[\Gr(n,m)]$ is an $\mathfrak{sl}_n$-web invariant even when $n \geq 4$. Note however that there is no appropriate analogue of web basis or Khovanov-Kuperberg bijection when $n >3$, and webs are merely a distinguished spanning set.

\begin{example}\label{eg:124356web}
Consider the tableau $T$ with columns $[1,2,4],[3,5,6]$ from Example~\ref{eg:124356}. We computed that $\ch(T) = P_{1,2,4}P_{3,5,6}-P_{1,2,3}P_{4,5,6}$. On the other hand, applying the Khovanov-Kuperberg bijection to $T$ yields the web 
$\begin{tikzpicture}[scale=0.3]
\draw[thick] (0,0) circle (2 cm);
  \newdimen\R
   \R=2cm
   \foreach \x/\l/\p/\q in
     {120- 60/{1}/above/v1,
      120-120/{2}/right/v2,
      120-180/{3}/below/v3,
      120-240/{4}/below/v4,
      120-300/{5}/left/v5,
      120-360/{6}/above/v6
     }
     \node[inner sep=1pt,circle,draw,fill,label={\p:\l}] at (\x:\R) (\q) {};
     \node[inner sep=1pt,circle,draw,fill,label={above:6}] at (120-360:\R) (v6) {};
     \node[inner sep=1pt,circle,draw] at (-0.5,0.5) (p1) {};   
     \node[inner sep=1pt,circle,draw] at (0.5,0) (p2) {};
     \node[inner sep=1pt,circle,draw,fill] at (0,-1) (p5) {};
     \node[inner sep=1pt,circle,draw] at (0,-1.5) (p6) {};
     \draw (v1)--(p2);
     \draw (v3)--(p6);
     \draw (v5)--(p1);
     \draw (v2)--(p2);
     \draw (p1)--(p5);
     \draw (p2)--(p5);
     \draw (p6)--(p5);
     \draw (v4)--(p6);
     \draw (v6)--(p1);
\end{tikzpicture}$.
Note this web has a ``fork'' between vertices 1 and 2 (and also between 3 and 4, and 5 and 6) as in the mnemonic above. To see that $[W(T)]$ agrees with $\ch(T)$, 
see the second equation in Figure~\ref{fig:equations of web invariants}, which is an instance of the first skein relation in Figure~\ref{fig:skein relations}. 
\end{example}

\begin{example}\label{example:chT is equal to web invariant}
Consider the tableau $T$ with columns $[1,3,5],[2,4,5]$. Example \ref{example:character of the tableau 135246} gave its expansion into standard monomials:
\begin{align*}
\ch(\begin{ytableau}
1 & 2  \\
3 & 4 \\
5 & 6 
\end{ytableau}) & = P_{1 3 5} P_{2 4 6} - P_{1 2 5} P_{3 4 6}- P_{1 3 4} P_{2 5 6}+ P_{1 2 4} P_{3 5 6}- 2  P_{1 2 3} P_{4 5 6}.
\end{align*}
The web for this tableau is $W_T = 
\begin{tikzpicture}[scale=0.3]
\draw[thick] (0,0) circle (2 cm);
  \newdimen\R
   \R=2cm
   \foreach \x/\l/\p/\q in
     {120- 60/{1}/above/v1,
      120-120/{2}/right/v2,
      120-180/{3}/below/v3,
      120-240/{4}/below/v4,
      120-300/{5}/left/v5,
      120-360/{6}/above/v6
     }
     \node[inner sep=1pt,circle,draw,fill,label={\p:\l}] at (\x:\R) (\q) {};
     \node[inner sep=1pt,circle,draw,fill,label={above:6}] at (120-360:\R) (v6) {};
     \node[inner sep=1pt,circle,draw] at (-0.5,0.5) (p1) {};   
     \node[inner sep=1pt,circle,draw] at (0.5,0) (p2) {};
     \node[inner sep=1pt,circle,draw] at (0,1.5) (p3) {};
     \node[inner sep=1pt,circle,draw,fill] at (0,1) (p4) {};
     \draw (v1)--(p3);
     \draw (v6)--(p3);
     \draw (p4)--(p3);
     \draw (p4)--(p1);
     \draw (p4)--(p2);
     \draw (v5)--(p1); 
     \draw (v4)--(p1); 
     \draw (v2)--(p2); 
     \draw (v3)--(p2);
\end{tikzpicture}.$ 
We will show that $\ch(T) = [W(T)]$ using skein relations, by converting each term in the standard monomial to the web basis.  
By the equations in Figure \ref{fig:equations of web invariants}, we can express $P_{135}P_{246}$ in the web basis as 
\begin{align*}
P_{135}P_{246} = P_{123}P_{456} + [W'] + [W_T] + P_{156}P_{234} + P_{126}P_{345},
\end{align*}
where $W' = P_{124}P_{356} -P_{123}P_{456}$ is the web from the Example~\ref{eg:124356web}. We likewise express $P_{125}P_{346} = [W']+P_{126}P_{345}$, express $P_{134}P_{256} = [W']+P_{156}P_{234}$, and we note that $P_{123}P_{456}$ is already a web. 
Then 
\begin{align*}
\ch(T) = &P_{135}P_{246}-P_{125}P_{346}-P_{134}P_{256}+P_{124}P_{356}-2P_{123}P_{456} \\
= & (P_{123}P_{456} + [W'] + [W_T] + P_{156}P_{234} + P_{126}P_{345})-([W']+P_{126}P_{345}) \\
 &-([W']+P_{156}P_{124})+([W']+P_{123}P_{456})-2P_{123}P_{456} \\
= &[W_T].
\end{align*}

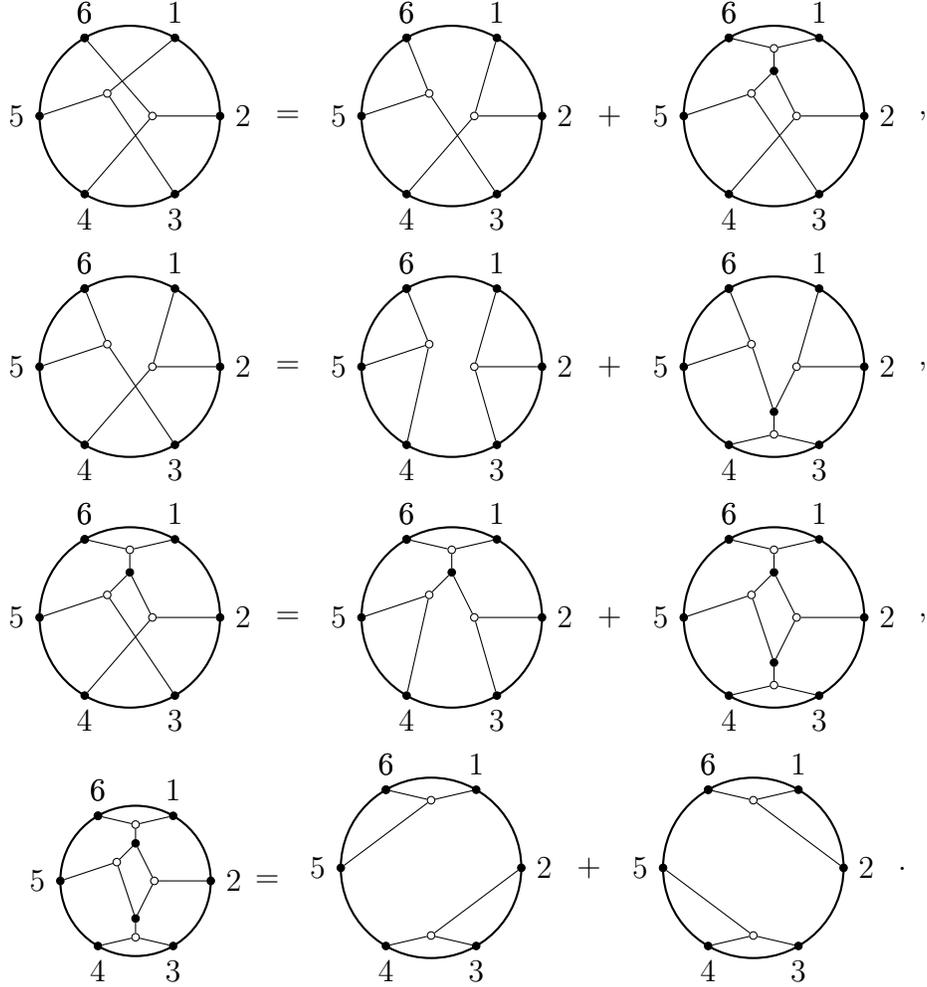
\begin{figure}[ht]
\begin{tikzpicture}[scale=0.6]
\draw[thick] (0,0) circle (2 cm);
  \newdimen\R
   \R=2cm
   \foreach \x/\l/\p/\q in
     {120- 60/{1}/above/v1,
      120-120/{2}/right/v2,
      120-180/{3}/below/v3,
      120-240/{4}/below/v4,
      120-300/{5}/left/v5,
      120-360/{6}/above/v6
     }
     \node[inner sep=1pt,circle,draw,fill,label={\p:\l}] at (\x:\R) (\q) {};
     \node[inner sep=1pt,circle,draw,fill,label={above:6}] at (120-360:\R) (v6) {};
     \node[inner sep=1pt,circle,draw] at (-0.5,0.5) (p1) {};   
     \node[inner sep=1pt,circle,draw] at (0.5,0) (p2) {};
     \draw (v1)--(p1);
     \draw (v3)--(p1);
     \draw (v5)--(p1);
     \draw (v2)--(p2);
     \draw (v4)--(p2);
     \draw (v6)--(p2);
     \node at (3.5,0) (p1) {$=$};   
\end{tikzpicture} 
\begin{tikzpicture}[scale=0.6]
\draw[thick] (0,0) circle (2 cm);
  \newdimen\R
   \R=2cm
   \foreach \x/\l/\p/\q in
     {120- 60/{1}/above/v1,
      120-120/{2}/right/v2,
      120-180/{3}/below/v3,
      120-240/{4}/below/v4,
      120-300/{5}/left/v5,
      120-360/{6}/above/v6
     }
     \node[inner sep=1pt,circle,draw,fill,label={\p:\l}] at (\x:\R) (\q) {};
     \node[inner sep=1pt,circle,draw,fill,label={above:6}] at (120-360:\R) (v6) {};
     \node[inner sep=1pt,circle,draw] at (-0.5,0.5) (p1) {};   
     \node[inner sep=1pt,circle,draw] at (0.5,0) (p2) {};
     \draw (v1)--(p2);
     \draw (v3)--(p1);
     \draw (v5)--(p1);
     \draw (v2)--(p2);
     \draw (v4)--(p2);
     \draw (v6)--(p1);
     \node at (3.5,0) (p1) {$+$};   
\end{tikzpicture} 
\begin{tikzpicture}[scale=0.6]
\draw[thick] (0,0) circle (2 cm);
  \newdimen\R
   \R=2cm
   \foreach \x/\l/\p/\q in
     {120- 60/{1}/above/v1,
      120-120/{2}/right/v2,
      120-180/{3}/below/v3,
      120-240/{4}/below/v4,
      120-300/{5}/left/v5,
      120-360/{6}/above/v6
     }
     \node[inner sep=1pt,circle,draw,fill,label={\p:\l}] at (\x:\R) (\q) {};
     \node[inner sep=1pt,circle,draw,fill,label={above:6}] at (120-360:\R) (v6) {};
     \node[inner sep=1pt,circle,draw] at (-0.5,0.5) (p1) {};   
     \node[inner sep=1pt,circle,draw] at (0.5,0) (p2) {};
     \node[inner sep=1pt,circle,draw] at (0,1.5) (p3) {};
     \node[inner sep=1pt,circle,draw,fill] at (0,1) (p4) {};
     \draw (v1)--(p3);
     \draw (v6)--(p3);
     \draw (p4)--(p3);
     \draw (p4)--(p1);
     \draw (p4)--(p2);
     \draw (v5)--(p1); 
     \draw (v3)--(p1); 
     \draw (v2)--(p2); 
     \draw (v4)--(p2);
     \node at (3.3,0) (p1) {$,$}; 
\end{tikzpicture} 

\begin{tikzpicture}[scale=0.6]
\draw[thick] (0,0) circle (2 cm);
  \newdimen\R
   \R=2cm
   \foreach \x/\l/\p/\q in
     {120- 60/{1}/above/v1,
      120-120/{2}/right/v2,
      120-180/{3}/below/v3,
      120-240/{4}/below/v4,
      120-300/{5}/left/v5,
      120-360/{6}/above/v6
     }
     \node[inner sep=1pt,circle,draw,fill,label={\p:\l}] at (\x:\R) (\q) {};
     \node[inner sep=1pt,circle,draw,fill,label={above:6}] at (120-360:\R) (v6) {};
     \node[inner sep=1pt,circle,draw] at (-0.5,0.5) (p1) {};   
     \node[inner sep=1pt,circle,draw] at (0.5,0) (p2) {};
     \draw (v1)--(p2);
     \draw (v3)--(p1);
     \draw (v5)--(p1);
     \draw (v2)--(p2);
     \draw (v4)--(p2);
     \draw (v6)--(p1);
     \node at (3.5,0) (p1) {$=$};   
\end{tikzpicture} 
\begin{tikzpicture}[scale=0.6]
\draw[thick] (0,0) circle (2 cm);
  \newdimen\R
   \R=2cm
   \foreach \x/\l/\p/\q in
     {120- 60/{1}/above/v1,
      120-120/{2}/right/v2,
      120-180/{3}/below/v3,
      120-240/{4}/below/v4,
      120-300/{5}/left/v5,
      120-360/{6}/above/v6
     }
     \node[inner sep=1pt,circle,draw,fill,label={\p:\l}] at (\x:\R) (\q) {};
     \node[inner sep=1pt,circle,draw,fill,label={above:6}] at (120-360:\R) (v6) {};
     \node[inner sep=1pt,circle,draw] at (-0.5,0.5) (p1) {};   
     \node[inner sep=1pt,circle,draw] at (0.5,0) (p2) {};
     \draw (v1)--(p2);
     \draw (v3)--(p2);
     \draw (v5)--(p1);
     \draw (v2)--(p2);
     \draw (v4)--(p1);
     \draw (v6)--(p1);
     \node at (3.5,0) (p1) {$+$};   
\end{tikzpicture} 
\begin{tikzpicture}[scale=0.6]
\draw[thick] (0,0) circle (2 cm);
  \newdimen\R
   \R=2cm
   \foreach \x/\l/\p/\q in
     {120- 60/{1}/above/v1,
      120-120/{2}/right/v2,
      120-180/{3}/below/v3,
      120-240/{4}/below/v4,
      120-300/{5}/left/v5,
      120-360/{6}/above/v6
     }
     \node[inner sep=1pt,circle,draw,fill,label={\p:\l}] at (\x:\R) (\q) {};
     \node[inner sep=1pt,circle,draw,fill,label={above:6}] at (120-360:\R) (v6) {};
     \node[inner sep=1pt,circle,draw] at (-0.5,0.5) (p1) {};   
     \node[inner sep=1pt,circle,draw] at (0.5,0) (p2) {};
     \node[inner sep=1pt,circle,draw,fill] at (0,-1) (p5) {};
     \node[inner sep=1pt,circle,draw] at (0,-1.5) (p6) {};
     \draw (v1)--(p2);
     \draw (v3)--(p6);
     \draw (v5)--(p1);
     \draw (v2)--(p2);
     \draw (p1)--(p5);
     \draw (p2)--(p5);
     \draw (p6)--(p5);
     \draw (v4)--(p6);
     \draw (v6)--(p1);
     \node at (3.3,0) (p1) {$,$}; 
\end{tikzpicture}


\begin{tikzpicture}[scale=0.6]
\draw[thick] (0,0) circle (2 cm);
  \newdimen\R
   \R=2cm
   \foreach \x/\l/\p/\q in
     {120- 60/{1}/above/v1,
      120-120/{2}/right/v2,
      120-180/{3}/below/v3,
      120-240/{4}/below/v4,
      120-300/{5}/left/v5,
      120-360/{6}/above/v6
     }
     \node[inner sep=1pt,circle,draw,fill,label={\p:\l}] at (\x:\R) (\q) {};
     \node[inner sep=1pt,circle,draw,fill,label={above:6}] at (120-360:\R) (v6) {};
     \node[inner sep=1pt,circle,draw] at (-0.5,0.5) (p1) {};   
     \node[inner sep=1pt,circle,draw] at (0.5,0) (p2) {};
     \node[inner sep=1pt,circle,draw] at (0,1.5) (p3) {};
     \node[inner sep=1pt,circle,draw,fill] at (0,1) (p4) {};
     \draw (v1)--(p3);
     \draw (v6)--(p3);
     \draw (p4)--(p3);
     \draw (p4)--(p1);
     \draw (p4)--(p2);
     \draw (v5)--(p1); 
     \draw (v3)--(p1); 
     \draw (v2)--(p2); 
     \draw (v4)--(p2);
     \node at (3.5,0) (p1) {$=$};   
\end{tikzpicture} 
\begin{tikzpicture}[scale=0.6]
\draw[thick] (0,0) circle (2 cm);
  \newdimen\R
   \R=2cm
   \foreach \x/\l/\p/\q in
     {120- 60/{1}/above/v1,
      120-120/{2}/right/v2,
      120-180/{3}/below/v3,
      120-240/{4}/below/v4,
      120-300/{5}/left/v5,
      120-360/{6}/above/v6
     }
     \node[inner sep=1pt,circle,draw,fill,label={\p:\l}] at (\x:\R) (\q) {};
     \node[inner sep=1pt,circle,draw,fill,label={above:6}] at (120-360:\R) (v6) {};
     \node[inner sep=1pt,circle,draw] at (-0.5,0.5) (p1) {};   
     \node[inner sep=1pt,circle,draw] at (0.5,0) (p2) {};
     \node[inner sep=1pt,circle,draw] at (0,1.5) (p3) {};
     \node[inner sep=1pt,circle,draw,fill] at (0,1) (p4) {};
     \draw (v1)--(p3);
     \draw (v6)--(p3);
     \draw (p4)--(p3);
     \draw (p4)--(p1);
     \draw (p4)--(p2);
     \draw (v5)--(p1); 
     \draw (v4)--(p1); 
     \draw (v2)--(p2); 
     \draw (v3)--(p2);
     \node at (3.5,0) (p1) {$+$};   
\end{tikzpicture} 
\begin{tikzpicture}[scale=0.6]
\draw[thick] (0,0) circle (2 cm);
  \newdimen\R
   \R=2cm
   \foreach \x/\l/\p/\q in
     {120- 60/{1}/above/v1,
      120-120/{2}/right/v2,
      120-180/{3}/below/v3,
      120-240/{4}/below/v4,
      120-300/{5}/left/v5,
      120-360/{6}/above/v6
     }
     \node[inner sep=1pt,circle,draw,fill,label={\p:\l}] at (\x:\R) (\q) {};
     \node[inner sep=1pt,circle,draw,fill,label={above:6}] at (120-360:\R) (v6) {};
     \node[inner sep=1pt,circle,draw] at (-0.5,0.5) (p1) {};   
     \node[inner sep=1pt,circle,draw] at (0.5,0) (p2) {};
     \node[inner sep=1pt,circle,draw] at (0,1.5) (p3) {};
     \node[inner sep=1pt,circle,draw,fill] at (0,1) (p4) {};
     \node[inner sep=1pt,circle,draw,fill] at (0,-1) (p5) {};
     \node[inner sep=1pt,circle,draw] at (0,-1.5) (p6) {};
     \draw (v1)--(p3);
     \draw (v6)--(p3);
     \draw (p4)--(p3);
     \draw (p4)--(p1);
     \draw (p4)--(p2);
     \draw (v5)--(p1); 
     \draw (v3)--(p6); 
     \draw (v2)--(p2); 
     \draw (v4)--(p6); 
     \draw (p1)--(p5);
     \draw (p2)--(p5);
     \draw (p5)--(p6);
     \node at (3.3,0) (p1) {$,$}; 
\end{tikzpicture}


\begin{tikzpicture}[scale=0.5]
\draw[thick] (0,0) circle (2 cm);
  \newdimen\R
   \R=2cm
   \foreach \x/\l/\p/\q in
     {120- 60/{1}/above/v1,
      120-120/{2}/right/v2,
      120-180/{3}/below/v3,
      120-240/{4}/below/v4,
      120-300/{5}/left/v5,
      120-360/{6}/above/v6
     }
     \node[inner sep=1pt,circle,draw,fill,label={\p:\l}] at (\x:\R) (\q) {};
     \node[inner sep=1pt,circle,draw,fill,label={above:6}] at (120-360:\R) (v6) {};
     \node[inner sep=1pt,circle,draw] at (-0.5,0.5) (p1) {};   
     \node[inner sep=1pt,circle,draw] at (0.5,0) (p2) {};
     \node[inner sep=1pt,circle,draw] at (0,1.5) (p3) {};
     \node[inner sep=1pt,circle,draw,fill] at (0,1) (p4) {};
     \node[inner sep=1pt,circle,draw,fill] at (0,-1) (p5) {};
     \node[inner sep=1pt,circle,draw] at (0,-1.5) (p6) {};
     \draw (v1)--(p3);
     \draw (v6)--(p3);
     \draw (p4)--(p3);
     \draw (p4)--(p1);
     \draw (p4)--(p2);
     \draw (v5)--(p1); 
     \draw (v3)--(p6); 
     \draw (v2)--(p2); 
     \draw (v4)--(p6); 
     \draw (p1)--(p5);
     \draw (p2)--(p5);
     \draw (p5)--(p6);
     \node at (3.5,0) (p1) {$=$};   
\end{tikzpicture} 
\begin{tikzpicture}[scale=0.6]
\draw[thick] (0,0) circle (2 cm);
  \newdimen\R
   \R=2cm
   \foreach \x/\l/\p/\q in
     {120- 60/{1}/above/v1,
      120-120/{2}/right/v2,
      120-180/{3}/below/v3,
      120-240/{4}/below/v4,
      120-300/{5}/left/v5,
      120-360/{6}/above/v6
     }
     \node[inner sep=1pt,circle,draw,fill,label={\p:\l}] at (\x:\R) (\q) {};
     \node[inner sep=1pt,circle,draw,fill,label={above:6}] at (120-360:\R) (v6) {};
     \node[inner sep=1pt,circle,draw] at (0,1.5) (p3) {};
     \node[inner sep=1pt,circle,draw] at (0,-1.5) (p6) {};
     \draw (v1)--(p3);
     \draw (v6)--(p3);
     \draw (v5)--(p3);
     \draw (v2)--(p6);
     \draw (v3)--(p6);
     \draw (v4)--(p6);
     \node at (3.5,0) (p1) {$+$};   
\end{tikzpicture} 
\begin{tikzpicture}[scale=0.6]
\draw[thick] (0,0) circle (2 cm);
  \newdimen\R
   \R=2cm
   \foreach \x/\l/\p/\q in
     {120- 60/{1}/above/v1,
      120-120/{2}/right/v2,
      120-180/{3}/below/v3,
      120-240/{4}/below/v4,
      120-300/{5}/left/v5,
      120-360/{6}/above/v6
     }
     \node[inner sep=1pt,circle,draw,fill,label={\p:\l}] at (\x:\R) (\q) {};
     \node[inner sep=1pt,circle,draw,fill,label={above:6}] at (120-360:\R) (v6) {};
     \node[inner sep=1pt,circle,draw] at (0,1.5) (p3) {};
     \node[inner sep=1pt,circle,draw] at (0,-1.5) (p6) {};
     \draw (v1)--(p3);
     \draw (v6)--(p3);
     \draw (v2)--(p3);
     \draw (v5)--(p6);
     \draw (v3)--(p6);
     \draw (v4)--(p6);   
     \node at (3.3,0) (p1) {$.$}; 
\end{tikzpicture} 
\caption{Converting the standard monomial expression to the web basis.}
\label{fig:equations of web invariants}
\end{figure}
\end{example}

\section{\texorpdfstring{$g$-vectors, dominant monomials, and tableaux}{g-vectors, dominant monomials, and tableaux}} \label{sec:g vectors and dominant monomials and tableaux}
We explain how results of Hernandez-Leclerc allow one to compute $g$-vectors for $\Gr(n,m)$ using the monoid ${\rm SSYT}(n,[m])$. 

Consider a cluster algebra whose initial cluster is denoted ${\bf x^{(0)}} = (x_1^{(0)},\dots,x_N^{(0)})$. For a vector $g \in \ZZ^N$ let $({\bf x^{(0)}})^g$ denote the Laurent monomial in these variables with exponent~$g$. Recall the Laurent monomial $\hat{y}_i = \prod_{j=1}^N (x_j^{(0)})^{\#(\text{ arrows } j \to i)- \#(\text{ arrows } i \to j)}$, defined for each mutable index $i \in [N]$, where $\#$ of arrows is computed with respect to the initial quiver. Then by definition, the $g$-{\sl vector} of a cluster monomial $x$ with respect to ${\bf x^{(0)}}$ is the exponent $g = g(x)$ such that $x = ({\bf x^{(0)}})^g F$, where $F$ is the {\sl $F$-polynomial} evaluated in the $\hat{y}_i$'s.

For our purposes, consider the initial cluster for $K_0(\mathcal{C}_\ell)$ indicated in Figure \ref{fig:initial cluster for Uqslhat5}. The initial cluster variables are of the form $[L(X_{i,j+1}^{(i-2j-2)})]$ where $i \in [n-1]$ and $j \in [0, \ell]$. Let ${\bf X} \subset \mathcal{P}^+$ be the dominant monomials $X_{i,j+1}^{(i-2j-2)}$ parameterizing the initial cluster variables. Let $N = |{\bf X}| = (n-1)(\ell+1)$. For $g \in \ZZ^N$, we let ${\bf X}^g \in \mathcal{P}$ denote the Laurent monomial in these variables with exponent $g$. 

By \eqref{eq:KR} the change of variables from the fundamental monomials $\{Y_{i,i-2j-2}\}_{i,j}$ to the initial monomials ${\bf X}$ is triangular, and in particular invertible. Therefore any $M \in \mathcal{P}^+$ can be uniquely expressed as $M = {\bf X}^g \in \mathcal{P}$ for some $g \in \ZZ^N$. Changing variables in this way, we obtain an injective homomorphism of monoids $\mathcal{P}_{\ell, A_{n-1}}^+ \hookrightarrow \ZZ^N$ sending $M \mapsto g_M$. 

Explicitly, any $L(M)$ in $\mathcal{C}_{\ell}^{A_{n-1}}$ is of the form $L(M)$ with
\begin{align*}
M = (Y_{1,-1}^{a_{1,0}} Y_{1,-3}^{a_{1,1}} \cdots Y_{1,-2\ell-1}^{a_{1,\ell}}) (Y_{2,0}^{a_{2,0}} Y_{2,-2}^{a_{2,1}} \cdots Y_{2,-2\ell}^{a_{2,\ell}}) \cdots (Y_{n-1,n-3}^{a_{n-1,0}} Y_{n-1,n-5}^{a_{n-1,1}} \cdots Y_{n-1,n-2\ell-3}^{a_{n-1,\ell}}),
\end{align*}
for nonnegative $(a_{i,j})_{i \in [n-1], j \in [0,\ell]}$. Using \eqref{eq:KR}, we also have that 
\begin{align*}
M = \ & (Y_{1,-1})^{g_{1,0}} (Y_{1,-1}Y_{1,-3})^{g_{1,1}} \cdots (Y_{1,-1}Y_{1,-3}\cdots Y_{1,-2\ell-1})^{g_{1,\ell}} \cdots \times \\
& \times (Y_{n-1,n-3})^{g_{n-1,0}} (Y_{n-1, n-3} Y_{n-1, n-5})^{g_{n-1,1}} \cdots (Y_{n-1,n-3}Y_{n-1,n-5} \cdots Y_{n-1, n-2\ell-3})^{g_{n-1,\ell}}.
\end{align*} 
Comparing these, it follows that 
$\sum_{k=j}^{\ell} g_{i,k} = a_{i,j}$ for $i \in [n-1]$ and $j \in [0, \ell]$, and these equations determine $g$. 

The following result is due to \cite[Section 5.2.2]{HL16}, see \cite[Section 2.6]{DS} for further explanation.  
\begin{lemma}[{\cite{HL16}}]\label{lem:hom from monomials to g vectors}
If $[L(M)]$ is a cluster monomial, then $g_M$ is its $g$-vector with respect to the initial seed in Figure~\ref{fig:initial cluster for Uqslhat5}.
\end{lemma}

\begin{example}\label{eg:gvectors} Consider the dominant monomial $M = Y_{1,-3}Y_{1,-5}Y_{2,0}Y_{2,-2}$. Inverting \eqref{eq:KR} we have $M = \frac{X_{1,3}^{(-5)}X_{2,2}^{(-2)}}{X_{1,1}^{(-1)}}$. Thus the nonzero coordinates of the $g_M$ are $+1$ in the entries corresponding to
$[L(X_{1,3}^{(-5)})]$ and $[L(X_{2,2}^{(-2)})]$ and $-1$ in the entry corresponding to $[L(X_{1,1}^{(-1)})]$.
\end{example}

Let $T_i^{(0)}$ denote the tableau indexing the $i$th initial cluster variable. Applying $\widetilde{\Phi}$, Lemma~\ref{lem:hom from monomials to g vectors} says that the $g$-vector of $\ch(T) \in \CC[\Gr(n,m,\sim)]$ is the vector of exponents $g_1,\dots,g_N$ such that $T  = \cup_{i=0}^N (T_i^{(0)})^{g_i}$ inside $K_0({\rm SSYT}(n,m,\sim))$. Let us explain how this lifts to $\CC[\Gr(n,m)]$. First, recall that the $\ZZ^m$-degree of a cluster monomial $\ch(T)$ is the content of $T$. Second, recall that the $F$-polynomial has constant term one. So the $\ZZ^m$-degree of the monomial ${\bf x^{(0)}}^g$ should match the content of $T$ (where $x_i^{(0)} = \ch(T_i^{(0)})$). But if $T$ equals $\cup_{i=0}^N (T_i^{(0)})^{g_i}$ up to trivial tableaux, and if the two have the same content, it follows that in fact $T = \cup_{i=0}^N (T_i^{(0)})^{g_i}$. 

We summarize this discussion and give the corresponding statement for $c$-vectors. 
\begin{corollary}\label{corollary:correspondence between g vectors and real tableaux}
Let $T^{(0)}_1,\dots,T^{(0)}_{N}$ be the tableaux labeling the initial cluster variables for $\CC[\Gr(n,m)]$ (cf.~Figure~\ref{fig:initial cluster for a quotient of Gr(5,10)}). Let $\ch(T) \in \CC[\Gr(n,m,\sim)]$ be a cluster monomial. Then the $g$-vector of $\ch(T)$ with respect to this cluster is the vector $(g_1,\dots,g_N)$ of exponents when $T$ is written as  Laurent monomial in the initial tableaux; i.e. $T = \cup_{i=1}^N (T_i^{(0)})^{g_i}$.  

The $c$-vectors with respect to this initial cluster are as follows. For any ``distant'' cluster $\{\ch(T_1),\dots,\ch(T_N)\}$ for $\bbc[\Gr(n,m,\sim)]$, express each {\sl initial} cluster variable as a $\cup$-Laurent monomial in the distant cluster, i.e. $T^{(0)}_i = \cup_{j=1}^N (T_j)^{c'_{ij}}$. Then $(c'_{1j},\dots,c'_{Nj})$ is the $j$th $c$-vector of this cluster with respect to the initial seed. 
\end{corollary}

The statement about $c$-vectors follows from the tropical duality $C = (G^{-1})^{t}$ between $c$-vectors and $g$-vectors \cite[Theorem 1.2]{NZ}. Here, 
$C$ and $G$ denote the matrices whose columns are the $c$- and $g$-vectors of the distant cluster with respect to our initial seed, and $t$ denotes transposition of matrices. 

\begin{example}
The tableau $T = \begin{ytableau}
1 & 2  \\
3 & 4 \\
5 & 6 
\end{ytableau}$ has $\Psi(T) = M$ where $M$ is the monomial from Example~\ref{eg:gvectors}. This tableau has the following expression as a Laurent monomial in the tableaux for the initial cluster:  
$$\begin{ytableau}
1 & 2  \\
3 & 4 \\
5 & 6 
\end{ytableau} = \begin{ytableau}
1   \\
2  \\
6  
\end{ytableau} \cup \begin{ytableau}
1   \\
4  \\
5  
\end{ytableau} \cup \begin{ytableau}
2   \\
3  \\
4  
\end{ytableau}
\cup 
\begin{ytableau}
1   \\
2  \\
4  
\end{ytableau}^{-1}
.
$$
This matches the previous computation, noting that the third factor in the numerator is trivial. 
\end{example}

\begin{remark}
The method for computing Corollary~\ref{corollary:correspondence between g vectors and real tableaux} is a special property of the initial seed $\{T^{(0)}_1,\dots,T^{(0)}_N\}$ in Figure~\ref{fig:initial cluster for a quotient of Gr(5,10)}. If we compute a $g$-vector with respect to a different initial seed $\{T_1,\dots,T_N\}$, then $\cup T_i^{g_i}$ is {\sl not} the tableau corresponding to this $g$-vector. Indeed $\cup T_i^{g_i}$ is often a nontrivial fraction of tableaux (i.e., an element of $K_0({\rm SSYT}(n,[m]))$ rather than of ${\rm SSYT}(n,[m])$). This is to be expected, because the initial-seed recursion for $g$-vectors \cite[Proposition 4.2(v)]{NZ} is piecewise linear, rather than linear.
\end{remark}

\section{Reality, primeness, and compatibility of cluster variables} \label{sec: real modules}
We give examples of nonreal modules and tableaux, and compare $\{\ch(T)\}$ with the basis of web invariants. We end by discussing primeness of modules and tableaux, and compatibility of cluster variables. 
\subsection{Smallest non-real examples}
Recall that a tableau $T$ is called real when $L(M_T)$ is real. By Lemma \ref{lem:qcharacter condition of real modules and prime modules} and Theorems \ref{thm:Hernandez-Leclerc quantum affine algebras and Grassmannians}, \ref{thm: parametrization of simple modules by tableaux}, $T \in {\rm SSYT}(n,[m],\sim)$ is real if and only if $\ch(T) \ch(T) = \ch(T \cup T)$ and a $U_q(\widehat{\mathfrak{g}})$-module $L(M)$ is real if and only if $\chi_q(M^2) = \chi_q(M) \chi_q(M)$. We start by cataloguing the smallest examples of nonreal tableaux. 
\begin{example}\label{eg:catalogue}
Consider $\Gr(n,m)$ where $2 \leq n \leq \frac{m}{2}$. (One can focus on these since $\Gr(n,m) \cong \Gr(m-n,m)$ .) It is known that such a Grassmannian has finite cluster type exactly when $n=2$ or $(n,m) \in \{(3,6),(3,7),(3,8)\}$. In these cases, every $\ch(T)$ is a cluster monomial and every simple module is real. 

The smallest Grassmannians which are not of finite type are $\Gr(3,9)$ and $\Gr(4,8)$, corresponding to $\mathcal{C}_{5}^{A_{2}}$ and 
$\mathcal{C}_{3}^{A_{3}}$ respectively. Consider the following tableaux $T_1,T_2,T_3,T_4,T_5$: 
\begin{equation}
\begin{ytableau}
1 & 3 & 4 \\
2 & 6 & 7 \\
5 & 8 & 9
\end{ytableau}, \hspace{.35cm}
\begin{ytableau}
1 & 2 & 5 \\
3 & 4 & 8 \\
6 & 7 & 9
\end{ytableau} , \hspace{.35cm}
\begin{ytableau}
1 & 2 & 3 \\
4 & 5 & 6 \\
7 & 8 & 9
\end{ytableau} , \hspace{.35cm}
\begin{ytableau}
1 & 3  \\
2 & 5  \\
4 & 7  \\
6 & 8
\end{ytableau} , \hspace{.35cm}
\begin{ytableau}
1 & 2  \\
3 & 4  \\
5 & 6  \\
7 & 8
\end{ytableau}.
\end{equation}
The webs corresponding to these five tableaux are the rotations of 
$$
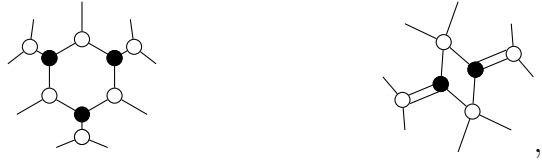
\begin{tikzpicture}
\draw (22.5:.25cm)--(112.5:.5cm)--(202.5:.25cm)--(292.5:.5cm)--(22.5:.25cm);
\draw (45:1cm)--(22.5:.8cm)--(0:1cm);
\draw (225:1cm)--(202.5:.8cm)--(180:1cm);
\draw (.6399cm,.356cm)--(.13cm,.145cm);
\draw (.689cm,.256cm)--(.18cm,.045cm);
\draw (-.6399cm,-.356cm)--(-.13cm,-.145cm);
\draw (-.689cm,-.256cm)--(-.18cm,-.045cm);
\draw (90:1cm)--(112.5:.5cm)--(135:1cm);
\draw (270:1cm)--(292.5:.5cm)--(315:1cm);
\filldraw[black] (22.5:.25cm) circle (0.1cm);
\draw (22.5:.25cm) circle (0.1cm);
\filldraw[black] (202.5:.25cm) circle (0.1cm);
\draw (202.5:.25cm) circle (0.1cm);
\filldraw[white] (112.5:.5cm) circle (0.1cm);
\draw (112.5:.5cm) circle (0.1cm);
\filldraw[white] (292.5:.5cm) circle (0.1cm);
\draw (292.5:.5cm) circle (0.1cm);
\filldraw[white] (22.5:.8cm) circle (0.1cm);
\draw (22.5:.8cm) circle (0.1cm);
\filldraw[white] (202.5:.8cm) circle (0.1cm);
\draw (202.5:.8cm) circle (0.1cm);
\begin{scope}[xshift = -5cm]
\draw (130:1cm)--(150:.8cm)--(170:1cm);
\draw (250:1cm)--(270:.8cm)--(290:1cm);
\draw (10:1cm)--(30:.8cm)--(50:1cm);
\draw (30:.5cm)--(90:.5cm)--(150:.5cm)--(210:.5cm)--(270:.5cm)--(330:.5cm)--(30:.5cm);
\draw (90:.5cm)--(90:1cm);
\draw (330:.5cm)--(330:1cm);
\draw (210:.5cm)--(210:1cm);
\draw (150:.5cm)--(150:.8cm);
\draw (30:.5cm)--(30:.8cm);
\draw (270:.5cm)--(270:.8cm);
\filldraw[white] (90:.5cm) circle (0.1cm);
\draw (90:.5cm) circle (0.1cm);
\filldraw[white] (210:.5cm) circle (0.1cm);
\draw (210:.5cm) circle (0.1cm);
\filldraw[white] (330:.5cm) circle (0.1cm);
\draw (330:.5cm) circle (0.1cm);
\filldraw[black] (30:.5cm) circle (0.1cm);
\draw (30:.5cm) circle (0.1cm);
\filldraw[black] (150:.5cm) circle (0.1cm);
\draw (150:.5cm) circle (0.1cm);
\filldraw[black] (270:.5cm) circle (0.1cm);
\draw (270:.8cm) circle (0.1cm);
\filldraw[white] (30:.8cm) circle (0.1cm);
\draw (30:.8cm) circle (0.1cm);
\filldraw[white] (150:.8cm) circle (0.1cm);
\draw (150:.8cm) circle (0.1cm);
\filldraw[white] (270:.8cm) circle (0.1cm);
\draw (270:.8cm) circle (0.1cm);
\end{scope}
\end{tikzpicture},
$$
the first of which is an $\mathfrak{sl}_3$ web and the second of which is an $\mathfrak{sl}_4$ web. The three cyclic shifts of the first web correspond to $W(T_1),W(T_2),W(T_3)$. There is no labeling of webs by tableaux when $n=4$, so we define $W(T_4),W(T_5)$ to be the two cyclic shifts of the second web. (This is the ``correct'' thing to do e.g. by \cite[Appendix]{FLL}). The correspondence between tableaux and cyclic shifts is fixed by the forks mnemonic above. Alternatively, we give the following explicit formulas: 
\begin{align*}
[W(T_1)] &= P_{145}P_{278}P_{369}-P_{245}P_{178}P_{369}-P_{123}P_{456}P_{789}-P_{129}P_{345}P_{678} \\
[W(T_4)] &= P_{1247}P_{3568}-P_{1237}P_{4568}-P_{1248}P_{3567}+P_{1238}P_{4567},
\end{align*}
with the other formulas determined by cyclically shifting indices. Using \eqref{eq:liftofchT} on a computer, 
we checked that $\ch(T_i) = W(T_i)$ for $i=1,\dots,5$. In particular, each $\ch(T_i)$ is in the cluster algebra. Moreover, the elements $\ch(T_1),\ch(T_2),\ch(T_3)$ are dual canonical by \cite[Theorem 1]{KK}. 

Using \eqref{eq:liftofchT} on a computer,  we checked that the tableau $T_1,T_4$ are not real, i.e. that $\ch(T_i \cup T_i) \neq \ch(T_i)^2$. Presumably, the cyclic shifts are also nonreal. It is known that these five are the ``smallest'' nonreal tableaux, meaning that every other tableaux in ${\rm SSYT}(3,[9])$ (resp. ${\rm SSYT}(4,[8])$) with at most three (resp. two) columns is a cluster monomial. We explicitly calculate $\ch(T_4 \cup T_4)$ and $\ch(T_1 \cup T_1)$ in Examples~\ref{example:non-real tableaux} and \ref{eg:T4}. 
\end{example}

\begin{remark}\label{rmk:notsurjective}
Consider the tableau $T_1$ from Example~\ref{eg:catalogue}. It has $\ch(T_1) = W(T_1)$ the web from above. We will show directly that $\ch(T_1) \neq {\rm MS}^*(f)$ for any dual canonical basis element $f \in \CC[\GL_9]$, where ${\rm MS} \colon \Gr(n,m) \to \GL_m$ is in \eqref{eq:MS}. By degree considerations, such an $f$ would correspond to a $3 \times 3$ generalized submatrix. Moreover, note that $W(T_1)$ has $\ZZ^9$-degree $(1,\dots,1)$. Each Pl\"ucker coordinate in the matrix ${\rm MS}(x)$ has gap weight $\leq 1$. The only way for a product of three such Pl\"ucker coordinates to have $\ZZ^9$-degree $(1,\dots,1)$, is if the product is a cyclic shift of either $P_{123}P_{456}P_{789}$ or $P_{124}P_{356}P_{789}$. The first option is a web, and the second option is a sum of two webs: $P_{124}P_{356}P_{789} = P_{123}P_{456}P_{789}+P_{789}\ch([1,2,4],[3,5,6])$ (cf.~Example~\ref{eg:124356web} for the web $\ch([1,2,4],[3,5,6])$). Clearly, $W(T_1)$ is not a linear combination of webs such as these, so that $W(T_1)$ is not the pullback of any $3 \times 3$ Kazhdan-Lusztig immanant. In order to ``see'' $W(T_1)$ as a Kazhdan-Lusztig immanant, it seems necessary to compute the $8\times 8$ immanant corresponding to the small gaps version of $T_1$, and then divide by appropriate frozen variables as in \eqref{eq:liftofchT}.  
\end{remark}

\subsection{Lapid-M\'{i}nguez's criterion}
 A representation $\pi$ in ${\rm Irr}$ is called {\sl square-irreducible} if $\pi \times \pi$ is irreducible, \cite{LM}. A multi-segment ${\bf m}=\Delta_1+\cdots +\Delta_k$ is called {\sl regular} if its left endpoints $b(\Delta_1), \ldots, b(\Delta_k)$ are distinct and if the same is true of its right endpoints $e(\Delta_1), \ldots, e(\Delta_k)$ \cite{LM}. 

A regular multi-segement ${\bf m} = \sum_{j=1}^k \Delta_j$, $k \ge 4$ with $e(\Delta_1)> \cdots > e(\Delta_k)$ is called of type $4231$ (resp. $3412$) if 
\begin{align*}
\Delta_{i+1} \prec \Delta_i, \ i=3, \ldots, k-1, \ \Delta_3 \prec \Delta_1, \ b(\Delta_k)<b(\Delta_2)<b(\Delta_{k-1})
\end{align*}
(resp. 
\begin{align*}
\Delta_{i+1} \prec \Delta_i, \ i=4, \ldots, k-1, \ \Delta_4 \prec \Delta_2, \ b(\Delta_3)<b(\Delta_k)<b(\Delta_1)<b(\Delta_l),
\end{align*}
where $l=2$ if $k=4$ and $l=k-1$ otherwise), cf. \cite[Definition 6.10]{LM}. 

Lapid and M\'{i}nguez classified square-irreducible representations for regular multi-segments: 

\begin{theorem}[{\cite{LM}}] \label{thm: Lapid-Minguez classification}
For a regular multi-segment ${\bf m}$, the representation ${\rm Z}({\bf m})$ is square-irreducible if and only if ${\bf m}$ does not admit a sub-multi-segment of type $4231$ or $3412$. 

For sufficiently large $n$, a $U_q(\widehat{\mathfrak{sl}_n})$-module is real if and only if ${\bf m}_M$ does not admit a sub-multi-segment of type $4231$ or $3412$. 
\end{theorem}

It would be quite interesting to generalize Theorem~\ref{thm: Lapid-Minguez classification} to nonregular multi-segments. 

We can rephrase the theorem in the language of small gaps tableaux. Regularity of the multi-segment ${\bf m}$ means that both sequences ${\bf i, j}$ from Definition~\ref{defn:usemicolonT} are without repetitions. Recall that the $a$th column of $T'$ is obtained from $[i_a,i_a+n]$ by removing an element $r_a \in (i_a,i_a+n)$ (explicitly $r_a = j_{w^{-1}(a)}$). Then the small gaps tableau $T'$ is real provided the sequence $r_1,\dots,r_k$ avoids both patterns $3412$ and $1324$.  

\begin{remark}
The ``minimal'' small gaps tableau $T'$ that contains the pattern $1324$ occurs in ${\rm SSYT}(4,[8])$ and has columns $[1,2,4,5],[2,3,4,6],[3,5,6,7],[4,5,7,8]$. The sequence $r_1,r_2,r_3,r_4$ is $3,5,4,6$, which is an instance of the pattern $1324$. There is an equivalence $T' \sim T_4$ where $T_4$ is the nonreal tableau from Example~\ref{eg:catalogue}. This small gaps tableau corresponds to the multisegment $[0,1]+[-2,0]+[-1,-1]+[-3,-2]$ (we have ordered the summands so that the right endpoints are weakly decreasing as in \eqref{eq:lammu}). The left endpoints form the pattern $4231$. Thus, the nonreality of $T_4$ follows from Theorem~\ref{thm: Lapid-Minguez classification}. The other nonreal tableaux in Example~\ref{eg:catalogue} are nonregular. 

If we translate this same multi-segment to a tableau using $n=3$, we get the tableau $T''$ with columns $[1,3,4],[2,3,5],[4,5,6],[4,6,7]$. This is not a small gaps tableau, so the reality criterion does not apply. In fact, $\ch(T'')$ must be a cluster monomial since $\CC[\Gr(3,7)]$ has finite cluster type.  
\end{remark}

\subsection{Zelevinsky duality} \label{sec:Zelevinsky dual}
Zelevinsky duality \cite{Zel} is an involution on $\mathcal{R}^G$ that preserves square-irreducibility. We will describe it via the {\sl M\oe{}glin-Waldspurger algorithm} \cites{MW,BR} for computing it on multi-segments. 

If $\Delta = [b,e]$ is a segment, we set $\Delta^- =[b,e-1]$, with the convention that $\Delta^-$ is empty if $b=e$. Given a multi-segment 
${\bf m}=\sum_{i=1}^k \Delta_i$, we define its {\sl Zelevinsky dual} multi-segment ${\bf m}^{\sharp}$ \cites{Zel, MW, LM} by induction on the degree of multi-segments. To begin, choose segments $\Delta_{i_0} = [b_{i_0},e_{i_0}],\dots,\Delta_{i_r} = [b_{i_r},e_{i_r}]$ in ${\bf m}$ defined as follows. First, $e_{i_0} = \max_{i \in [k]} e(\Delta_i)$ is the rightmost endpoint appearing in ${\bf m}$, and $e_{i_s} = e_{i_0}-s$ for $s=1,\dots,r$. Second, for $s=0,\dots,r$, $b_{i_s}$ is largest amongst left endpoints whose right endpoint is $e_{i_s}$. And third, $r$ is maximal, meaning that ${\bf m}$  contains no segment $[b,e_{i_0}-r]$ with $b < b_{i_r}$. The $\Delta_{i_0},\dots,\Delta_{i_r}$ with these properties are well defined as abstract segments, although since ${\bf m}$ can have repetitions there might be choices for ``which'' segment is called $\Delta_{i_s}$. We define $\Delta_i'$ as either $\Delta_i^-$ or $\Delta_i$ according to whether $i \in \{i_0,\dots,i_r\}$ or not. We define the multi-segment ${\bf m'} =\sum_{i = 0}^k \Delta_{i}'$, which is a multi-segment of strictly smaller degree. Finally, we can therefore define ${\bf m}^\sharp = [e_{i_0}-r,e_{i_0}]+({\bf m'})^\sharp$ by induction on degree. This a degree-preserving involution on multi-segments. 

This definition extends by linearity to an involution of $\mathcal{R}^G$. It is easy to see that $({\bf  m} + {\bf m})^\sharp = {\bf m}^\sharp + {\bf m}^\sharp$ for any multi-segment ${\bf m}$. We caution that $({\bf m}_1 + {\bf m}_2)^\sharp $ does not have to coincide with ${\bf m}_1^\sharp + {\bf m}_2^\sharp$ for distinct ${\bf m}_1, {\bf m}_2$.

\begin{proposition}[{\cite[Proposition 3.15]{LM}}] \label{thm:Zelevinsky dual preserve reality}
If ${\rm Z}({\bf m}) \in {\rm Irr}$ is square-irreducible, then ${\rm Z}({\bf m}^\sharp)$ is also square-irreducible.
\end{proposition}

For $M \in \mathcal{P}^+$, we denote $M^{\sharp}=M_{{\bf m}^{\sharp}}$ and call $M^{\sharp}$ the Zelevinsky dual of $M$.

\begin{example}\label{example:non-real tableaux}
Let 
$T_1 = \begin{ytableau}
1 & 3 & 4   \\
2 & 6 & 7 \\
5 & 8 & 9
\end{ytableau}$ be the nonreal tableau from Example~\ref{eg:catalogue}. It has gap weight 8. It corresponds to the dominant monomial
$Y_{1,-1} Y_{1,-3} Y_{2,-4} Y_{2,-6}^2 Y_{1,-9} Y_{2,-8} Y_{1,-11}$, and to the multi-segment 
$${\bf m} = [0, 0] + [-1, -1] + [-2, -1] + [-3, -2] + [-3, -2] + [-4, -3] + [-4, -4] + [-5, -5].$$ It is not so difficult to check nonreality 
of $T_1$ by showing that $\ch(T_1 \cup T_1)$ and $\ch(T_1)^2$ disagree in a particular standard monomial coefficient. On the other hand, directly computing $\ch(T_1 \cup T_1)$ is cumbersome: there are many permutations $u \in S_{16}$ such that $P_{u;T ' \cup T'} \neq 0$, thus many Kazhdan-Lusztig values to compute. One can simplify such a calculation using Zelevinsky duality. The first contribution to ${\bf m}^\sharp$ is the segment $[-3,0]$, with the segments 
$[0,0],[-1,-1],[-3,-2],[-4,-3]$ serving as $\Delta_{i_0},\dots,\Delta_{i_3}$ in the definition. The end result is  
\begin{align*}
 {\bf m}_{M^{\sharp}} = [-3, 0] + [-2, -1] + [-5, -2] + [-4, -3], \hspace{.5cm} \text{ and } M^{\sharp} = Y_{2,-4} Y_{4,-4} Y_{2,-8} Y_{4,-8}. \\
\end{align*}
The advantage is that ${\bf m}_{M^{\sharp}}$ has $k=4$, so we can test its nonreality in $S_8$. 
By applying the formula in Theorem \ref{thm:multi-segment character formula}, one has that
\begin{align} \label{eq: decomposition of tensor product 2}
\begin{split}
& [{\rm Z}({\bf m}_M^{\sharp})] [{\rm Z}({\bf m}_M^{\sharp})] = [{\rm Z}({\bf m}_M^{\sharp} + {\bf m}_M^{\sharp})] \\
&  \quad \quad + [{\rm Z}([-3, 0] + [-5, 0] + [-2, -1] + [-4, -1] + [-3, -2] + [-5, -2] + [-4, -3])].
\end{split}
\end{align}
(We are very grateful to Erez Lapid for his computer program to compute this decomposition.) Therefore
\begin{align*}
& [{\rm Z}({\bf m}_M) ] [{\rm Z}({\bf m}_M)] = [{\rm Z}({\bf m}_M + {\bf m}_M)] + [{\rm Z}({\bf m}')], \text{ where }\\
& {\bf m}' = [0, 0] + [-2, 0] + [-1, -1] + [-3, -1] + [-3, -1] + [-2, -2] \\
& \quad \quad + [-4, -2] +[-4, -2] + [-3, -3] + [-5, -3] + [-4, -4] + [-5, -5].
\end{align*}
Translating to $q$-characters, 
\begin{align*}
\chi_q(M)\chi_q(M) & = \chi_q(M^2) + \chi_q(Y_{1,-1} Y_{1,-3} Y_{1,-5} Y_{3,-3} Y_{1,-7} Y_{3,-5}^2 Y_{1,-9} Y_{3,-7}^2 Y_{1,-11} Y_{3,-9}).
\end{align*}
Translating to tableaux in the case $n=3$, 
\begin{align*}
\ch(T_1)^2 = \ch(T_1 \cup T_1) + \ch(\begin{ytableau}
1 & 1 & 3 & 4 & 6 & 7\\
2 & 2 & 4 & 5 & 7 & 8  \\
3 & 5 & 6 & 8 & 9 & 9 
\end{ytableau}) \in \CC[\Gr(3,9)]. 
\end{align*}
We let $F$ be the product of frozen variables appearing as the second term on the right hand side, and note that 
$F = P_{1,2,9} \in \CC[\Gr(3,9,\sim)]$. 
It is easy to see that $\ch(T_1)^2 = W(T_1)^2 = W(T_1 \cup T_1)$. In fact, {\sl every} power of the web $W(T_1)$ is a web given by the $k$-{\sl thickening} procedure \cite[Definition 10.8]{FP}. Thus $\ch(T_1 \cup T_1) = W(T_1)^2-F$ is {\sl not} a web, but rather a difference of two webs. This matches behavior of the dual canonical basis \cite[Theorem 4]{KK}. 
\end{example}

\begin{example}\label{eg:T4}
Let 
$T_4 = \begin{ytableau}
1 & 3    \\
2 & 5  \\
4 & 7  \\
6 & 8 
\end{ytableau}$ as in Example~\ref{eg:catalogue}. Using \eqref{eq:liftofchT} on a computer, we checked that 
$$\ch(T_4 \cup T_4) = \ch(T_4)^2 - \ch(\begin{ytableau}
1 & 1 & 3 & 5    \\
2 & 2 & 4 & 6  \\
3 & 5 & 7 & 7 \\
4 & 6 & 8 & 8 
\end{ytableau}) \in \CC[\Gr(4,8)].$$
The last term simplifies to $P_{1278} \in \CC[\Gr(4,8,\sim)]$. The corresponding dominant monomial is $M = Y_{2,-6} Y_{1,-3} Y_{3,-3} Y_{2,0}$. The simple module $L(M)$ is also shown to be a nonreal $U_q(\widehat{\mathfrak{sl}_4})$-module in \cite[Section 13]{HL10}. 
\end{example}

\subsection{Prime modules} \label{sec:prime modules}
By Lemma \ref{lem:qcharacter condition of real modules and prime modules} and Theorems \ref{thm:Hernandez-Leclerc quantum affine algebras and Grassmannians}, \ref{thm: parametrization of simple modules by tableaux}, a tableau $T \in {\rm SSYT}(n,[m],\sim)$ is prime if and only if there are no nontrivial tableaux $T', T''$ such that $\ch(T) = \ch(T')\ch(T'')$.

Clearly, if $\ch(T) = \ch(T')\ch(T'')$, then $T' \cup T'' = T$. For a given semistandard tableau $T$, there are finitely many pairs of semistandard tableau $T', T''$ such that $T' \cup T'' = T$. Therefore to check whether a semistandard Young tableau is prime or not, it suffices to exhaust over all such pairs, checking whether 
$\ch(T) = \ch(T')\ch(T'')$ using the formula in Theorem~\ref{cor:qcharacter formula tableaux}.  

Equivalently, one can check whether a simple $U_q(\widehat{\mathfrak{sl}_n})$-module $L(M)$ is prime by checking whether there are simple modules $L(M')$, $L(M'')$ such that $\chi_q(M) = \chi_q(M')\chi_q(M'')$.

\begin{remark}
Usually for a tableau $T$, the expression for $\ch(T)$ is complicated. When one knows that $\ch(T)$ happens to be a cluster monomial, one can use cluster algebras to check whether $T$ is prime. 
\end{remark}

\begin{example}\label{example: prime and non-prime modules}
By the exchange relation in Example~\ref{example:exchange relations}, the tableau $T = \begin{ytableau}
1 & 3 & 4  \\
2 & 5 & 6 \\
4 & 7 & 8 
\end{ytableau}$ labels a cluster variable $\ch(T)  = \CC[\Gr(3,8)]$.
By \cite[Theorem 1.2.1]{Qin}, it follows that the corresponding module $L(M_T) = L(Y_{1,-1} Y_{2,-4} Y_{1,-7} Y_{2,-6} Y_{1,-9})$ is prime.
\end{example}

The elements $\ch(\begin{ytableau}
1  \\
2 \\
8 
\end{ytableau})$ and  $\ch(\begin{ytableau}
3 & 4  \\
5 & 6 \\
7 & 8
\end{ytableau})$ are compatible cluster variables. Therefore 
\begin{align*}
\ch(\begin{ytableau}
1 & 3 & 4  \\
2 & 5 & 6 \\
7 & 8 & 8
\end{ytableau}) =
\ch(\begin{ytableau}
1  \\
2 \\
8 
\end{ytableau}) \ch(\begin{ytableau}
3 & 4  \\
5 & 6 \\
7 & 8
\end{ytableau}).
\end{align*}
is a nontrivial cluster monomial. 

By \cite[Theorem 1.2.1]{Qin}, the corresponding module 
\begin{align*}
L(M_T) = L(Y_{1,-1}  Y_{1,-3}  Y_{1,-5}  Y_{2,-4}  Y_{1,-7}^2  Y_{2,-6}  Y_{1,-9}^2)
\end{align*}
is not prime, and admits a factorization 
\begin{align*}
[L(M_T)] = [L(Y_{1,-1}  Y_{1,-3}  Y_{1,-5}  Y_{1,-7}  Y_{1,-9}) \otimes L(Y_{2,-4}  Y_{1,-7}  Y_{2,-6}  Y_{1,-9})].
\end{align*}

\subsection{Compatibility of cluster variables} \label{sec:compatibility of two cluster variables}
Two cluster variables in a cluster algebra are called {\sl compatible} if they are in a common cluster. We also call tableaux $T, T'$ corresponding to cluster monomials {\sl compatible} if  $\ch(T)$, $\ch(T')$ are. 

\begin{conjecture} \label{conj: condition of two tableaux are compatible}
Two cluster variables $\ch(T), \ch(T')$ in $\bbc[\Gr(n,m)]$ are compatible only if 
\begin{align*}
\ch(T)\ch(T')=\ch(T \cup T').
\end{align*}
\end{conjecture}

The ``if'' direction of Conjecture \ref{conj: condition of two tableaux are compatible} is clear. If $\ch(T)$ and $\ch(T')$ are compatible, then $\ch(T) \ch(T')$ is a cluster monomial. By Theorem \ref{thm: cluster monomials are tableaux}, the cluster monomial $\ch(T) \ch(T')$ is $\ch(T'')$ for some $T''$. And by 
Lemma \ref{lem:LMLMprime decomposition}, we conclude that $T'' = T \cup T'$. Therefore $\ch(T)\ch(T')=\ch(T \cup T')$.

The combinatorial criterion for compatibility of Pl\"ucker coordinates is known as {\sl weak separation}. It was conjectured by  
\cites{LZ,Sco} and proved by \cites{OPS,DKK}. As a first step towards Conjecture~\ref{conj: condition of two tableaux are compatible}, it would be interesting to verify that for single column tableaux, $\ch(T) \ch(T') = \ch(T \cup T')$ implies that $P_T$ and  $P_{T'}$ are weakly separated.

\begin{example}
The tableaux $
\ytableausetup{mathmode}
T=\begin{ytableau}
1    \\
2   \\
5
\end{ytableau}
$ and 
$T'=\ytableausetup{mathmode}
\begin{ytableau}
1    \\
3   \\
4
\end{ytableau}
$ are compatible. We have
\begin{align*}
\ch(T) \ch(T') = P_{125}P_{134}= P_{124} P_{135}- P_{123} P_{145}= \ch( \ytableausetup{mathmode}
\begin{ytableau}
1  & 1   \\
2  & 3 \\
4  & 5
\end{ytableau}) = \ch(T \cup T').
\end{align*}
The sets $\{1,2,5\}$ and $\{1,3,4\}$ are weakly separated.
\end{example}

\begin{example}
The tableaux $
T=\ytableausetup{mathmode}
\begin{ytableau}
1    \\
2   \\
4
\end{ytableau}
$ and 
$T'=\ytableausetup{mathmode}
\begin{ytableau}
3    \\
5   \\
6
\end{ytableau}
$ are not compatible. We have
\begin{align*}
\ch(T) \ch(T') = P_{124}P_{356},
\end{align*}
and 
\begin{align*}
\ch(T \cup T') = \ch( \ytableausetup{mathmode}
\begin{ytableau}
1  & 3   \\
2  & 5 \\
4  & 6
\end{ytableau})
= P_{124}P_{356}-P_{123}P_{456}.
\end{align*}
The sets $\{1,2,4\}$ and $\{3,5,6\}$ are not weakly separated.
\end{example}

\end{document}